\documentclass[10pt, a4paper]{amsart}

\usepackage{hyperref}
\usepackage{epsfig}
\usepackage{texmac}
\usepackage{enumitem}
\usepackage[all,cmtip]{xy}
\usepackage{tikz-cd}
\usepackage{graphicx}

\usepackage{cmap}
\usepackage[T2A, T1]{fontenc}

\DeclareSymbolFont{cyrillic}{T2A}{cmr}{m}{it}
\SetSymbolFont{cyrillic}{bold}{T2A}{cmr}{bx}{it}

\DeclareMathSymbol{\mE}{\mathord}{cyrillic}{221}
\DeclareMathSymbol{\mYu}{\mathord}{cyrillic}{222}
\DeclareMathSymbol{\mYa}{\mathord}{cyrillic}{223}

\tikzset{%
symbol/.style={
  draw=none,
  every to/.append style={
    edge node={node [sloped, allow upside down, auto=false]{$#1$}}
  },
},
}

\calsymbols{c}{A,O,Y,D,X,F,V,R,H,N,T,G}
\bbsymbols{}{G,F}
\bbsymbols{b}{R,H,E,T}
\fraksymbols{f}{p,m,g,k,P,N,M,Q,n,q,a,b,d,gl,so,su,F,A}
\bmsymbols{bm}{i}

\operators{PSL,nrd,Ad,vol,disc,free,Ann,ss,MaxSpec,split,ram,SO,SU,sign,AI,O,ind,Cl,Art,
gcd,cyc,tors,Isom,Aut,End,diag,iso,Sym,U,JL,lcm,nr,alg,new}
\DeclareMathOperator{\centre}{Z}

\DeclareMathOperator{\Mat}{M}
\DeclareMathOperator{\normaliser}{N}
\DeclareMathOperator{\norm}{Nm}
\DeclareMathOperator{\PS}{PS} 
\DeclareMathOperator{\DS}{DS} 
\DeclareMathOperator{\FD}{F}
\DeclareMathOperator{\Jac}{J}
\DeclareMathOperator{\rL}{L}
\DeclareMathOperator{\WRes}{Res}
\DeclareMathOperator{\tame}{t}
\DeclareMathOperator{\heckegroup}{\mathbf{HC}}

\newcommand{\Asterisk}{\mathop{\scalebox{1.5}{\raisebox{-0.2ex}{$\ast$}}}}%
\newcommand{\hecke}{\mathbf{T}}
\newcommand{\bfG}{\mathbf{G}}
\newcommand{\abstracthecke}{\bT}

\newcommand{\heckebig}{\mathbf{A}}
\newcommand{\adel}{\mathbb{A}}
\newcommand{\Hamil}{\mathbb{H}}

\newcommand{\lquo}{\backslash}
\newcommand{\lieg}{\fg}
\newcommand{\liegl}{\fgl}
\newcommand{\liek}{\fk}
\newcommand{\lieso}{\fso}
\newcommand{\liesu}{\fsu}
\newcommand{\vecti}{\underline{\bmi}}
\newcommand{\vectlambda}{\underline{\bm{\lambda}}}
\newcommand{\vectzero}{\underline{\bm{0}}}
\newcommand{\harmon}{\mathcal{H}}

\newcommand{\Fpbar}{\overline{\F}_p}
\newcommand{\Id}{\mathrm{Id}}
\newcommand{\pairing}[1]{\langle\cdot,\cdot\rangle_{#1}}

\newcommand{\dcr}{\cR}
\newcommand{\Cdual}{\widehat{C}}
\newcommand{\Ciso}{C_{\iso}}
\newcommand{\Cisodual}{\widehat{C}_{\iso}}
\newcommand{\eps}{\varepsilon}
\newcommand{\resfield}{\bE}
\newcommand{\galrep}{\rho}
\newcommand{\galchar}{\psi}
\newcommand{\automrep}{\Pi}
\newcommand{\automchar}{\Psi}
\newcommand{\CC}{\mathbb{C}}
\newcommand{\shady}{\mathrm{shady}}

\title[Vign\'eras orbifolds: isospectrality, regulators, torsion homology]{Vign\'eras orbifolds: isospectrality,\\ regulators, and torsion homology}
\dedicatory{Dedicated to the memory of Nicolas Bergeron}

\author{Alex Bartel and Aurel Page}
\date{}
\begin{document}

\vspace*{-4.4em}
\maketitle
\vspace{-2.1em}
\begin{abstract}
  We develop a new approach to the isospectrality of the orbifolds constructed by Vign\'eras.
  We give fine sufficient criteria for $i$-isospectrality in given degree $i$ and
  for representation equivalence. These allow us to produce very small exotic examples of
  isospectral orbifolds: hyperbolic $3$-orbifolds that are $i$-isospectral for all $i$ but
  not representation equivalent, hyperbolic $3$-orbifolds that are $0$-isospectral but not
  $1$-isospectral, and others. Using the same method, we also give sufficient criteria for
  rationality of regulator quotients $\Reg_i(Y_1)^2/\Reg_i(Y_2)^2$ for Vign\'eras
  orbifolds $Y_1$, $Y_2$, sometimes even when they are not isospectral.
  Moreover, we establish a link between the primes that enter in these
  regulator quotients and at which torsion homology of $Y_1$ and $Y_2$ can differ, and
  Galois representations.
%
%
\end{abstract}


\tableofcontents
\newpage

\section{Introduction}

\subsection{Isospectral orbifolds}
This paper is devoted to the venerable question \cite{Kac} of which geometric and topological
properties of a closed Riemannian manifold or orbifold are encoded in the spectrum
of the Laplace-de Rham operator~$\Delta$ on the space of
differential $i$-forms for various $i\in \Z_{\geq 0}$.

Given $i\in \Z_{\geq 0}$, two closed Riemannian orbifolds are said to be
\emph{$i$-isospectral} if the multisets of eigenvalues of $\Delta$ on the
spaces of differential $i$-forms of these two orbifolds coincide.
We abbreviate ``$0$-isospectral'' to just \emph{isospectral}, and
``$i$-isospectral for all $i$'' to \emph{$\Omega^\bullet$-isospectral}.
Sunada \cite{Sunada} and Vign\'eras \cite{vigneras} each proposed a general
construction for pairs of orbifolds that are $\Omega^\bullet$-isospectral.
Vign\'eras's method is the focus of this paper.

We briefly sketch the setup of Vign\'eras in a special case. In Section \ref{sec:Vigneras}
we describe the construction in detail and in much greater generality.
Let $F$ be a number field with exactly one complex place, let $D$ be a division quaternion
algebra over $F$ that is ramified at all real places, and let $\Gamma_1$ and $\Gamma_2$ be unit groups of two maximal orders in $D$.
A complex embedding of $F$ identifies them with two discrete subgroups of~$\GL_2(\CC)$.
For the rest of the introduction we fix two such groups. They act on hyperbolic $3$-space
$X=\GL_2(\CC)/\CC^\times\U_2(\CC)$, and under some additional conditions the closed orbifolds
$\Gamma_1\lquo X$ and $\Gamma_2\lquo X$ are isospectral.
A lot of literature is devoted to the exploration of such conditions, see
\cite{VoightLinowitz} and extensive references therein, but all known conditions
actually imply the stronger relationship that the $\rL^2$-spaces
$\rL^2(\Gamma_1\lquo \PGL_2(\CC))$ and $\rL^2(\Gamma_2\lquo \PGL_2(\CC))$
are isomorphic as unitary representations of $\GL_2(\CC)$. When that relationship holds,
we say that $\Gamma_1$ and $\Gamma_2$ are \emph{representation equivalent}.
It is a theorem of DeTurck and Gordon \cite[Theorem 1.16, Remark 1.18]{DeTurckGordon}
that the quotients by representation equivalent groups are in fact $\Omega^\bullet$-isospectral.
In particular, as far as we are aware, the smallest currently known pair of connected
isospectral hyperbolic $3$-orbifolds comes from representation equivalent groups, and was found by
Linowitz--Voight~\cite{VoightLinowitz}. Their orbifolds have (necessarily equal) volume $2.83366\ldots$.
The converse, whether $\Omega^\bullet$-isospectrality of two hyperbolic $3$-orbifolds
implies representation equivalence of the corresponding groups, has until
now been an open question \cite{Pesce}, \cite[\S 4.1]{RajanConj}, \cite[Remark 2.6]{VoightLinowitz},
\cite[Question 8.11]{LauretLinowitz}. The converse does hold for hyperbolic
$2$-orbifolds \cite{DoyleRossetti} and it has been conjectured to hold in dimension $3$,
too \cite[\S 12]{DoyleRossetti}.

In this paper we develop much finer criteria for different kinds of isospectrality
in Vign\'eras's construction. Before describing the nature of
these criteria, we state some of their applications.
The following result will be proven in Example \ref{ex:1notL2}.

\begin{introtheorem}\label{thm:IntroSmallIso}
  There exists a pair of closed connected orientable arithmetic hyperbolic $3$-orbifolds
  with volume $0.251\ldots$ that are $i$-isospectral for all $i$, but not representation equivalent,
  only one of which has a cyclic isotropy group of order $10$.
\end{introtheorem}
This might well be the smallest $\Omega^\bullet$-isospectral pair of connected
hyperbolic $3$-orbifolds.
We make two remarks on this. Firstly, there exist only finitely many hyperbolic $3$-orbifolds
of volume less than that in Theorem \ref{thm:IntroSmallIso}, see \cite{volacc}.
As far as we are aware, no isospectral pair with this property was known until now.
Secondly, we will see in Section \ref{sec:IntroSunada} that Sunada's
method can never produce a smaller pair.

It is known that there exist pairs of orbifolds that are $i$-isospectral for some $i$
but not all \cite{Gordon}. However, no such examples have been known that are $3$-dimensional
or hyperbolic \cite[Question 8.10]{LauretLinowitz}. In \cite{LauretEtAl} it is predicted that such examples should exist, but that
they must be difficult to construct. We have the following application of our criteria,
which will be proven in Example \ref{ex:0not1}.
\begin{introtheorem}\label{thm:IntroZeroNotOne}
  There exists a pair of closed connected orientable arithmetic hyperbolic $3$-orbifolds
  with volume $0.246\ldots$ that are isospectral, but not $1$-isospectral.
\end{introtheorem}
Again, one may wonder whether this is the smallest pair of connected hyperbolic isospectral $3$-orbifolds.

Our method also shows that, without any further hypotheses, the groups
$\Gamma_1$ and~$\Gamma_2$ are always ``close to'' representation equivalent, in the following sense.
Let~$j\in \{1,2\}$. As a representation of $\U_2(\CC)$ the space $\rL^2(\Gamma_j\lquo \PGL_2(\CC))$
decomposes as a direct sum $\bigoplus_V V^{m_{V,j}}$ over the irreducible representations
$V$ of $\U_2(\CC)$, where $m_{V,j}\in \Z_{\geq 0}$.
Each of the isotypical subspaces $\Omega(V,\Gamma_j)=V^{m_{V,j}}$ decomposes into a direct
sum of eigenspaces $\Omega(V,\Gamma_j)_{\Delta=\lambda}$ under the Casimir operator, which
is the appropriate generalisation of the Laplace--de Rham operator.
Weyl's law implies that as $T\to \infty$, one has
$\sum_{\lambda\leq T}\dim\Omega(V,\Gamma_j)_{\Delta=\lambda} \sim cT^{3/2}$
for some constant $c>0$. The following result, which is a special case of
Theorem~\ref{thm:KelmerConverse}, therefore implies that the difference
in the spectra for $\Gamma_1$ and $\Gamma_2$ is always vanishingly small, in the limit,
even when no criteria for isospectrality apply.
\begin{introtheorem}\label{thm:IntroKelmerConverse}
For every irreducible representation $V$ of $\U_2(\CC)$ there exists a constant $c_V\geq 0$
such that for all $T>0$ one has
$$
\sum_{\lambda\leq T}\bigl| \dim\Omega(V,\Gamma_1)_{\Delta=\lambda} -
  \dim\Omega(V,\Gamma_2)_{\Delta=\lambda} \bigr| \leq c_VT^{1/2}.
$$
\end{introtheorem}
This result defies some expert expectations. Indeed,
Kelmer proves \cite[Theorem~1]{Kelmer} that, in the opposite direction, if the groups $\Gamma_1$
and $\Gamma_2$ are close to being representation equivalent, then in fact they are
representation equivalent. He shows that if for all irreducible $\U_2(\CC)$-representations
$V$ the quotient between the left hand side in Theorem \ref{thm:IntroKelmerConverse} and $T^{1/2}$
tends to $0$ as $T\to \infty$, then~$\Gamma_1$ and~$\Gamma_2$ are representation equivalent.
He then speculates that his bound
might be far from optimal, and that perhaps this type of repulsion already happens
close to the asymptotic of Weyl's law. Theorem \ref{thm:IntroKelmerConverse}
demonstrates that in fact Kelmer's result is sharp, since there do exist groups
$\Gamma_1$ and $\Gamma_2$ as above that are not representation equivalent.

%
%
%
%
%
%

\subsection{Hierarchy of isospectralities}\label{sec:IntroHierarchy}
Suppose that $Y_1$ and $Y_2$ are closed orientable Riemannian manifolds of a common dimension $d$.
If they are $\Omega^\bullet$-isospectral, then the Cheeger--M\"uller Theorem
\cite{Cheeger,Mueller1,Mueller2} implies that one has
\begin{eqnarray}\label{eq:ChM}
  \prod_{i=0}^d\left(\frac{\Reg_i(Y_1)}{\#H_i(Y_1,\Z)_{\tors}}\right)^{(-1)^i}=
    \prod_{i=0}^d\left(\frac{\Reg_i(Y_2)}{\#H_i(Y_2,\Z)_{\tors}}\right)^{(-1)^i},
\end{eqnarray}
where, for a Riemannian manifold $Y$, the real number $\Reg_i(Y)$ is the covolume of the
lattice $H_i(Y,\Z)/H_i(Y,\Z)_{\tors}$ in the inner product space $H_i(Y,\bR)$ \cite[\S 1.6]{Raimbault}.

There has been a lot of recent interest in understanding regulators and torsion
homology of arithmetic manifolds -- see for instance~\cite{BVtorsion, BSVtorsion, brockdunfield, torsionJL,Raimbault}.
In light of equation \eqref{eq:ChM}, the following vague question seems natural.
\begin{question}
  What can be said about $\Reg_i(Y_1)/\Reg_i(Y_2)$ if $Y_1$ and $Y_2$ are $i$-isospectral manifolds?
\end{question}
Poincar\'e duality implies that for a closed orientable $d$-manifold $Y$
and for all $i$, one has $\Reg_i(Y)=\Reg_{d-i}(Y)^{-1}$. Thus,
if $d$ is even, then
equation \eqref{eq:ChM} says nothing about $\Reg_i(Y_1)/\Reg_i(Y_2)$;
while if $d$ is odd, and $Y_1$ and $Y_2$ are $\Omega^\bullet$-isospectral, equation (\ref{eq:ChM}) implies that one has
\begin{eqnarray}\label{eq:weakrationality}
\prod_{i=0}^{(d-1)/2}\left(\frac{\Reg_i(Y_1)^2}{\Reg_i(Y_2)^2}\right)^{(-1)^i}\in \Q^{\times}.
\end{eqnarray}
There is no analogue of the Cheeger--M\"uller formula for a single degree $i$. Nevertheless,
it seems natural to wonder: if $Y_1$ and $Y_2$ are merely assumed to be $i$-isospectral
for some $i$, then does it follow that $\frac{\Reg_i(Y_1)^2}{\Reg_i(Y_2)^2}$ is a rational number?
And if it is, then what primes can enter? Is there
a connection between that and the rational number $\frac{\#H_i(Y_1,\Z)_{\tors}}{\#H_i(Y_2,\Z)_{\tors}}$?
We do not know the answers to these general questions, but we prove strong results in the case
when $Y_j=\Gamma_j\lquo X$ arise from the Vign\'eras construction.

We think of different ``kinds of isospectrality'' between $\Gamma_1$ and $\Gamma_2$
as sitting in a kind of approximate hierarchy:
representation equivalence\footnote{For hyperbolic manifolds, representation equivalence of the groups is
equivalent to so-called strong isospectrality of the manifolds \cite{Pesce},
but in general it is stronger.}, which we sometimes also call $\rL^2$-isospectrality,
$\Omega^\bullet$-isospectrality, isospectrality, which we also
refer to as $\Omega^0$-isospectrality, rationality of $\frac{\Reg_i(Y_1)^2}{\Reg_i(Y_2)^2}$
for all $i$, which we refer to as $\cH^\bullet$-isospectrality, and finally, for a prime
number $p$, the conditions
\[
  \ord_p\left(\frac{\Reg_i(Y_1)^2}{\Reg_i(Y_2)^2}\right) = 0\quad\text{ and }\quad
\ord_p\left(\frac{\#H_i(Y_1,\Z)_{\tors}}{\#H_i(Y_2,\Z)_{\tors}}\right)=0
\]
for all $i$, where
$\ord_p$ denotes $p$-adic valuation, which we call $\Z_p$-isospectrality. Of course
that last condition only makes sense if the regulator quotients are rational in
the first place, i.e. under the assumption of $\cH^\bullet$-isospectrality.
For each of these kinds of isospectralities, we prove sufficient criteria.
We do not state them here precisely, but we give an impressionistic sketch 
of the general shape of our results.

For each kind $\Asterisk$ of isospectrality as above, we identify certain types
of Hecke characters such that if there are no Hecke characters of that type, then $\Gamma_1$ and $\Gamma_2$
are $\Asterisk$-isospectral, but we do not expect the converse to hold.
For that reason we do not call the characters an obstruction, but instead
call them \emph{$\Asterisk$-shady characters}: their presence merely stops
our method from proving $\Asterisk$-isospectrality.
In summary, our main results take the following shape.
\begin{introtheorem}\label{thm:IntroProtoMain}
  At least one of the following two statements is true:
  \begin{enumerate}[leftmargin=*,label=\upshape{(\roman*})]
    \item there exists a number field $L$ in an a-priori finite list and a $\Asterisk$-shady
      character of $L$;
    \item the groups $\Gamma_1$ and $\Gamma_2$ are $\Asterisk$-isospectral.
  \end{enumerate}
\end{introtheorem}

For $\rL^2$-, $\Omega^\bullet$-, $\Omega^0$-, and $\cH^\bullet$-isospectrality, we define the respective 
shady characters and prove Theorem \ref{thm:IntroProtoMain} as
Theorem~\ref{thm:repEquivPsi}, Corollaries \ref{cor:allisospectralPsi}
and~\ref{cor:zeroisospectralPsi}, and Theorem~\ref{thm:ratlRegPsi}, respectively.
Our theorem for $\Z_p$-isospectrality is proven as Theorem~\ref{thm:pshady},
subject to a widely believed 
conjecture on Galois representations, 
Conjecture~\ref{conj:attachedgalreps}.

The conditions of being $\Asterisk$-shady are completely explicit. Presence or
absence of $\Asterisk$-shady can be checked efficiently using existing algorithms
for computations of Hecke characters \cite{MolinPage}. Moreover, we have
the following implications between existence of different shady characters:
\[
\xymatrix@R+1pc@C+5pc{
  *++[F-,]\txt{$\rL^2$-shady} \ar@{=>}[]!<7ex,0ex>;[r]!<-6.6ex,0ex> &
  *++[F-,]\txt{$\Omega^\bullet$-shady}\ar@{=>}[]!<7ex,0ex>;[r]!<-6.6ex,0ex> \ar@{=>}[]!<0ex,-3.5ex>;[d]!<0ex,3ex> & *++[F-,]\txt{$\Omega^0$-shady}\\
          & *++[F-,]\txt{$\cH^\bullet$-shady}\ar@{=>}[]!<7ex,1ex>;[r]!<-6.6ex,1ex>^(.75){\txt{for all $p$}} &
          *++[F-,]\txt{$\Z_p$-shady}\ar@{=>}[]!<-7ex,-1ex>;[l]!<7ex,-1ex>^(.2){\txt{for $\infty$-ly\\ many $p$}}
}
\]

It is worth noting, in particular, that our criteria for rationality of regulator quotients
are \emph{weaker} than those for $i$-isospectrality, and in fact 
in Example \ref{ex:Hnot0betti} we exhibit the following peculiar phenomenon.

\begin{introtheorem}\label{thm:IntroRatlRegQuo}
There exists a pair of closed connected orientable hyperbolic $3$-orbifolds $Y_1$, $Y_2$
with volume 5.902\ldots that are not isospectral, nor $1$-isospectral, and for which
$\dim H_1(Y_1,\R) = \dim H_1(Y_2,\R)=1$, yet $\Reg_1(Y_1)^2/\Reg_1(Y_2)^2$ is rational.
\end{introtheorem}

%
%
%
%
%

\subsection{Homology and regulators: theorems and musings}\label{sec:IntroTorsionLL}

From a number theoretic point of view, it is natural to consider the homology of
arithmetic orbifolds together with their action of Hecke operators.
Taking this additional structure into account, we considerably strengthen our results on the
homology of the~$Y_j$ sketched in Theorem~\ref{thm:IntroProtoMain}.

The ``right'' setup for Vign\'eras's construction is an ad\'elic one, in which
$Y_1=\Gamma_1\lquo X$ and $Y_2=\Gamma_2\lquo X$ are two connected components of one
orbifold $\cY$. For now we continue with a special case.
Let~$F$ be a number field, let $\Z_F$ be its ring of integers, and let~$D$ be a
division quaternion algebra over~$F$ such that
at least one infinite place of~$F$ splits in~$D$.
There is a ray class group~$C$ that parametrises left ideal classes of maximal
orders of~$D$, and thereby connected components of $\cY$;
for simplicity of the exposition, assume that~$C$ has order~$2$,
corresponding to a quadratic extension~$L$ of~$F$, so that one has $\cY=Y_1\sqcup Y_2$.
Let~$\abstracthecke_1$ denote the Hecke algebra generated by Hecke operators~$T_\fa$
for all ideals~$\fa$ of~$\Z_F$ whose class in~$C$ is trivial.
Then~$\abstracthecke_1$ naturally acts on each of the homology groups~$H_i(Y_j,\Z)$ for $j\in \{1,2\}$, and for
every prime number~$p$ we have a decomposition of~$\Z_p\otimes\abstracthecke_1$-modules

\[
  H_i(Y_j,\Z_p) \cong \bigoplus_\fn H_i(Y_j,\Z)_\fn,
\]
where the sum ranges over maximal ideals~$\fn$ of~$\abstracthecke_1$ of residue
characteristic~$p$.
We say that such a maximal ideal~$\fn$ \emph{corresponds to} a $\Z_p$-shady
character~$\automchar$ of~$L$ if for every prime ideal~$\fq$ of~$\Z_F$ that splits in~$L$,
we have~$T_{\fq} \bmod\fn =  \automchar(\fQ)+\automchar(\fQ')$, where~$\fq\Z_L =
\fQ\fQ'$. We prove the following result as Theorem~\ref{thm:pshadylocal}.

\begin{introtheorem}\label{thm:IntroTorsionLL}
  Assume Conjecture \ref{conj:attachedgalreps}.
  Let~$p$ be an prime number, and let~$\fn$ be a maximal ideal
  of~$\abstracthecke_1$ of residue characteristic~$p$ that does not correspond to
  a~$\Z_p$-shady character.
  Then for every~$i$, there is an isomorphism of~$\Z_p\otimes \abstracthecke_1$-modules
      \[
        H_i(Y_1,\Z)_{\fn} \cong H_i(Y_2,\Z)_{\fn}.
      \]
\end{introtheorem}

We could also prove an unconditional version, at the cost of replacing the
condition on~$\fn$ with a less useful one, namely $T_{\fq^2} \equiv
-\norm(\fq) \bmod \fn$ for all~$\fq$ inert in~$L$, which can be shown to be equivalent to the
previous one under Conjecture~\ref{conj:attachedgalreps}.

It is instructive to compare Theorem~\ref{thm:IntroTorsionLL} with work of
Calegari and Venkatesh.
In~\cite{torsionJL}, they study the torsion homology and regulators in the more
difficult setting of Jacquet--Langlands
pairs $(Y_1',Y_2')$ of orbifolds, which share some similarities with Vign\'eras
orbifolds.
The orbifolds~$Y_1'$ and~$Y_2'$ are not isospectral, but their spectra are closely
related. In this setting there is also a Hecke algebra~$\abstracthecke'$ generated by
Hecke operators~$T_\fp$ for maximal ideals~$\fp$ of~$\Z_F$, and acting on
all~$H_i(Y_j',\Z)$. We say that a maximal ideal~$\fm$ of~$\abstracthecke'$ is
\emph{Eisenstein} if there exist Hecke characters~$\automchar,\automchar'$ of~$F$ such
that~$T_\fp\bmod \fm = \automchar(\fp)+\automchar'(\fp)$ for all but finitely
many~$\fp$.
Calegari and Venkatesh define a certain quotient~$H_i(Y_j',\Z)^{\new}$
of~$H_i(Y_j',\Z)$ and formulate the following
conjecture~\cite[Conjecture 2.2.8]{torsionJL}.

\begin{conjecture}[Calegari--Venkatesh]\label{conj:torsionJL}
  Let~$(Y_1',Y_2')$ be a Jacquet--Langlands pair of $3$-orbifolds.
  Let~$\fm$ be a non-Eisenstein maximal ideal of~$\abstracthecke'$.
  Then we have
  \[
    |H_1(Y_1',\Z)_\fm^{\new}| = |H_1(Y_2',\Z)_\fm^{\new}|.
  \]
\end{conjecture}

In addition, they prove several partial results~\cite[Section
6.8]{torsionJL} towards their conjecture, but only when averaging over all
maximal ideals~$\fm$, hence forgetting the Hecke action.
The Calegari--Venkatesh conjecture is an analogue for torsion homology of a
famous theorem from the theory of automorphic representations, the
Jacquet--Langlands correspondence (see Theorem~\ref{thm:JL}).
We now explain how Theorem~\ref{thm:IntroTorsionLL} can be seen as a torsion
analogue of another automorphic phenomenon, studied by Labesse and
Langlands~\cite{LabesseLanglands}. The relationship between
Vign\'eras's construction of isospectral orbifolds and the work of
Labesse--Langlands was first exposed by Rajan~\cite{Rajan}.
The Langlands programme postulates the existence of a compact group~$L_F$ and
conjectures that, to each
cuspidal automorphic representation~$\automrep$ of~$D^\times$, one should be
able to attach a
continuous irreducible representation, called a Langlands parameter,
\[
  \varphi = \varphi_{\automrep} \colon L_F\to \GL_2(\CC),
\]
such that $\varphi_{\automrep}\cong \varphi_{\automrep'}$ if and only
if~$\automrep\cong \automrep'$. Moreover,
each such~$\automrep$ appears with multiplicity~$1$ in
the space of automorphic forms (see Theorem~\ref{thm:MultOne}).
In contrast, to each cuspidal automorphic representation~$\automrep$
of the kernel~$D^1$ of the reduced norm~$D^\times\to F^\times$,
one should be able to attach an irreducible Langlands parameter
\[
  \varphi \colon L_F \to \PGL_2(\CC),
\]
but several
non-isomorphic~$\automrep$ can have isomorphic parameters~$\varphi$: they are said
to form an $L$-packet. While the existence of the group~$L_F$ and the
parameters~$\varphi$ are still conjectural, Labesse and Langlands gave an ad-hoc
definition of $L$-packets for~$D^1$, realised that sometimes not
all admissible~$\automrep$ in the same $L$-packet are automorphic, and proved a formula for
the multiplicity of each~$\automrep$ in the space of automorphic forms.
This manifests concretely as follows. To every~$\automrep = \automrep_\infty
\otimes \automrep_f$ corresponds a Casimir eigenvalue
(coming from~$\automrep_\infty$) and system of eigenvalues for the Hecke
operators of~$D^1$ (coming from~$\automrep_f$), and these eigenvalues are the
same for every~$\automrep$ in an $L$-packet. However, 
given a compact open subgroup~$K_f$ of the points of~$D^1$ over
the finite adeles, the dimension of the space~$\automrep_f^{K_f}$
of fixed points under~$K_f$ may depend on~$\automrep$, even when
$\automrep$ varies inside a fixed $L$-packet.
This may result in the system of eigenvalues attached to a given $L$-packet
appearing with different multiplicities in the $\rL^2$-space attached to $\Gamma$ for the various arithmetic
groups~$\Gamma = K_f \cap D^1$.
If the distributions of the automorphic multiplicities in $L$-packets were
completely random, Vign\'eras pairs of orbifolds would never be isospectral.
However, as Labesse and Langlands proved \cite[Proposition
7.2]{LabesseLanglands} (see also \cite[Theorem 4]{Rajan}), this is far from true.

\begin{theorem}[Labesse--Langlands]\label{thm:LabesseLanglands}
  Suppose~$\varphi\colon L_F \to \PGL_2(\CC)$ is not induced from a character of
  an index~$2$ subgroup. Then the automorphic multiplicity of every~$\automrep$ in the
  $L$-packet of~$\varphi$ is the same.
\end{theorem}

This strong restriction allowed Rajan to prove \cite[Theorem 2]{Rajan} the
representation equivalence of many orbifolds generalising Vign\'eras's construction.
Returning to torsion homology, let~$p>2$ be a prime number, and let~$G_F$ be the
absolute Galois group of~$F$.
For every maximal ideal~$\fn$ of~$\abstracthecke_1$ with residue field~$\F$
of characteristic~$p$, there should conjecturally exist a continuous semisimple Galois
representation
\[
  \galrep \colon G_F \to \GL_2(\F)
\]
such that one has~$\Tr \galrep(\Frob_\fq)
\equiv T_\fq \bmod \fn$ for almost all~$\fq$ whose class in~$C$ is trivial, see Conjecture
\ref{conj:attachedgalreps}.
In fact, the maximal ideal~$\fn$ does not completely determine~$\galrep$ but it
completely determines its projectivisation
\[
  P\galrep \colon G_F \to \PGL_2(\F).
\]
One may think of~$P\galrep$ as the ``Langlands parameter'' of the ``$L$-packet''
corresponding to~$\fn$. In addition, the condition that~$\fn$
correspond to a $\Z_p$-shady character~$\automchar$ of~$L$ is equivalent
to~$\galrep\cong \Ind_{G_F/G_L}\galchar$, where~$\galchar$ corresponds
to~$\automchar$ via class field theory. This makes
Theorem~\ref{thm:IntroTorsionLL} a torsion analogue of
Theorem~\ref{thm:LabesseLanglands}.
More generally, Labesse and Langlands prove a formula for the automorphic
multiplicity of every~$\automrep$.
We think that this naturally leads to a very interesting question.

\begin{question}
  Can one formulate a torsion analogue of the full Labesse--Langlands
  multiplicity formula? Can one prove it?
\end{question}

We also obtain a Hecke-equivariant refinement of our results on the $p$-adic
valuation of regulator ratios. It may seem difficult to make sense of it, as
regulators are only numbers and therefore cannot afford an action of the Hecke
algebra. One might hope to have a factorisation of the form~$\Reg_i(Y_j) =
\prod_\fn\Reg_i(Y_j)_\fn$, but this is unlikely to be possible, as regulators
come from constructions over~$\R$, while decompositions according to maximal
ideals of~$\abstracthecke_1$ are $p$-adic in nature.
However, we show that such a decomposition exists for the ratios of regulators,
under conditions ensuring that they are rational numbers.
More precisely, to certain modules~$M$ we attach a number~$\cC(M)$ that we call
a \emph{regulator constant} by analogy with \cite{tamroot}, satisfying
\[
  \cC(H_i(\cY,\Z)) \in \Q^\times,
\]
and for every maximal ideal~$\fn$ of~$\abstracthecke_1$ of residue characteristic~$p$,
\[
  \cC(H_i(\cY,\Z)_\fn) \in \Q_p^\times/(\Z_p^\times)^2.
\]
We will provide more details on our construction of regulator constants in
Section~\ref{sec:IntroMethodReg}. The following theorem is a combination of
Lemma \ref{lem:RegConstFormalism}, Theorem \ref{thm:ratlRegPsi}, and Theorem \ref{thm:pshadylocal}.

\begin{introtheorem}\label{thm:IntroRegLL}
  Assume that there is no~$\cH^\bullet$-shady character of~$L$. Then
  the following assertions hold.
  \begin{enumerate}[leftmargin=*,label={\upshape(\arabic*)}]
  \item\label{item:IntroRegLLRegConst} For
    all~$i$ we have
    \[
      \frac{\Reg_i(Y_1)^2}{\Reg_i(Y_2)^2} = \cC(H_i(\cY,\Z)).
    \]
  \item\label{item:IntroRegLLProd} Let $p$ be a prime number. Then for all~$i$ we have
    \[
      \cC(H_i(\cY,\Z))
      \equiv \prod_\fn\cC(H_i(\cY,\Z)_\fn)
      \bmod (\Z_p^\times)^2,
    \]
    where the product ranges over maximal ideals~$\fn$ of~$\abstracthecke_1$ of residue
    characteristic~$p$.
  \item\label{item:IntroRegLLTriv} Let $p>2$ be a prime number. Assume, in addition, Conjecture \ref{conj:attachedgalreps}.
    Let~$\fn$ be a maximal ideal
    of~$\abstracthecke_1$ of residue characteristic~$p$ that does not correspond to
    a~$\Z_p$-shady character.
    Then for all~$i$ we have
    \[
      \cC(H_i(\cY,\Z)_\fn) \in \Z_p^\times.
    \]
  \end{enumerate}
\end{introtheorem}

The definition of~$\cC(H_i(\cY,\Z)_\fn)$ involves the choice of arbitrary non-degenerate
$\Q_p$-valued pairings on the free parts of~$H_i(Y_j,\Z)_\fn$ satisfying certain properties.
We speculate that there should exist a canonical construction of such pairings,
leading to a notion of $p$-adic regulators~$\Reg_{i,p}(Y_j)\in\Q_p^\times/(\Z_p^\times)^2$,
compatible with localisation at maximal ideals of~$\abstracthecke_1$ and such that one has
\[
  \frac{\Reg_{i,p}(Y_1)^2}{\Reg_{i,p}(Y_2)^2} \equiv \cC(H_i(\cY,\Z))
  \bmod (\Z_p^\times)^2.
\]
%
%
\begin{question}
  Can one define a notion of $p$-adic regulator in this context? Of $p$-adic analytic torsion?
  Is there a $p$-adic Cheeger--M\"uller formula?
\end{question}

Finally, note that the Cheeger--M\"uller formula merely provides a motivation for our results
on torsion homology and regulators, we do not use this formula in our proofs.
In particular our methods are very different from those of~\cite{torsionJL}.
What makes the Vign\'eras case more accessible than the Jacquet--Langlands case
is that the orbifolds~$Y_1$ and~$Y_2$ are commensurable, unlike Jacquet--Langlands pairs.
The next section explains how we exploit this commensurability.


\subsection{Our approach to isospectrality}\label{sec:IntroMethod}
The basic idea of our approach is applicable much more widely than
in the setting sketched in the previous subsections. In Section \ref{sec:Vigneras},
where we explain the general setup, we take $D$ to be any division
algebra over an arbitrary number field $F$. An ad\'elic construction then yields
a closed orbifold $\cY$, and the only restriction on $D$ is that $\cY$ be positive-dimensional.
One has a Hecke algebra $\abstracthecke$
acting on, say, the space $\Omega^i(\cY)$ of differential $i$-forms on $\cY$,
and commuting with the action of the Laplace--de Rham operator $\Delta$.
That space decomposes as
\[
  \Omega^i(\cY)=\bigoplus_{c\in C}\Omega^i(Y_c)
\]
with the direct summands corresponding to the connected components $Y_c$ of
$\cY$, indexed by a certain class group~$C$.
The Hecke algebra is generated by elements $T_{\fp}$ indexed by maximal ideals $\fp$
of $\Z_F$. The operators $T_{\fp}$ permute the direct summands~$\Omega^i(Y_c)$ of $\Omega^i(\cY)$,
sending~$\Omega^i(Y_c)$ to~$\Omega^i(Y_{[\fp]c})$ where~$[\fp]$ denotes the class of $\fp$ in~$C$.
In particular, the algebra $\abstracthecke_1$ introduced in
Subsection~\ref{sec:IntroTorsionLL} is the subalgebra of $\abstracthecke$ consisting of
those Hecke operators that preserve all direct summands in the above decomposition.

Suppose once again, for simplicity, that we have $\cY=Y_1\sqcup Y_2$, and let
$L/F$ be the quadratic extension corresponding to~$C$, so that $T_{\fp}$
preserves the direct summands~$\Omega^i(Y_j)$ if and only if $\fp$ splits in $L/F$.
Fix $\lambda\in \R$. The algebra $\abstracthecke$ acts on the $\lambda$-eigenspace
$\Omega^i(\cY)_{\Delta=\lambda}$.
If there exists a Hecke operator $T\in \abstracthecke$ inducing a map
$$
T\colon \Omega^i(Y_1)_{\Delta=\lambda}\to\Omega^i(Y_2)_{\Delta=\lambda}
$$
and acting invertibly on $\Omega^i(\cY)_{\Delta=\lambda}$, then that Hecke operator witnesses
equal multiplicity of $\lambda$ in the spectra of $Y_1$ and $Y_2$ by defining an explicit
isomorphism between the respective $\Delta$-eigenspaces.
A priori one might imagine that such an invertible Hecke operator can easily fail
to exist in such a way that for every $\fp$ that is inert in $L/F$,
there is some line in $\Omega^i(\cY)_{\Delta=\lambda}$ that is annihilated by $T_{\fp}$. If these
different lines had nothing to do with each other, then it would be completely
unclear how to control when this happens. However, we show in Proposition
\ref{prop:littleExercise} that if no Hecke operator $T\in \abstracthecke$ that sends $\Omega^i(Y_1)_{\Delta=\lambda}$
to $\Omega^i(Y_2)_{\Delta=\lambda}$ acts invertibly, then in fact there is a \emph{single}
line in $\Omega^i(\cY)_{\Delta=\lambda}$ that is annihilated by $T_{\fp}$ for every prime $\fp$
that is inert in $L/F$. This is the first main result of Section \ref{sec:gradedAlg}.

We deduce that if no Hecke operator that maps $\Omega^i(Y_1)_{\Delta=\lambda}$
to $\Omega^i(Y_2)_{\Delta=\lambda}$ acts invertibly,
then there is a simultaneous eigenvector in $\Omega^i(\cY)_{\Delta=\lambda}$ under all the
Hecke operators such that ``half'' of the Hecke operators $T_{\fp}$ have eigenvalue $a_{\fp}=0$.
In other words, we obtain a Hecke eigenvalue system $(a_{\fp})_{\fp}$ such that for all
$\fp$ one has 
$$
  a_{\fp} = \chi(\fp)a_{\fp},
$$
where $\chi\colon C\to \CC^\times$ is the quadratic Hecke character
attached by class field theory to the unique non-trivial character of the Galois group $\Gal(L/F)$.
This \emph{self-twist condition} might remind the reader of coefficients $a_p$ attached to an 
elliptic curve with complex multiplication. The associated Galois representation
is then induced from a linear character of the absolute Galois group of the
quadratic extension~$L$. 

It is at this point in our analysis that we specialise to the case of quaternion algebras.
In this case, as the example of elliptic curves suggests,
the automorphic representation of $\GL_2$ over $F$ attached
to the Hecke eigenvalue system $(a_{\fp})_{\fp}$ as above is
the automorphic induction of a Hecke character
from the quadratic extension~$L/F$. In Section \ref{sec:isoconds}
we use the explicit classification of representations of $\GL_2$ to arrive at precise
conditions on this Hecke character, and thus at our definitions of $\Asterisk$-shady
characters, and to prove various instances of Theorem \ref{thm:IntroProtoMain}.

\subsubsection*{Technical discussion: algebraic axiomatisation}
In Section \ref{sec:gradedAlg} we work in a completely abstract algebraic setting:
the class group parametrising
the connected components of $\cY$ is replaced by an arbitrary finite abelian group $C$,
the Hecke algebra $R\otimes_{\Z}\abstracthecke$ for a suitable ring $R$ of scalars becomes a
commutative $R$-algebra $\hecke$ graded by $C$, the space
that the Hecke algebra acts on, such as $\Omega^i(\cY)$ or $H_i(\cY)$, is just a graded $\hecke$-module.
Also in Section \ref{sec:gradedAlg} the reader will encounter a certain group homomorphism
$\nu\colon W\to C$, which could benefit from some motivating comments. Some of the components
of $\cY$ correspond to orders in the division algebra that are conjugate to each other,
so that the components are isometric to each other. To obtain the strongest results,
we do not limit ourselves to Hecke operators that map $\Omega^i(Y_1)$ to $\Omega^i(Y_2)$,
but consider all those mapping $\Omega^i(Y_1)$ to $\Omega^i(Y)$ for any component $Y$ isometric to $Y_2$.
Formalising this is not just a simple matter of passing to the quotient by this
equivalence relation, since such a quotient does not carry a well-defined $\hecke$-action.
The purpose of $C_{\iso}=C/\nu(W)$ is to parametrise connected components up to this equivalence.
If $\Cisodual$ were replaced by
$\Cdual$ in all these results, then we would only pick out the elements of $\hecke$
that map $\Omega^i(Y_1)$ to $\Omega^i(Y_2)$ and would get more restrictive sufficient
conditions for isospectrality.

\subsection{Comparison with the Sunada method}\label{sec:IntroSunada}
Sunada~\cite{Sunada} constructs pairs of $\Omega^\bullet$-isospectral orbifolds as follows:
let $M$ be a Riemannian orbifold, let $G$ be a finite group acting
on $M$ by isometries, and let $H$, $H'\leq G$ be two subgroups such that the linear representations
$\CC[H\lquo G]$ and $\CC[H'\lquo G]$ of $G$ are isomorphic. Then the quotients $H\lquo M$ and
$H'\lquo M$ are $\Omega^\bullet$-isospectral.

The analogy between the permutation representations $\CC[H\lquo G]$ and the unitary
representations $\rL^2(\Gamma\lquo \PGL_2(\CC))$ has led many authors
to view Vign\'eras's method as a special case of an extension of Sunada's. However, as
Theorems \ref{thm:IntroSmallIso} and \ref{thm:IntroZeroNotOne} illustrate,
this framing does not do justice to the flexibility of Vign\'eras's construction, which
is capable of producing isospectral orbifolds without isomorphisms of $\rL^2$-spaces.
Instead, in this paper we make the case that the parallel between the two methods
is of a different nature: as we explained in Section \ref{sec:IntroMethod}, our conditions
ensure that the isospectralities in Vign\'eras's construction are realised
by Hecke operators, which one can view as so-called \emph{transplantation} maps.
It is well known that Sunada's construction
can be formulated in such a way that the same is true there \cite[\S 2.2]{GordonSurvey}.

We claimed earlier that Sunada's construction cannot yield connected hyperbolic $3$-orbifolds
of smaller volume than that appearing in Theorem \ref{thm:IntroSmallIso}. We are now
in a position to explain this. The orbifolds $H\lquo M$ and $H'\lquo M$ as above cover a common
orbifold, namely $G\lquo M$, with degree $[G:H] = [G:H']$. The smallest hyperbolic
$3$-orbifold has volume $0.039\ldots$ \cite{MarshallMartin}. Moreover, it is known that among all finite groups
$G$ and non-conjugate subgroups $H$, $H'$ such that $\CC[H\lquo G]$ and $\CC[H'\lquo G]$ are isomorphic
$G$-representations, the smallest index $[G:H]$ is $7$ 
(realised by two subgroups of
$G=\GL_3(\F_2)$) \cite[Theorem 3]{Perlis}. Thus the smallest orbifolds that Sunada's method could produce have
volume at least $7\cdot 0.039\ldots = 0.273\ldots$.

In \cite{us1} we showed that the regulator quotients of Sunada-isospectral
manifolds $H\lquo M$, $H'\lquo M$ are rational numbers, and that for all $i\in \Z_{\geq 0}$
and all prime numbers $p\nmid\#G$ we have
\[
  \ord_p\left(\frac{\Reg_i(H\lquo M)^2}{\Reg_i(H'\lquo M)^2}\right) = 0
  \text{ and }
  \ord_p\left(\frac{\#H_i(H\lquo M)_{\tors}}{\#H_i(H'\lquo M)_{\tors}}\right) = 0.
\]
If $Y_1$, $Y_2$ are orbifolds that come out of Vign\'eras's construction,
then they typically do not sit in a common finite covering, so there is no obvious
analogue of the condition $p\nmid \#G$ that would exclude all but finitely many
primes from contributing to the regulator quotients.

In Theorem \ref{thm:finiteSetS} we find an analogue of ``the set of prime divisors of $\#G$''
in the Vign\'eras setting.

\begin{introtheorem}\label{thm:IntroFiniteSetS}
  Assume Conjecture \ref{conj:attachedgalreps}.
  Then there is an explicit set $S$ of prime numbers such that:
  \begin{itemize}[leftmargin=*]
    \item every prime number $p$ for which there exists a $\Z_p$-shady character is contained in $S$, and
    \item the set $S$ is infinite if and only if there exists a $\cH^\bullet$-shady character.
  \end{itemize}
\end{introtheorem}

\subsection{Our approach to rationality of regulator quotients}\label{sec:IntroMethodReg}
As we explained in Section \ref{sec:IntroMethod}, we obtain conditions under
which the spaces $H_i(Y_1,\R)$ and $H_i(Y_2,\R)$ are not merely abstractly isomorphic,
but such that there are isomorphisms given by Hecke operators. In particular,
they preserve the rational structure. That by itself is still not enough to show
that $\Reg_i(Y_1)^2/\Reg_i(Y_2)^2$ is rational. The crucial additional property
is that the Hecke algebra is stable under adjoints
with respect to the harmonic forms pairing, which is used to define the regulators.
This is sufficient to imply rationality of regulator quotients and allows us to
analyse them.

In order to show this, we once again axiomatise the situation purely algebraically.
In Section \ref{sec:regconstHecke} we first describe the very general setup.
As a by-product, we reprove in a Hecke-algebraic way a foundational result of
the Dokchitsers on so-called regulator constants,
Theorem \ref{thm:Dokchitsers}, which we used in \cite{us1,us2} to analyse the
regulator quotients of Sunada-isospectral manifolds.

Then in Section \ref{sec:polarisations} we prove that the existence of a suitable
pairing on a graded module $M$ with the adjointness property described above is equivalent
to the existence of such a pairing that takes values in $\Q$.
The crucial ingredient in this proof is a nice criterion, Proposition \ref{prop:polarisation},
for existence of pairings with the required properties.
The following is a special case. We refer to Section \ref{sec:gradedAlg} for the
definitions of the standard algebraic terms used here.
\begin{introprop}\label{prop:IntroPolarisation}
  Let $C$ be a finite abelian group, let $\heckebig=\bigoplus_{c\in C} \heckebig_c$ be a
  $C$-graded commutative reduced finite-dimensional $\Q$-algebra, and
  let $M=\bigoplus_{c\in C}M_c$ be a finitely generated graded $\heckebig$-module.
  Let $L$ be a field of characteristic $0$.
  Then the following are equivalent:
  \begin{itemize}[leftmargin=*]
    \item there exists a non-degenerate $L$-valued bilinear pairing on $M$ with respect to
      which all $T\in \heckebig$ are self-adjoint and such that for all $c\neq c'\in C$
      the homogeneous components~$M_c$ and $M_{c'}$ are orthogonal to each other;
    \item for every $c\in C$ the $(L\otimes \heckebig_1)$-module
      $L\otimes M_c$ is self-dual.
  \end{itemize}
\end{introprop}
The descent of pairings from $\CC$ to $\Q$ easily follows from this and the Deuring--Noether
Theorem, see Proposition \ref{prop:descend-polarisation}.

In Section \ref{sec:gradedregconst} we define regulator quotients $\cC(M)$ in this
abstract setting, a~priori with respect to a given pairing,
but then we use the formalism of Section~\ref{sec:regconstHecke} to show that
in fact the values of such regulator quotients do not depend on the pairing.
In conjunction with the previously mentioned existence of $\Q$-valued pairings
this proves rationality of regulator quotients of Vign\'eras pairs.
Moreover, this formalism is very flexible, and allows us to prove Theorem \ref{thm:IntroRegLL}.

As a by-product of our approach, we have stumbled upon a curious connection between
fields generated by Hecke eigenvalues and regulator quotients. We do not pursue this
connection seriously in this paper, but Proposition \ref{prop:regcst-squares} is an
example of this phenomenon. The following is a special case.
\begin{introprop}\label{prop:IntroHeckeField}
  Suppose that $\cY$ has exactly two connected components, $Y_1$ and~$Y_2$,
  suppose that $\dim H_1(Y_1,\CC)=\dim H_1(Y_2,\CC)=1$, and suppose that there
  is no $\cH^\bullet$-shady character. Then the field generated by the Hecke eigenvalues of
  $\abstracthecke$ acting on~$H_1(\cY,\CC)$ is $\Q(\Reg_1(Y_1)/\Reg_1(Y_2))$.
\end{introprop}

It was the observation of this numerical coincidence that suggested to us that
it should be possible to expand the Dokchitsers' formalism of regulator constants
to Hecke algebras, and led us to our Hecke operator approach to isospectrality.

\subsection{Comparison with other approaches to Vign\'eras pairs}
In most of the existing work devoted to Vign\'eras's construction~\cite{vigneras}, the
isospectrality is derived from the trace formula via so-called
\emph{selectivity} conditions~\cite{VoightLinowitz,Linowitz}, from which one
determines which~$\Gamma_i$ intersect a given conjugacy class.
The trace formula suggests a kind of duality between this approach and ours: in the
selectivity method one studies the geometric side, while we study the spectral side.
In Corollary~\ref{cor:LinowitzVoight}, we recover a sufficient criterion for
representation equivalence that had been obtained from the selectivity method.
We believe that one might be able to obtain results on finer notions of
isospectrality, e.g. $\Omega^\bullet$- and $\Omega^0$-isospectrality, from the
selectivity method and a delicate analysis of the trace formula. It would be
interesting to work out such criteria explicitly. With our approach, working out
conditions for different notions of isospectrality amounts to different but
routine local computations. The input from the trace formula required for our
method is packaged in Langlands's automorphic induction theorem.

An alternative approach, due to Rajan~\cite{Rajan}, is to use the
Labesse--Langlands multiplicity formula. In Corollary~\ref{cor:Rajan}, we
recover a variant of a sufficient criterion for representation equivalence that
had been obtained by Rajan. Although the Labesse--Langlands work has so far only
been applied to representation equivalence, it is likely that one could recover
the results of Section~\ref{sec:isoconds} on finer notions of isospectrality
with this approach.

Our approach has two advantages:
firstly, it allows us to realise the isospectralities through direct identification
of eigenfunctions rather than just a numerical comparison of multiplicities;
and secondly and most importantly for our purposes, it applies to the torsion setting
thanks to the Hecke action on $H_i(\cY,\Z_p)$.

Finally, we expect that one might be able to extract some information on the algebraicity
or even rationality of the regulator quotients by relating them to special values of $L$-functions.
For an analogous observation in the Jacquet--Langlands setting see \cite[\S 6.5]{torsionJL}.
However, it seems hard to obtain precise information on the primes that enter
or Hecke equivariance properties along the lines of Theorem~\ref{thm:IntroRegLL}.

\subsection{Notation and conventions}
All our modules will be left modules. A module that is a direct sum of a number
of copies of a simple module is called \emph{isotypical}.

If $B$ is a ring, $n\in \Z_{\geq 1}$, and $M$ is a $B$-module, then $\MaxSpec(B)$
denotes the set of maximal ideals of $B$, the ring of $n\times n$ matrices over $B$
is denoted by~$\Mat_n(B)$, and $\Ann_B(M)$ denotes the annihilator of $M$ in $B$.
If $Q$ is a field and~$A$ is a $Q$-algebra, then $\norm_{A/Q}$ denotes
the $Q$-algebra norm.
The sign $\otimes$ denotes tensor product over~$\Z$.
If~$R$ is a ring, $M$ is an $R$-module, and $L$ an $R$-algebra, then we abbreviate the
$L$-module $L\otimes_R M$ to $M_L$.
By an \emph{ideal} without further specification we mean a two-sided ideal.
The \emph{radical} $\Jac(R)$ of a ring $R$ is defined as the
intersection of its maximal left ideals. It is a two-sided ideal.
%
If $p$ is a prime number, then
$\Z_{(p)}=\{ a/b : a, b\in \Z, b \not \in p\Z \}$ denotes
the localisation of $\Z$ at $p$, and~$\Z_p$ denotes the ring of $p$-adic
integers.

If $C$ is a finite abelian group, then $\widehat{C}=\Hom(C,\CC^\times)$ denotes its group of complex characters.
If $X$ is a set, $\sigma$ is an automorphism of $X$, and $f$ is a function
defined on~$X$, then~$f^\sigma$ denotes the function $x\mapsto f(\sigma x)$.
We will also use this notation when~$X$ is a normal subgroup of a group containing $\sigma$,
in which case the automorphism of $X$ is $x\mapsto \sigma x \sigma^{-1}$.
We denote the centre of a group $G$ by $\centre(G)$.

When talking about homology, differential forms, etc. on orbifolds, we always
intend the meaning of these words in the orbifold sense, rather than just
in reference to the underlying topological space. We refer to \cite{GordonOrbifolds,Caramello}
for the definitions and many additional details. Note that all our orbifolds
are ``good'' in the sense of ibid.: they have (finite) coverings by manifolds.

\begin{acknowledgements}
Several people were very generous with their time and expertise.
We would like to thank Andy Baker, Nicolas Bergeron, Frank Calegari, Fred Diamond, Nathan Dunfield,
Toby Gee, Emilio Lauret, Hendrik W. Lenstra Jr., Ben Linowitz, Mike Lipnowski, Adam Morgan, Olivier Ta\"ibi, Jack Thorne, Akshay Venkatesh, and
  John Voight for very helpful conversations, comments, and answers to our questions.
  We are particularly grateful to Frank Calegari for fixing our original formulation of Conjecture \ref{conj:attachedgalreps}.
  Part of this research was done during a Research in Pairs visit at Mathematisches Forschungsinstitut
  Oberwolfach. We are very grateful to the Institut for providing an amazing research environment.
Parts of this research were funded by
   EPSRC Fellowship EP/P019188/1,
  `Cohen–Lenstra heuristics, Brauer relations, and low-dimensional manifolds' 
  and by the ANR AGDE project (ANR-20-CE40-0010).
\end{acknowledgements}

%
%
\section{Regulator constants via Hecke algebras}\label{sec:regconstHecke}

In this section we start by describing a very general setting in which one can
define algebraic invariants that specialise to regulator quotients in certain
number theoretic and geometric situations. We then explain how the regulator
constants of Dokchitser--Dokchitser attached to Brauer relations of finite groups
\cite{tamroot} can be viewed as a special case of this construction, and we use
this description to give a new basis-free proof of their main algebraic theorem,
\cite[Theorem 2.17]{tamroot}. That last application will not be needed in the
remainder of the paper.

\subsection{Adjoint pairs}\label{sec:adjointpairs}
Throughout this section, let $R$ be a commutative domain, and let $Q$ be its
field of fractions. An \emph{isogeny} of $R$-modules is a homomorphism of
$R$-modules with $R$-torsion kernel and cokernel, equivalently a homomorphism
$\phi$ of $R$-modules such that $Q\otimes_R \phi$ is an isomorphism of
$Q$-vector spaces. Let $N_1$ and~$N_2$ be free finite rank $R$-modules, and
let $\phi\colon N_1\to N_2$, $\phi^*\colon N_2\to N_1$ be isogenies.
If~$B_i$ is an $R$-basis for $N_i$ for $i=1$, $2$, then we define
$$
\cD_{B_1,B_2,\phi,\phi^*} = \frac{\det_{B_1,B_2}\phi}{\det_{B_2,B_1}\phi^*} \in Q^\times,
$$
where the determinants are computed with respect to the bases $B_1$ and $B_2$.
If $B_1'$ and $B_2'$ are other $R$-bases of $N_1$, respectively $N_2$, then
$$
\cD_{B_1,B_2,\phi,\phi^*}\equiv \cD_{B_1',B_2',\phi,\phi^*}\mod{(R^\times)^2},
$$
so we have a well-defined invariant
$$
\cD_{\phi,\phi^*} = \cD_{B_1,B_2,\phi,\phi^*}(R^\times)^2\in Q^\times/(R^\times)^2
$$
for any choice of bases $B_1$, $B_2$ as above, which does not depend on these choices.
Now let $L$ be a field containing $R$, and suppose that $\pairing{1}$ and $\pairing{2}$
are non-degenerate $L$-valued $R$-bilinear pairings on $N_1$, respectively $N_2$,
satisfying
\begin{eqnarray}\label{eq:adjoint}
  \langle \phi n_1,n_2\rangle_2 = \langle n_1,\phi^* n_2\rangle_1
\end{eqnarray}
for all $n_1\in N_1$ and $n_2\in N_2$, in other words making the diagram
$$
\xymatrix{
  L\otimes_R N_1 \ar[r]^{\phi}\ar[d]_{\pairing{1}} & L\otimes_R N_2\ar[d]^{\pairing{2}}\\
  \Hom(N_1,L)\ar[r]_{\phi^*} & \Hom(N_2,L),
}
$$
commute.
Then one has
\begin{eqnarray}\label{eq:regquo}
  \cD_{B_1,B_2,\phi,\phi^*} = \frac{\det_{B_1}\pairing{1}}{\det_{B_2}\pairing{2}}.
\end{eqnarray}
This basic observation has the following two important consequences.
\begin{enumerate}[leftmargin=*,label=(O\arabic*)]
  \item\label{obs:indeppairing} The right hand side of equation \eqref{eq:regquo} is independent of the
    pairings, and determines a well-defined invariant in $Q^\times/(R^\times)^2$
    (rather than just in $L^\times/(R^\times)^2$), when the determinants are
    evaluated with respect to any bases on $N_1$ and $N_2$. More precisely, if
    $\pairing{1}'$ and $\pairing{2}'$ are also non-degenerate $R$-bilinear pairings
    on $N_1$, respectively $N_2$, with values in a field containing $R$, with respect
    to which the maps $\phi$ and $\phi^*$ are adjoint as in equation \eqref{eq:adjoint},
    then we have
    $$\left(\frac{\det_{B_1}\pairing{1}}{\det_{B_2}\pairing{2}}\right)(R^\times)^2=
    \left(\frac{\det_{B_1}\pairing{1}'}{\det_{B_2}\pairing{2}'}\right)(R^\times)^2\in Q^\times/(R^\times)^2.$$
  \item\label{obs:indepmaps} The expression $\cD_{\phi,\phi^*}$ is independent of the pair $\phi$, $\phi^*$
    of adjoint maps with respect to the pairings $\pairing{1}$ and $\pairing{2}$.
    More precisely, if $\psi\colon N_1\to N_2$ and $\psi^*\colon N_2\to N_1$
    are $R$-linear embeddings satisfying adjointness property \eqref{eq:adjoint}
    with $\phi$ replaced by $\psi$, then
    $\cD_{\phi,\phi^*}\equiv \cD_{\psi,\psi^*}\mod{(R^\times)^2}$.
\end{enumerate}
\noindent
These consequences are particularly useful in situations in which one has
a supply of pairs of maps as above that are adjoint with respect to fixed pairings
and/or a supply of pairings that make a given pair of maps adjoint. One such
situation is provided by the formalism of Brauer relations and regulator constants.
The rest of the section is devoted to recalling that formalism, reinterpreting it
in Hecke algebra terms, and applying the above observations in that context.

\subsection{Brauer relations and regulator constants}\label{sec:finRegConst}
Let $G$ be a finite group. If~$H$ is a subgroup and $V$ is an $R[G]$-module,
then we identify $\Hom_{R[G]}(R[G/H],V)$ with the group $V^H$ of $H$-fixed
points in $V$ via $f\mapsto f(1\cdot H)$. In particular if $H$ and~$H'$
are two subgroups of $G$, then one has isomorphisms
$$
\Hom_{R[G]}(R[G/H'],R[G/H])\cong R[G/H]^{H'} \cong R[H'\backslash G/H],
$$
where the inverses of the last and the first isomorphism are given by
\begin{eqnarray}\label{eq:HeckeToHom}
  H'gH \mapsto \sum_{u\in H'/(H'\cap {}^gH)}ugH\quad\text{and}\quad gH \mapsto (\alpha\colon hH'\mapsto hgH),
\end{eqnarray}
respectively. For $g\in G$, we will denote the element of $\Hom_{R[G]}(R[G/H'],R[G/H])$
corresponding to the double coset $H'gH$ by $T_{H'gH}$.
Composition of homomorphisms defines a product on $R[H\backslash G/H]$. This
$R$-algebra is called a \emph{Hecke algebra}, and its elements are called
\emph{Hecke operators}.
There is an $R$-linear map
\begin{eqnarray*}
  R[H'\backslash G/H] & \to & R[H\backslash G/H'],\\
  T & \mapsto & T^*,
\end{eqnarray*}
defined by $(T_{H'gH})^* = T_{Hg^{-1}H'}$.
Moreover, for every module~$V$ over the group ring~$R[G]$, there is a map
\begin{eqnarray}\label{eq:HomN1N2}
  \Hom_{R[G]}(R[G/H'],R[G/H]) & \to & \Hom_{R}(V^H,V^{H'})\nonumber\\
  \alpha & \mapsto & \phi_\alpha = (f\mapsto f\circ \alpha).
\end{eqnarray}
Slightly abusing notation, we will sometimes denote the image of an element $T$
of $\Hom_{R[G]}(R[G/H'],R[G/H])$
under this map by $T\colon V^H\to V^{H'}$.

Now let $V$ be a finitely generated $R[G]$-module that is free over $R$,
and let $S_1=\bigsqcup_i G/H_i$ and
$S_2=\bigsqcup_j G/H_j'$ be finite $G$-sets, where $H_i$ and $H_j'$ are
subgroups of $G$.
Define
$$
V^{S_1}=\Hom_{R[G]}(R[S_1],V)\cong \bigoplus_i V^{H_i},\;\;\; V^{S_2}=\Hom_{R[G]}(R[S_2],V)\cong \bigoplus_j V^{H_j'},
$$
where $R[S_i]$ are the $R$-linear permutation modules attached to $S_i$ for $i\in \{1,2\}$.
Then similarly to \eqref{eq:HomN1N2}, an element $T\in \Hom_{R[G]}(R[S_2],R[S_1])$
induces an $R$-homomorphism $T\colon V^{S_1}\to V^{S_2}$.

We say that the pair $\Theta=(S_1,S_2)$ is a \emph{Brauer relation} if there is
an isogeny of $R$-modules in $\Hom_{R[G]}(R[S_2],R[S_1])$.

Assume that $\Theta=(S_1,S_2)$ is a Brauer relation.
If $\pairing{}$ is a non-degenerate $R$-bilinear $G$-invariant pairing on $V$
with values in a field $L$ containing $R$,
and $U$ is a subgroup of $G$, then let $\pairing{U}$ be the pairing on $V^U$
defined by $\langle v_1,v_2\rangle_U = \frac{1}{\#U}\langle v_1,v_2\rangle$
for $v_1$, $v_2\in V^U$. It is not hard to see that this pairing is
non-degenerate on $V^U$. Thus, such a pairing $\pairing{}$ induces an
$R$-bilinear non-degenerate pairing $\pairing{1}$ on $V^{S_1}$ whose restriction to
each direct summand $V^{H_i}$ is given by $\pairing{H_i}$ and that makes the
distinct direct summands orthogonal, and analogously there is an induced
non-degenerate pairing $\pairing{2}$ on $V^{S_2}$. The \emph{regulator constant}
of $V$ with respect to the Brauer relation $\Theta$ is defined as
$$
\cC_{\Theta}(V) = \frac{\det(\pairing{1} | V^{S_1})}{\det(\pairing{2} | V^{S_2})}
  \in L^\times/(R^\times)^2.
$$

\begin{theorem}[{\cite[Theorem 2.17]{tamroot}}]\label{thm:Dokchitsers}
  The value of $\cC_{\Theta}(V)$ does not depend on $\pairing{}$, i.e.
  if $\pairing{}$ and $\pairing{}'$ are non-degenerate $G$-invariant $R$-bilinear
  pairings on $V$, inducing pairings $\pairing{j}$ and $\pairing{j}'$,
  respectively, on $V^{S_j}$ for $j\in \{1,2\}$ as above, then 
  $$
\frac{\det(\pairing{1}' | V^{S_1})}{\det(\pairing{2}' | V^{S_2})}\equiv \frac{\det(\pairing{1} | V^{S_1})}{\det(\pairing{2} | V^{S_2})}
\mod{(R^\times)^2}.
  $$
\end{theorem}

We now give a new proof of Theorem \ref{thm:Dokchitsers}, using the
observations of Section~\ref{sec:adjointpairs}.
\begin{proof}[Proof of Theorem \ref{thm:Dokchitsers}]
%
%
%
We claim that if $H$ and $H'$ are subgroups of $G$, and $\pairing{}$ is a non-degenerate
$G$-invariant $R$-bilinear pairing on $V$, then the adjoint of
$$
T_{H'gH}\colon V^H\to V^{H'}
$$
with respect to the induced pairings $\pairing{H}$ and $\pairing{H'}$
is $T_{Hg^{-1}H'}\colon V^{H'}\to V^H$. Indeed, an elementary calculation shows that there is a
well-defined function $H'gH\to H'/(H'\cap {}^gH)$ satisfying $ugh\mapsto u(H'\cap {}^gH)$
for all $u\in H'$ and $h\in H$, all of whose fibres have cardinality $\#H$,
whence we deduce that for every $v\in V^H$ and $v'\in V^{H'}$ we have
  \begin{eqnarray*}
    \frac{1}{\#{H'}}\langle T_{{H'}gH}v,v'\rangle & = &
    \frac{1}{\#{H'}}\sum_{u\in {H'}/({H'}\cap {}^gH)}\langle ugv,v'\rangle\\
      & = & \frac{1}{\#{H'}\cdot\#H}\sum_{x\in {H'}gH} \langle xv,v'\rangle\\
      & = & \frac{1}{\#{H'}\cdot\#H}\sum_{y\in Hg^{-1}{H'}}\langle v,yv'\rangle\\
      & = & \frac{1}{\#H}\sum_{h\in H/(H\cap {}^{g^{-1}}{H'})}\langle v,hg^{-1}v'\rangle\\
      & = & \frac{1}{\#H}\langle v,T_{Hg^{-1}{H'}}v'\rangle,
  \end{eqnarray*}
  as claimed.

  In particular, it follows that if $\pairing{}$ and $\pairing{}'$ are
  non-degenerate $G$-invariant $R$-bilinear pairings on $V$, and
  $T\in \Hom_{R[G]}(R[S_2],R[S_1])$ is an isogeny, then the adjoint
  of the induced isogeny $T\colon V^{S_1}\to V^{S_2}$
  with respect to $\pairing{1}$ and $\pairing{2}$ is the same as with respect
  to $\pairing{1}'$ and $\pairing{2}'$, namely equal to $T^*$ in both cases.
  Thus, $\cC_{\Theta}(V) = \cD_{T,T^*}$,
  and in particular is the same when computed with respect to $\pairing{}$
  as with respect to $\pairing{}'$.
\end{proof}
\section{Algebras graded by finite abelian groups}\label{sec:gradedAlg}

In this section we continue working in a purely algebraic setting, investigating graded
modules over certain rings graded by a finite abelian group.

The main result of Section
\ref{sec:linkage} is Proposition \ref{prop:littleExercise}, in which we give a sufficient criterion
for existence of an isomorphism between different homogeneous components of such a graded
module. In the geometric setting this isomorphism will realise an isospectrality.
To prove Proposition \ref{prop:littleExercise} we first construct paths of isomorphisms
for the simple subquotients, using Lemma \ref{lem:littleExercise}, and then piece them
together using d\'evissage, Lemma \ref{lem:invertibleSwap}.

In Section \ref{sec:polarisations} we investigate certain pairings on our graded
modules, which we call polarisations.
The main result is Proposition \ref{prop:polarisation}, a necessary and sufficient criterion
for existence of polarisations, which in particular
implies Proposition \ref{prop:IntroPolarisation} from the Introduction.
An important consequence of this criterion is Proposition \ref{prop:descend-polarisation},
which says that the existence of a polarisation with values in a field extension
implies the existence of one with values in the base field. This is an important
ingredient in our proof of rationality of quotients of regulators in the geometric setting.

In Section \ref{sec:gradedregconst} we apply the formalism of adjoint pairs from Section \ref{sec:adjointpairs}
to define our notion of regulator constants and to prove in Proposition \ref{prop:indepT}
that it is independent of any choices involved. These invariants become quotients of regulators
in the geometric setting. They enjoy some natural properties, such as rationality (Corollary \ref{cor:rational}),
triviality (Lemma \ref{lem:triv}), good behaviour with respect to extension of scalars
(Lemma \ref{lem:extendring}), and with respect to various decompositions (Lemma \ref{lem:regcstsum},
Proposition \ref{prop:obviousProd}, and Proposition \ref{prop:HeckeProd}).

The following notation and assumptions will remain in place for the rest of
the section. Let $R$ be a commmutative Noetherian domain
and let $Q$ be its field of fractions. 
An \emph{$R$-algebra} is a ring
$\heckebig$ equipped with a ring homomorphism from $R$ to the centre of $\heckebig$.
If $C$ is a group and $\heckebig$ is an $R$-algebra, then a \emph{grading of~$\heckebig$
by $C$} is a collection of $R$-submodules $\heckebig_c\subset \heckebig$ indexed by
$c\in C$ such that there is a direct sum decomposition of $R$-modules
$\heckebig=\bigoplus_{c\in C}\heckebig_c$ and such that for all $c$, $c'\in C$ one has
$\heckebig_{c}\heckebig_{c'}\subset \heckebig_{cc'}$. An $R$-algebra that is equipped with a grading
by a group $C$ is said to be \emph{graded by $C$}, and the $R$-submodules
$\heckebig_c$ of such an $R$-algebra~$\heckebig$ are referred to as the \emph{homogeneous components}
of $\heckebig$.

For the rest of the section, let $\hecke$ be a commutative $R$-algebra that
is graded by a finite abelian group $C$. The trivial homogeneous component $\hecke_1$
is an $R$-subalgebra of $\hecke$, and for every $c\in C$, the $R$-submodule $\hecke_c$
of $\hecke$ is an $\hecke_1$-module. Next, let $W$ be a group equipped with a group
homomorphism $\nu\colon W\to C$. By setting $W_c=\nu^{-1}(c)$
for $c\in C$, we obtain a $C$-grading on the group
algebra $\hecke[W]$. Let $C_{\iso}=C/\nu(W)$. Thus, $\Cisodual$
is canonically identified with the group of all $\chi \in \Cdual$ that are trivial on $\nu(W)$.
Let $\heckebig$ be a quotient of $\hecke[W]$ by a homogeneous ideal, equivalently
a $C$-graded $R$-algebra together with a grading-preserving surjection from $\hecke[W]$, and
assume that $\heckebig$ is finitely generated as an $R$-module.
In particular, every $\heckebig$-module is also a $\hecke[W]$-module.

For us, a \emph{graded $\heckebig$-module} is a finitely generated
$\heckebig$-module $M$ equipped with a collection of $R$-submodules $M_c\subset M$
indexed by $c\in C$ such that one has a direct sum decomposition
of $\hecke_1$-modules $M=\bigoplus_{c\in C}M_c$ and such that for all $c$,
$c'\in C$ one has $\heckebig_{c}M_{c'}\subset M_{cc'}$.

If an element of a ring is invertible modulo the radical, then it is invertible.
A ring is called \emph{semilocal} if its quotient by the radical
is a semisimple ring. A commutative ring is semilocal if and only if it has
only finitely many maximal ideals. We refer to \cite{Lam} for these
definitions and basic facts from ring theory.



\subsection{Linked modules}\label{sec:linkage}
%
\begin{definition}\label{def:linked}
  Let $M$ be a graded $\heckebig$-module, and
  let~$c,c'\in C$. If $L$ is a ring containing $R$, we say that~$M_{c}$
  is \emph{$L$-linked} to~$M_{c'}$ by~$T\in \heckebig_{c'c^{-1}}$ if the
  induced map $T\colon (M_{c})_L\to (M_{c'})_L$ is an isomorphism.
  We say that~$M_{c}$ is \emph{$L$-linked} to~$M_{c'}$ if it is $L$-linked to
  it by some~$T\in \heckebig_{c'c^{-1}}$.
  We say that~$M_{c}$ is \emph{linked} to~$M_{c'}$ if for some ring $L$
  containing $R$, the module $M_{c}$ is $L$-linked to $M_{c'}$.
\end{definition}


\begin{lemma}\label{lem:invertibleSwapField}
  Suppose that $\hecke_1$ is a field. Let $c\in C$, and let $M$ be a non-zero
  $\heckebig$-module. Then the following are equivalent:
  \begin{enumerate}[leftmargin=*]
    \item\label{item:invertible} there exists an element in $\heckebig_c$ that is invertible in $\heckebig$;
    \item\label{item:nonzerosimple} there exists a simple subquotient $N$ of $M$ and an
      element of $\heckebig_c$ that does not annihilate $N$.
  \end{enumerate}
  If, moreover, the image of~$\hecke$ in~$\heckebig$ is reduced, then the conditions are
  equivalent to the following:
  \begin{enumerate}[leftmargin=*]
  \setcounter{enumi}{2}
    \item\label{item:nonzero} the homogeneous component $\heckebig_c$ does not annihilate $M$.
  \end{enumerate}
\end{lemma}
\begin{proof}
  The forward implications are trivial, without the reduced assumption.

  We now prove that \ref{item:nonzero} implies \ref{item:nonzerosimple}.
  The homogeneous component $\heckebig_c$ is generated by the images of
  $\hecke_{d}\cdot W_{d^{-1}c}$, as $d$ runs over $C$. The hypothesis
  therefore implies that there exist $d\in C$, $T\in \hecke_{d}$, and $w\in W_{d^{-1}c}$
  such that $Tw$ does not annihilate~$M$. Since $w$ is invertible, this is equivalent
  to $T$ not annihilating $M$. Since the image of~$\hecke$ in~$\heckebig$ is reduced
  and finite-dimensional over the field~$\hecke_1$, it is a product of
  fields, so its quotient~$\hecke/\Ann_\hecke M$ is reduced.
  The element~$T$ is therefore not nilpotent in~$\End_{\hecke_1}(M)$.
  Since $M$ has finite length,
  this implies that $T$ cannot annihilate every simple subquotient of $M$,
  and therefore the same is true for $Tw\in \heckebig_c$.

  Finally we prove that \ref{item:nonzerosimple} implies \ref{item:invertible}.
  By the same argument as in the previous paragraph, the assumption implies that
  there exist $d\in C$, $T\in \hecke_{d}$, and $w\in W_{d^{-1}c}$ such that
  $Tw$ does not annihilate a subquotient $N$, equivalently that $T$ does not annihilate $N$,
  equivalently that $T$ acts invertibly on $N$. Let $n$ be the order of $d\in C$.
  Then $T^n\in \hecke_1$ also acts invertibly on $N$. Since $\hecke_1$ is a field,
  this implies that $T^n$ is invertible, therefore so is $Tw$, whose image in $\heckebig$
  belongs to~$\heckebig_c$.
%
%
%
\end{proof}

\begin{lemma}\label{lem:invertibleSwap}
Suppose that $R$ is semilocal. Let $c\in C$,
and let $M$ be an~$\heckebig$-module such that for every simple subquotient
$N$ of~$M$, there exists an element of $\heckebig_c$ that acts non-trivially
(equivalently invertibly) on $N$. Then there exists an element of~$\heckebig_c$
that acts invertibly on $M$.
\end{lemma}
\begin{proof}
  By replacing $\heckebig$ by $\heckebig/(\Ann_{\hecke_1} M)\cdot \heckebig$,
  we may assume that $\hecke_1$ acts faithfully on $M$.
  Let $\fm$ be a maximal ideal of $\hecke_1$. Then by faithfulness,
  it is the annihilator of a simple $\hecke_1$-subquotient
  $M'$ of $M$. By Lemma \ref{lem:invertibleSwapField}, applied to $\hecke_1/\fm$
  in place of~$\hecke_1$ and to $M'$ in place of $M$, there exists $u_{\fm}\in \heckebig_c$ that acts invertibly on
  $M/\fm M$. By the Chinese Remainder Theorem, there exists $u\in \heckebig_c$
  that acts invertibly on $M/\Jac(\hecke_1)M$, where recall that
  $\Jac(\hecke_1)=\prod_{\fm\in \MaxSpec(\hecke_1)}\fm$
  denotes the radical of~$\hecke_1$. By
  \cite[Corollary 5.9]{Lam}, $\Jac(\hecke_1)\cdot \heckebig$ is contained in the radical
  of~$\heckebig$, so~$u$ acts invertibly on~$M$.
\end{proof}
%
%

\begin{lemma}\label{lem:subquosub}
Suppose that $R$ is semilocal. Let $M$ be a $\hecke$-module.
Then for every simple $\hecke$-subquotient $E$ of $M$
there exists $\fp\in \MaxSpec(R)$ such that $E$ is isomorphic to a
$\hecke$-submodule of $R/\fp\otimes_R M$.
\end{lemma}
\begin{proof}
  By replacing $\hecke$ by $\hecke/\Ann_{\hecke} M\subset \End M$, we may assume
  that $\hecke$ is finitely generated over $R$.
  Let $E$ be a simple subquotient of $M$. Then its annihilator is a maximal
  ideal $\fP$ of $\hecke$. By \cite[Corollary 4.17]{Eisenbud} the intersection
  $\fp=\fP \cap R$ is a maximal ideal of $R$, and $E$ is a subquotient
  of the localisation $(R/\fp\otimes_R M)_{\fP}$. Now, the quotient~$R/\fp$ is
  a field, so that $\hecke/\fp\hecke$ is Artinian. By \cite[Corollary 2.16]{Eisenbud},
  the localisation $(\hecke/\fp\hecke)_{\fP}$ is local and is a direct factor of
  $\hecke/\fp\hecke$, so that $(R/\fp\otimes_R M)_{\fP}$ is a direct $\hecke$-summand of
  $R/\fp\otimes_R M$. Since this localisation is Artinian, it contains
  a simple submodule. But since $(\hecke/\fp\hecke)_{\fP}$ is local, all simple
  $(\hecke/\fp\hecke)_{\fP}$-modules are pairwise isomorphic, and in particular
  $E$ is isomorphic to a simple $\hecke$-submodule of
  $(R/\fp\otimes_R M)_{\fP}\subset R/\fp\otimes_R M$, which proves the result.
\end{proof}

\begin{lemma}\label{lem:littleExercise}
  Let $c\in C$ be non-trivial, and for every $\chi\in \Cdual$ with $\chi(c) \neq 1$,
  let $c_\chi\in C$ be such that $\chi(c_\chi)\neq 1$. Then
  $c$ is contained in the group generated by the elements $c_\chi$.
\end{lemma}
\begin{proof}
  Suppose that $X$ is the set of all $\chi\in \Cdual$ satisfying $\chi(c)\neq 1$.
  Let $U$ be the subgroup of $C$ generated by $c_{\chi}$ for all $\chi\in X$.
  If $c$ is  not in $U$,
  then there exists $\chi'\in \Cdual$ such that $\chi'(c)\neq 1$
  and $\chi'|_U=1$, in particular such that we have $\chi'\not\in X$. Therefore, $c$ is in $U$,
  as claimed.
\end{proof}

\begin{proposition}\label{prop:littleExercise}
  Suppose that $R$ is semilocal. Let $c\in C$, and let
  $M$ be a graded $\heckebig$-module. Suppose that for all the finitely many
  $\fp\in \MaxSpec(R)$, for all simple (ungraded) $\hecke$-submodules $E$
  of~$R/\fp\otimes_R M$, and for all $\chi\in \Cisodual$ with
  $\chi(c)\neq 1$, there exist $c_{\chi,E}\in C$ with
  $\chi(c_{\chi,E})\neq 1$ and $T_{\chi,E}\in \hecke_{c_{\chi,E}}$ that
  does not annihilate~$E$. Then
  there exists an element of $\heckebig_{c}$ that acts invertibly
    on $M$.

\end{proposition}
\begin{proof}
  Let $N$ be a simple $\heckebig$-subquotient of $M$.
  We claim that the hypothesis implies that for all $\chi\in \Cdual$ satisfying $\chi(c)\neq 1$
  there exist $c_{\chi,N}\in C$ with $\chi(c_{\chi,N})\neq 1$ and
  $S_{\chi,N}\in \heckebig_{c_{\chi,N}}$ that does not annihilate $N$.
  Indeed, for $\chi\in \Cdual$ that do not vanish on~$\nu(W)$, we may take
  $S_{\chi,N}$ to be the image of any $w\in W$ for which $\chi(\nu(w))\neq 1$.
  Suppose instead that $\chi \in \Cisodual$ satisfies $\chi(c)\neq 1$.
  Since the image of $\hecke$ in $\heckebig$ is central and $N$ is finite over $R$,
  the $\hecke$-module
  $N$ is isomorphic to $E^r$ for a simple $\hecke$-module $E$ and some $r\in \Z_{>0}$.
  By Lemma~\ref{lem:subquosub}, the module $E$ is isomorphic to a simple submodule of
  $R/\fp\otimes_R M$ for some $\fp\in \MaxSpec(R)$. We may then take
  $S_{\chi,N}=T_{\chi,E}$.
  By Lemma~\ref{lem:littleExercise}, there exist $\alpha_{\chi,N}\in \Z$ such that
  $c=\prod_{\chi}c_{\chi,N}^{\alpha_{\chi,N}}$, so that
  $\prod_{\chi}S_{\chi,N}^{\alpha_{\chi,N}}$ belongs to $\heckebig_{c}$. This product
  clearly acts invertibly on $N$. The conclusion follows from Lemma~\ref{lem:invertibleSwap}.
\end{proof}

\begin{corollary}\label{cor:linked}
  Let~$c\in C$ and let~$M$ be a graded~$\heckebig$-module.
  Suppose that for all simple (ungraded) $\hecke$-submodules $E$
  of~$M_Q$, and for all $\chi\in \Cisodual$ with
  $\chi(c)\neq 1$, there exist $c_{\chi,E}\in C$ with
  $\chi(c_{\chi,E})\neq 1$ and $T_{\chi,E}\in \hecke_{c_{\chi,E}}$ that
  does not annihilate~$E$.
  Then for every $b\in C$, the homogeneous component~$M_{b}$ is linked to $M_{cb}$.
\end{corollary}
\begin{proof}
  Apply Proposition~\ref{prop:littleExercise} with $R=Q$.
\end{proof}

\subsection{Polarisations}\label{sec:polarisations}
%
\begin{definition}
An \emph{involution} on $\hecke$ is an
$R$-module homomorphism $\iota\colon \hecke\to \hecke$
satisfying the following properties:
\begin{itemize}
  \item $\iota^2=\id$,
  \item for all $x$, $y\in \hecke$ one has $\iota(xy) = \iota(y)\iota(x)$,
  \item for all $c\in C$ one has $\iota(\hecke_c)=\hecke_{c^{-1}}$.
\end{itemize}
In particular, every involution on $\hecke$ stabilises~$\hecke_1$.
Extend any involution $\iota$ on $\hecke$ to an involution
on $\hecke[W]$ by defining $\iota(w)=w^{-1}$ for all $w \in W$, and
assume that this involution descends to $\heckebig$.
Whenever~$M$ is an~$\heckebig$-module (resp.~$\hecke_1$-module), we give~$M^*=\Hom_R(M,R)$ the
structure of an~$\heckebig$-module (resp.~$\hecke_1$-module), called the \emph{dual of $M$},
via~$\iota$, i.e.~$(T\phi)(m)=\phi(\iota(T)m)$ for~$\phi\in M^*, m\in M$
and~$T\in \heckebig$ (resp.~$T\in \hecke_1$). We say that an~$\heckebig$-module
(resp.~$\hecke_1$-module) $M$ is \emph{self-dual} if $M^*$ is isomorphic to $M$.
\end{definition}

\begin{assumption}\label{ass:invol}
  For the rest of the section, assume that the image of~$\hecke$ in~$\heckebig$ is reduced.
  Further, assume that~$\hecke$ is equipped with an involution $\iota$, and
  that $\iota$, when extended to~$\hecke W$ as above, descends to $\heckebig$.
\end{assumption}

\begin{definition}\label{def:pol}
  Let $M$ be a graded $\heckebig$-module that is $R$-torsion-free.
  If $L$ is a field containing $Q$, then a \emph{polarisation} on~$M$
  over~$L$ is a non-degenerate bilinear pairing $\pairing{}\colon
  M\otimes_R M \to L$ such that for $c\neq c'\in C$, the module~$M_{c}$ is orthogonal
  to $M_{c'}$, and such that for every $T\in \heckebig$, the element $T$ is adjoint to
  $\iota(T)$. If~$M$ is equipped with a polarisation over $L$, then we say
  that~$M$ is \emph{polarised} over $L$, and if~$M$ admits a polarisation
  over $L$, then we say that~$M$ is \emph{polarisable} over $L$.
\end{definition}

\begin{remark}\label{rmrk:Wadjoint}
  A pairing $\pairing{}\colon M\otimes_R M \to L$ is a polarisation
  for~$\heckebig$ if and only if it is a polarisation for $\hecke$
  that is $W$-invariant, i.e. such that for all $w\in W$ and all $m, m'\in M$
  we have $\langle wm,wm'\rangle=\langle m,m'\rangle$.
\end{remark}


\begin{remark}
  If~$M$ is as in Definition \ref{def:pol}, then 
  $M^*=\Hom_R(M,R)$ is a graded $\heckebig$-module, where for every $c\in C$,
  the homogeneous component $(M^*)_c$ consists of the homomorphisms $M\to R$ that
  factor through $M_{c^{-1}}$.
  A polarisation on~$M$ over a field $L$ is the same as an isomorphism of
  (not graded) $\heckebig_L$-modules $M_L\to (M^*)_L$ that sends
  $M_c$ to $(M^*)_{c^{-1}}$.
\end{remark}

\begin{proposition}\label{prop:polarisation}
  Let $M$ be a graded $\heckebig$-module that is $R$-torsion-free,
  and let $L$ be a field containing $Q$.
  Then~$M$ is polarisable over~$L$ if and only if for every~$c\in C$,
  $(M_c)_L$ is a self-dual $(\heckebig_1)_L$-module.
\end{proposition}
\begin{proof}
  The ``only if'' direction is clear, so we prove the other direction.

  We may, without loss of generality, replace $R$ by $L$, replace all the algebras
  and modules by their tensor product with $L$ over $R$, and replace~$\hecke$
  by its image in~$\heckebig$. Then $\hecke$, being a reduced
  commutative algebra over a field by Assumption \ref{ass:invol}, is a product of fields. In particular, every
  quotient of $\hecke$ is also reduced.
  The algebra~$\hecke_1$ is a product of factors $B$ that are each
  either a field that is stabilised by $\iota$, or a product $B=A_1\times A_2$ of fields
  with $\iota$ inducing an isomorphism from $A_1$ to $A_2$ and from $A_2$ to $A_1$.
  Accordingly the tensor products $B\otimes_{\hecke_1} \bullet$ for all such factors $B$
  induce a direct product decomposition on the
  $\hecke_1$-algebras $\heckebig_1$, $\hecke$, and $\heckebig$, and a direct
  sum decomposition on the modules $M$ and $M_c$ for all $c\in C$. Clearly, $M$ is
  polarisable if and only if for all $B$ as above the module $B\otimes_{\hecke_1} M$ is,
  and similarly, for a given $c\in C$, the $\heckebig_1$-module $M_c$ is self-dual
  if and only if for all $B$, the $B\otimes_{\hecke_1} \heckebig_1$-module
  $B\otimes_{\hecke_1} M_c$ is. Moreover, $\hecke$ being reduced implies that
  so is $B\otimes_{\hecke_1} \hecke$ for all $B$ as above.
  Thus, we assume, without loss of generality, that
  $\hecke_1$ is either a field stabilised by $\iota$ or isomorphic to $A_1\times A_2$
  as just described.

  Let 
    $$
      \tilde{C}=\{c\in C: \heckebig_cM\neq 0\}.
    $$
  For an arbitrary $T\in \heckebig$, the non-degeneracy of the pairing on $M$ implies
  that $TM\neq 0$ if and only $\iota(T)M\neq 0$. It follows that if $c\in C$, and
  $\hecke_1=A_1\times A_2$, then $(A_1\otimes_{\hecke_1}\heckebig)_c$
  does not annihilate $M$ if and only if $(A_2\otimes_{\hecke_1}\heckebig)_c$ does not.
  Therefore, Lemma \ref{lem:invertibleSwapField} implies that we have
  $$
    \tilde{C}= \{c\in C: \heckebig_c\text{ contains an invertible element}\}.
  $$
  In particular, $\tilde{C}$ is a subgroup of $C$.

  Now, suppose that for every $c$, the $\heckebig_1$-module $(M_c)_L$ is self-dual.
  Let $D$ be a complete set of coset representatives for $C/\tilde{C}$,
  and for each $d\in D$, let $\langle \cdot , \cdot \rangle_d$ be a non-degenerate $L$-valued
  $L$-bilinear pairing on $M_d$ such that for all $T\in \heckebig_1$, the
  operator $\iota(T)$ is the adjoint of $T$. For every $c\in \tilde{C}$, fix
  an invertible element $T_c\in \heckebig_c$. For $c\in \tilde{C}$ and $d\in D$,
  define a pairing $\langle \cdot,\cdot\rangle_{cd}$ on $M_{cd}$ by setting, for every $m$, $m'\in M_{cd}$,
  $$
    \langle m,m'\rangle_{cd} = \langle T_c^{-1}m,\iota(T_c)m'\rangle_d.
  $$
  Let $\langle \cdot,\cdot\rangle$ be the unique pairing on $M$ that restricts to $\langle \cdot,\cdot\rangle_{c}$
  for all $c\in C$  and that makes the different homogeneous components pairwise orthogonal.
  It is clear that each $\langle \cdot,\cdot\rangle_c$ is non-degenerate, therefore so is $\langle\cdot ,\cdot\rangle$.
  We claim that for all $T\in \heckebig$, the adjoint of $T$ with respect to this pairing is
  $\iota(T)$. It suffices to prove this for homogeneous elements $T$. Let $c,c'\in \tilde{C}$,
  $d\in D$, $T\in \heckebig_c$, $m\in M_{c^{-1}c'd}$ and $m'\in M_{c'd}$. Then we have
  \begin{eqnarray*}
    \langle Tm, m'\rangle_{c'd} & = & \langle T_{c'}^{-1}Tm,\iota(T_{c'})m'\rangle_d\\
                                & = & \langle \underbrace{(T_{c'}^{-1}TT_{c'c^{-1}})}_{\in \heckebig_1}T_{c'c^{-1}}^{-1}m,\iota(T_{c'})m'\rangle_d\\
                                & = & \langle  T_{c'c^{-1}}^{-1}m,\iota(T_{c'}^{-1}TT_{c'c^{-1}})\iota(T_{c'})m'\rangle_d\\
                                & = & \langle  T_{c'c^{-1}}^{-1}m,\iota(T_{c'c^{-1}})\iota(T)m'\rangle_d\\
                                & = & \langle m,\iota(T)m'\rangle_{c^{-1}c'd},
  \end{eqnarray*}
  as required. This shows that $\langle\cdot,\cdot\rangle$ is a polarisation on $M$,
  and completes the proof.
\end{proof}

\begin{remark}
For~$M$ as in Definition \ref{def:pol} to be polarisable over a
field~$L$, it is not sufficient that~$M_L$ be a self-dual $(\heckebig_1)_L$-module. Indeed,
let $W=1$, $\heckebig_1=\hecke_1 = Q\times Q$, with $\iota$ swapping the two factors,
$\hecke = \heckebig = \hecke_1$ with
trivial grading by a group $C=\{1,c\}$ of order $2$, and let~$M_1=Q$ with~$\hecke$ acting
via projection onto the first factor $Q$, and~$M_c=Q$ with~$\hecke$ acting via 
projection onto the second factor. Then $M_1$ and $M_c$ are dual to each other,
so that $M=M_1\oplus M_c$ is a self-dual $\hecke_1$-module. However, $M_1$ is not
self-dual, so $M$ 
is not polarisable over~$Q$.
\end{remark}

\begin{cor}\label{cor:T1field}
  Assume that $W=1$ and that either $\hecke_1$ is a field or $\iota$ is trivial.
  Let $M$ be a graded $\hecke$-module that is $R$-torsion-free.
  Then~$M$ is polarisable over~$Q$.
\end{cor}
\begin{proof}
  Without loss of generality, assume that~$R=Q$ and replace~$\hecke$ by its
  image in~$\heckebig$.
  Under the assumptions, $\hecke_1$ is a product of fields, each preserved
  by $\iota$, and~$M$ is polarisable over $Q$ if and only if it is after taking the
  tensor product with each of these factors. Thus, it suffices to assume that $\hecke_1$
  is a field.
  In that case, for every $c\in C$ the $\hecke_1$-vector spaces $(M_c)_{Q}$ and $(M^*_c)_{Q}$ have
  the same dimension, and are therefore isomorphic. The result follows from Proposition \ref{prop:polarisation}.
\end{proof}

\begin{prop}\label{prop:descend-polarisation}
  Let~$M$ be a graded~$\heckebig$-module that is $R$-torsion-free, and
  let~$L$ be a field containing $Q$.
  Then~$M$ is polarisable over $Q$ if and only if $M$ is polarisable over $L$.
\end{prop}
\begin{proof}
  Assume that~$M$ is polarisable over~$L$. Let $c\in C$ be arbitrary.
  Then the $(\hecke_1)_L$-module $(M_c)_L$ is
  self-dual, i.e.~$(M_c)_L$ and~$(M_c^*)_L$ are
  isomorphic. By the Deuring--Noether Theorem \cite[Theorem 19.25]{Lam}
this implies that~$(M_c)_Q$ and~$(M_c^*)_Q$ are isomorphic, i.e. the $(\hecke_1)_Q$-module~$(M_c)_Q$
is self-dual. By Proposition~\ref{prop:polarisation}, the module~$M$ is polarisable over $Q$.
\end{proof}

\begin{definition}
  We will say that a graded $\heckebig$-module $M$ is \emph{polarisable}
  if it is polarisable over some field containing $Q$. In light of Proposition
  \ref{prop:descend-polarisation} this is equivalent to $M$ being polarisable over $Q$.
\end{definition}

\subsection{Regulator constants for graded modules}\label{sec:gradedregconst}
%
In this subsection we continue to make Assumption \ref{ass:invol}.

\begin{definition}\label{def:gradedregconst}
Let $M$ be a graded $\heckebig$-module that is free over~$R$.
    Let $c,c'\in C$. Suppose that $M$ is polarisable
      and that~$M_{c}$ is $L$-linked to~$M_{c'}$ by~$T\in (\heckebig_{c'c^{-1}})_L$
      for some field $L$ containing $Q$.
      Then $\iota(T)$, being the adjoint of an isomorphism with respect to a
      polarisation, induces an isomorphism~$(M_{c'})_L\to (M_{c})_L$.
      Define the \emph{regulator constant of $M$ with respect to the pair $(c,c')$} by
      $$
      \cC_{c,c'}(M)=\frac{\det(T\colon (M_{c})_L\to (M_{c'})_L)}
      {\det(\iota(T)\colon (M_{c'})_L\to (M_{c})_L)}=
      \cD_{T,\iota(T)}\in L^\times/(R^\times)^2,
      $$
      using the notation of Section~\ref{sec:adjointpairs}, where recall that the
      determinants are evaluated with respect to any bases $B_1$ of $M_{c}$ and
      $B_2$ of $M_{c'}$.
\end{definition}

\begin{proposition}\label{prop:indepT}
  For every polarisation $\langle\cdot,\cdot\rangle$ on $M$ we have
  $$
    \cC_{c,c'}(M) = \frac{\det(\langle\cdot,\cdot\rangle\, |\,M_{c})}{\det(\langle\cdot,\cdot\rangle\,|\,M_{c'})}.
  $$
  The value of $\cC_{c,c'}(M)$ does not depend on $T$, nor on the polarisation.
\end{proposition}
\begin{proof}
  Apply equation \eqref{eq:regquo} and observations
  \ref{obs:indeppairing} and \ref{obs:indepmaps} of Section \ref{sec:adjointpairs}.
\end{proof}



\begin{remark}
  Let $M$ be a graded $\heckebig$-module that is free over $R$ and polarisable over a field $L$.
  Then being  $L$-linked is an
  equivalence relation on the homogeneous components of $M$, and for
  all~$c_1,c_2,c_3\in C$ such that~$M_{c_1}$, $M_{c_2}$, and~$M_{c_3}$ are
   linked we have
  \[
    \cC_{c_1,c_2}(M) = \cC_{c_2,c_1}(M)^{-1} \text{ and } \cC_{c_1,c_3}(M) =
    \cC_{c_1,c_2}(M)\cC_{c_2,c_3}(M).
  \]
\end{remark}

\begin{cor}\label{cor:rational}
  Let~$M$ be a polarisable graded~$\heckebig$-module that is free over~$R$,
  and let~$c,c'\in C$ be such that~$M_{c}$ and~$M_{c'}$ are linked. Then we have
  \[
    \cC_{c,c'}(M)\in Q^\times/(R^\times)^2.
  \]
\end{cor}
\begin{proof}
  By Proposition~\ref{prop:descend-polarisation} there exists
  a polarisation on $M$ that takes values in~$Q$, and by Proposition
  \ref{prop:indepT} the value of $\cC_{c,c'}(M)$ does not depend on
  the polarisation.
\end{proof}


\begin{lemma}\label{lem:triv}
  Let $M$ be a graded $\heckebig$-module that is free over~$R$.
  Let $c$, $c'\in C$, and set $c=c'c^{-1}$. Suppose that there exists
  $T\in \heckebig_c$ such that $T\colon M\to M$ is surjective. Then one has
  $\cC_{c,c'}(M)\in R^{\times}/(R^\times)^2$.
\end{lemma}
\begin{proof}
%
  By Nakayama's lemma \cite[Theorem 2.4]{Matsumura}, the $\hecke$-module
  endomorphism $T$ of $M$, being surjective, is in fact an isomorphism.
  This also implies that $\iota(T)$ is an isomorphism.
  Thus $\det(T\colon M_{c}\to M_{c'})$
  and $\det(\iota(T)\colon M_{c'}\to M_{c})$, evaluated with
  respect to any $R$-bases on $M_{c}$ and $M_{c'}$, are units
  in $R$, and the result follows. 
\end{proof}

\begin{lemma}\label{lem:extendring}
  Let $R\hookrightarrow \tilde{R}$ be an embedding of commutative domains,
  let $L$ be a field containing $\tilde{R}$, and let $M$ be a graded $\heckebig$-module
  that is free over $R$ and that is polarisable over $L$.
  Then for all $c$, $c'\in C$ one has
$$
\cC_{c,c'}(M) = \cC_{c,c'}(\tilde{R}\otimes_R M) \in L^\times/(\tilde{R}^\times)^2.
$$
\end{lemma}
\begin{proof}
  This is immediate from the definitions.
\end{proof}


%
The next result is not used in the sequel, but it illustrates a connection
between Hecke eigenvalues and regulator ratios. Proposition \ref{prop:IntroHeckeField}
is a special case.
\begin{prop}\label{prop:regcst-squares}
  Let~$R=Q$ be a field, let~$\hecke_1$ be a field, take $W$ to be trivial,
  let~$M$ be a graded $\hecke$-module.
  Let $c\in C$ have order $2$, assume that $M_1$ and $M_c$ are linked, by
  $T$, say, and let $T^2=a\in \hecke_1$.
  Then $M$ is polarisable,
    and we have
  \[
    \cC_{1,c}(M) = \norm_{\hecke_1/Q}(a)^n \bmod{(Q^\times)^2},
  \]
  where $n$ is the common dimension of $M_1$ and $M_c$ over $\hecke_1$, and
  where $\norm_{\hecke_1/Q}$ denotes the $Q$-algebra norm on $\hecke_1$.
\end{prop}
\begin{proof}
  By Corollary \ref{cor:T1field}, $M$ is polarisable.
  Write the action of~$T$ on~$M_1\oplus M_{c}$ as a block-diagonal matrix $\begin{pmatrix}
      0 & A \\
      B & 0
    \end{pmatrix}$ over the
  field~$\hecke_1$ with~$A,B$ invertible. Then~$T^2 = a$ acts as
  \[
    \begin{pmatrix}
      AB & 0 \\
      0 & BA 
    \end{pmatrix}
    =
    \begin{pmatrix}
      a\cdot\Id_n & 0 \\
      0 & a\cdot\Id_n
    \end{pmatrix}.
  \]
  Since~$\iota$ is an automorphism of~$\hecke$ that is trivial
  on~$Q$, we have~$\det_Q(\iota(B))=\det_Q(B)$.
  We get~$\cC_{1,c}(M) = \det_Q(A)/\det_Q(\iota(B)) \equiv \det_Q(AB) \equiv
  \norm_{\hecke_1/Q}(a^n) \bmod{(Q^\times)^2}$.
  This proves the result.
\end{proof}


The next three results show that in many situations
regulator constants of a module can be factored as products of
regulator constants of ``simpler'' modules.

\begin{lem}\label{lem:regcstsum}
  Let~$M$ be a polarisable graded $\heckebig$-module that is free over $R$ and
  that decomposes as a direct sum~$M = \bigoplus_i M_i$ of polarisable graded
  $\heckebig$-submodules. Let~$c,c'\in C$, and assume that~$M_{c}$ and~$M_{c'}$
  are  linked.
  Then for every~$i$, $(M_i)_{c}$ and~$(M_i)_{c'}$ are
   linked, and we have
  \[
    \cC_{c,c'}(M) = \prod_i \cC_{c,c'}(M_i).
  \]
\end{lem}
\begin{proof}
  The first assertion is clear. Also, by hypothesis, we may choose a
  polarisation on~$M$ that makes all summands~$M_i$ pairwise orthogonal, whence
  the second assertion follows.
\end{proof}

\begin{prop}\label{prop:obviousProd}
  Let~$R=Q$ be a field, and let~$M$ be a polarisable graded $\heckebig$-module
  that is free over $R$.
  Let~$c,c'\in C$, and assume that~$M_{c}$ and~$M_{c'}$
  are  linked.
  Let the image of~$\hecke_1$ in $\heckebig$ decompose as a direct product $\prod_i B_i$
  of $Q$-algebras with involution,
  and for all~$i$, let~$M_i = B_i \otimes_{\hecke_1} M$. 
  Then for all~$i$, the graded~$\heckebig$-module $M_i$ is polarisable, $(M_i)_{c}$
  and~$(M_i)_{c'}$ are linked, and we have
  \[
    \cC_{c,c'}(M) = \prod_i \cC_{c,c'}(M_i) \in Q^\times/(Q^\times)^2.
  \]
\end{prop}
\begin{proof}
  By Proposition~\ref{prop:polarisation}, since~$M$ is polarisable, the~$M_i$ are
  polarisable. We have a decomposition $M=\bigoplus_i M_i$.
  The last two assertions follow from Lemma~\ref{lem:regcstsum}.
\end{proof}

\begin{prop}\label{prop:HeckeProd}
  Let~$R=\Z_p$, and let~$M$ be a polarisable graded $\heckebig$-module that is free over $R$.
  Let $\cN=\{\fn\cap \iota(\fn): \fn\in \MaxSpec(\hecke_1)\}$, and for all
  $\fn'\in \cN$ let $M_{\fn'} = (\hecke_1)_{\fn'}\otimes_{\hecke_1} M$.
  Then for all~$\fn'\in \cN$, the graded~$\heckebig$-module~$M_{\fn'}$ is
  polarisable.

  Let~$c,c'\in C$, and assume that~$M_{c}$ and~$M_{c'}$
  are  linked.
  Then for all~$\fn'\in \cN$, $(M_{\fn'})_{c}$
  and~$(M_{\fn'})_{c'}$ are  linked, and we have
  \[
    \cC_{c,c'}(M) = \prod_{\fn'\in \cN} \cC_{c,c'}(M_{\fn'})\in \Q_p^\times / (\Z_p^\times)^2.
  \]
\end{prop}
\begin{proof}
  By \cite[Corollary 7.6]{Eisenbud}, there is a
  decomposition $\hecke_1=\prod_{\fn}(\hecke_1)_{\fn}$, with the product running over the
  maximal ideals of $\hecke_1$. Accordingly, by grouping terms, there is a direct sum
  decomposition~$M = \bigoplus_{\fn'\in \cN} M_{\fn'}$.
  By Proposition~\ref{prop:polarisation}, since~$M$ is polarisable, the~$M_{\fn'}$ are
  polarisable. The last two statements follow from Lemma~\ref{lem:regcstsum}.
\end{proof}

%
%
\newpage
\section{Vign\'eras pairs of manifolds}\label{sec:Vigneras}

In the present section we will introduce the number theoretic situation
we are interested in, a general version of the Vign\'eras construction,
and apply the results of Section \ref{sec:gradedAlg}
to this situation, deriving inexplicit self-twist conditions for various
kinds of isospectralities.
%
%
We fix the following notation.

\setlength\tabcolsep{0.2em}
\setlength\arraycolsep{0.2em}
\begin{tabular}{ll}
  $F$ & a number field.\\
  $\Z_L$ & the ring of integers of a number field or $p$-adic field~$L$.\\
  $\fp$ & a maximal ideal of~$\Z_F$.\\
  $\Z_{\fp}$ & ring of integers of the completion $F_{\fp}$.\\
  $\fN$ & a non-zero ideal of $\Z_F$.\\
  $\infty$ & the set of all real places of $F$.\\
  $\Cl_F(\fN \cV)$ & the ray class group with modulus~$\fN$ times a subset~$\cV$ 
    of $\infty$\\
  $U_F(\fN \cV)$ & the group $\{u\in \Z_F^\times: u-1\in \fN\text{ and } v(u)>0\text{ for all }v\in\cV\}$.\\
  $F_\R$ & $\R\otimes_{\Q}F$.\\
  $\adel_{F,f}$ & the ring~$(\prod_p \Z_p)\otimes F$ of finite ad\`eles of $F$.\\
  $\adel_F$ & the ring~$F_\R \times \adel_{F,f}$ of ad\`eles of $F$.\\
  $D$ & a division algebra over $F$ of degree $d\geq 2$ that is not totally definite.\\
  $\Hamil$ & the $\R$-algebra of Hamilton quaternions.\\
  $\cV_{\R}(D)$, & the sets of infinite places of $F$ that are real and split in $D$,\\
  $\cV_{\CC}(D)$,                     & respectively complex (and necessarily split in $D$),\\
  $\cV_{\Hamil}(D)$                     & respectively ramified in $D$.\\
  $\cO$ & a maximal order in $D$.\\
  $\G$ & the algebraic group corresponding to $D^\times$, i.e.
  representing\\
       & the functor $(\bullet \otimes_F D)^\times$ on the category of $F$-algebras.\\
  $G_\infty$ & $\G(F_{\R})$.\\
  $K_\infty$ & a maximal compact subgroup of $G_\infty$.\\
  $Z_{\infty}$ & the centre of $G_{\infty}$.\\
  $\Lambda_{\fp}$, $d_{\fp}$ & are defined up to isomorphism by
                              $\Z_{\fp}\otimes_{\Z_F}\cO\cong \Mat_{d_{\fp}}(\Lambda_{\fp})$, where \\
                             & $\Lambda_{\fp}$ is a local ring.
                            From now on fix such isomorphisms.\\
  $\Delta_{\fp}$ & the division algebra $F_{\fp}\otimes_{\Z_{\fp}}\Lambda_{\fp}$.\\
  $\delta_D$ & the finite product $\prod_{\fp}\fp^{d-d_{\fp}}\subset \Z_F$.\\
  $K(\fp^i)$, & $\GL_{d_{\fp}}(\Lambda_{\fp})$ if $i=0$,
  $1+\Jac(\Lambda_{\fp})^i\Mat_{d_{\fp}}(\Z_{\fp})$ if $i>0$ (i.e. the kernel of\\
$\quad{i\in \Z_{\geq 0}}$ & reduction modulo $\Jac(\Lambda_{\fp})^i$), where $\Jac(\Lambda_{\fp})$ is 
                                               the maximal ideal of $\Lambda_{\fp}$.\\
  $K(\fN)$ & $\prod_{\fp^i\parallel\fN}K(\fp^i)$, the product running over all maximal ideals of $\Z_F$.\\
  $K_f$ & an open subgroup of $\G(\adel_{F,f})$ containing
  $K(\fN)$ and\\
  & such that $K_f/\centre(K_f)$ is compact.\\
  $\nrd$ & the reduced norm map from $\G$ to the multiplicative group $\G_m$.\\
  $\G(\adel_F)^+$ & $\{g=(g_v)_v\in \G(\adel_F):
  \det(g_v)>0\text{ for all }v\in \cV_{\R}(D)\}$.\\
    $H^+$ & $H\cap \G(\adel_F)^+$, where $H$ is a subgroup of $\G(\adel_F)$.\\
    $\cY$ & $\G(F)^+\backslash \G(\adel_F)/Z_{\infty} K_{\infty}K_f$.\\ 
  $X$ & $G_{\infty}/(Z_{\infty} K_{\infty})$.\\ 
  $C$ & $\nrd(\G(F)^+)\backslash\adel_{F,f}^\times/\nrd(K_f)$.\\ 
  $\Gamma_c,\, c\in C$ & $\Gamma_g = \G(F)^+\cap gK_fg^{-1}$ where the class of~$\nrd(g)$ in~$C$ is~$c$. \\
  $Y_c$ & $\Gamma_c\lquo X$.\\ 
  $\normaliser(K_f)$ & the normaliser of~$K_f$ in~$\G(\adel_{F,f})$.\\
  $W$ & $F^\times\backslash(Z_{\infty}K_{\infty}\times \normaliser(K_f)/K_f)$.\\ 
  $C_{\iso}$ & see discussion after Proposition \ref{prop:plustransfer}, in
  particular Equation~(\ref{eq:Ciso}). \\
      $\bar{c}$ & the class of~$c$ in~$C_{\iso}$, where $c\in C$.\\
  $c^\perp$ & the set of~$\chi\in \Cisodual$ such that~$\chi(c)=1$, where $c\in C$.\\
  $\abstracthecke$ & the subalgebra of $\Z[K_f\backslash \G(\adel_{F,f})/K_f]$
           generated by the\\
           & $K_f\G(F_{\fp})K_f$ for all $\fp\nmid \delta_D\fN$.\\[0.5em]
\end{tabular}

We will now recall some basic facts about these objects. We refer to
\cite[\S 2.6, \S 5.5, \S 15.2]{GetzHahn} for these facts and further material.

The double quotient $\cY$ is a compact orientable orbifold.
Fix a set $\dcr$ of double coset
representatives of $\G(F)^+\backslash \G(\adel_{F,f})/K_f$. For each
$g\in \dcr$, define $\Gamma_g=\G(F)^+\cap gK_fg^{-1}$, and $Y_g = \Gamma_g\backslash X$.
Then each $Y_g$ is connected, and one has $\cY=\bigsqcup_{g\in \dcr}Y_g$.
The reduced norm map induces a bijection between the component
set $\dcr$ and the group $C$, and we will from now on implicitly use
this bijection, and will index the connected components of $\cY$ by $C$.
The group $C$ is a quotient of
$$
\nrd(\G(F)^+)\backslash\nrd(\G(\adel_{F,f}))/\nrd(K(\fN)) = \Cl_F(\fN\infty).
$$

We have
$$
G_{\infty} \cong \GL_d(\R)^{\cV_{\R}(D)} \times \GL_d(\CC)^{\cV_{\CC}(D)} \times \GL_{d/2}(\Hamil)^{\cV_{\Hamil}(D)}.
$$
The maximal compact subgroup $K_\infty$ is accordingly a product of maximal
compact subgroups over the infinite places of $F$, is unique up to conjugation,
and may be taken to be
$$
K_{\infty} \cong \O_d(\R)^{\cV_{\R}(D)} \times \U_d(\CC)^{\cV_{\CC}(D)} \times \U_{d/2}(\Hamil)^{\cV_{\Hamil}(D)}.
$$
Accordingly, the dimension of $X$ (and therefore also of $\cY$) can be computed to be
$\left(\frac{d^2+d-2}{2}\right)\cdot\#\cV_{\R}(D)+ \left(d^2-1\right)\cdot\#\cV_{\CC}(D)+\left(\frac{d^2-d-2}{2}\right)\cdot\#\cV_{\Hamil}(D)$.
Note, in particular, that when $d=2$, the condition on $D$ to be not totally definite,
equivalently the set $\cV_{\R}(D)\cup \cV_{\CC}(D)$ to be non-empty, is precisely
the condition to ensure that the dimension of $\cY$ is positive.

We now give an alternative description of~$\cY$, which we will use in the next
section.

\begin{lemma}\label{lem:KplusZplus}
  We have $Z_{\infty}^+K_{\infty}^+ = G_{\infty}^+ \cap (Z_{\infty} K_{\infty})$.
\end{lemma}
\begin{proof}
  One inclusion is obvious. For the other, it suffices to argue place by place, and the
  statement is only non-empty at the places $v\in \cV_{\R}(D)$. Locally at such
  a place, the claim is that we have $\R^\times_{>0}\SO_d(\R)= \GL_d(\R)^+ \cap (\R^\times\O_d(\R))$,
  where $\GL_d(\R)^+$ denotes the group of real $d\times d$ matrices with positive determinant,
  and where $\R^\times$ denotes the subgroup of $\GL_d(\R)$ consisting of scalar matrices.
  To prove the claim, suppose that we have $z\in\R^\times$, $k\in \O_d(\R)$ such that
  $\det(zk)>0$. If we have $\det(z)>0$, then we also have $\det(k)>0$, hence
  $zk\in \R^\times_{>0}\SO_d(\R)$. If, on the other hand, we have $\det(z)<0$, so that in particular $d$ is odd,
  then $-1\in \R^\times$ has determinant $-1$, and is also an orthogonal matrix,
  so that we have $zk = (-1\cdot z)(-1\cdot k)\in \R^\times_{>0}\SO_d(\R)$. This proves
  the claim, and hence the lemma.
\end{proof}

\begin{proposition}\label{prop:plustransfer}
We have canonical isomorphisms
$$
\cY = \G(F)\backslash \G(\adel_F)/ Z_{\infty}^+K_{\infty}^+K_f
$$
and
 $$ 
 C=F^{\times}_+\backslash\adel_{F,f}^\times/\nrd(K_f) =
 F^{\times}\backslash (\adel_{F,f}^\times\times\{\pm 1\}^{\cV_\R(D)\cup \cV_{\Hamil}(D)})/\nrd(K_f),
  $$
  where $F^{\times}_{+}=\{\alpha\in F: v(\alpha)>0 \text{ for all }v\in \cV_{\R}(D)\cup \cV_{\Hamil}(D)\}$.
\end{proposition}
\begin{proof}
  First we prove the first isomorphism.
  By Lemma \ref{lem:KplusZplus}, the inclusion $G_{\infty}^+\hookrightarrow G_{\infty}$ induces
  an injection
  $$
  \G(F)^+\backslash (G_{\infty}^+\cdot \G(\adel_{F,f}))/Z_{\infty}^+K_{\infty}^+K_f\hookrightarrow
  \G(F)^+\backslash (G_{\infty}\cdot \G(\adel_{F,f}))/Z_{\infty} K_{\infty}K_f.
  $$
  Since the map $\det\colon K_{\infty}\to \{\pm 1\}^{\cV_{\R}(D)}$ is surjective, this
  injection is also a surjection.

  Next, the same inclusion induces an injection
  $$
  \G(F)^+\backslash (G_{\infty}^+\cdot\G(\adel_{F,f}))/Z_{\infty}^+K_{\infty}^+K_f \hookrightarrow
  \G(F)\backslash (G_{\infty}\cdot \G(\adel_{F,f}))/Z_{\infty}^+K_{\infty}^+K_f.
  $$
  By weak approximation, the map
  $\sign\nrd\colon \G(F)\to \{\pm 1\}^{\cV_{\R}(D)}$ is also surjective, so
  this injection, too, is a surjection, and the first isomorphism is proven.

  Next, we show that $C=F^{\times}_+\backslash\adel_{F,f}^\times/\nrd(K_f)$. Indeed,
  by the Hasse--Schilling--Maass Theorem \cite[Theorem 33.15]{MaxOrders},
  the image of $\G(F)^+$ under the reduced norm map is precisely $F^\times_+$.

  The last isomorphism follows from the fact that the map
  $F^\times\to \{\pm 1\}^{\cV_{\R}(D) \cup \cV_{\Hamil}(D)}$ is surjective.
\end{proof}

Applying Proposition \ref{prop:plustransfer}, we have a map
\begin{eqnarray*}
  \nu\colon W & \to & C=F^{\times}\backslash (\{\pm 1\}^{\cV_\R(D)\cup \cV_{\Hamil}(D)}\times \adel_{F,f}^\times)/\nrd(K_f)\\
  w=(w_{\infty},w_f) & \mapsto & F^{\times}(\sign\nrd(w_{\infty}),\nrd(w_f))^{-1}\nrd(K_f),
\end{eqnarray*}
and we define, as in Section \ref{sec:gradedAlg}, the quotient $C_{\iso}=C/\nu(W)$.
We have a canonical isomorphism 
\begin{eqnarray}\label{eq:Ciso}
C_{\iso}\cong F^\times\backslash(\{\pm 1\}^{\cV_{\Hamil}(D)})\times\adel_{F,f}^\times/\nrd(\normaliser(K_f)),
\end{eqnarray}
and since~$\normaliser(K_f)$ contains~$\adel_{F,f}^\times K_f $, the group~$C_{\iso}$ is a quotient of
\[
  \nrd(\G(F)^+)\backslash\nrd(\G(\adel_{F,f}))/\nrd(K_{\infty}K(\fN)\adel_{F,f}^\times)
  \hspace{-0.2em}=\hspace{-0.2em} \Cl_F(\fN\cV_{\Hamil}(D))/(\Cl_F(\fN\cV_{\Hamil}(D)))^d.
\]
Elements of $W$ act via isometries on $\cY=\G(F)\backslash \G(\adel_F)/ Z_{\infty}^+K_{\infty}^+K_f$ by right multiplication.

The algebra $\abstracthecke$ is commutative. It is generated by double cosets
$$
T_{\cD} = K_f\diag(\cD)K_f,
$$
where 
$\cD=(\fa_1,\ldots,\fa_d)\in \left(\adel_{F,f}^\times/\prod_{\fp}\Z_{\fp}^\times\right)^d$
is such that for every $\fp|\delta_D\fN$ and for every~$\fa_i$ one has $\ord_{\fp}(\fa_i)=0$,
i.e. $\cD$ is a $d$-tuple of non-zero fractional ideals of $F$ that are coprime
to~$\delta_D\fN$.
The algebra $\abstracthecke$ is a subalgebra of $\Z[K_f\backslash \G(\adel_{F,f})/K_f]$.

There is an involution on $\Z[K_f\backslash \G(\adel_{F,f})/K_f]$ -- see \cite[Ch. 1, \S 8.6(d)]{VignerasRep} --
given by inversion, which induces an involution $\iota$ on $\abstracthecke$. Explicitly, it is given by
\begin{eqnarray*}
  \iota\colon \abstracthecke & \to & \abstracthecke,\\
  T_{\cD} & \mapsto & T_{\cD^{-1}}.
\end{eqnarray*}
This involution is extended to $\abstracthecke[W]$ by defining it to be inversion on $W$.

The algebra $\abstracthecke[W]$ is graded by $C$,
the grading being defined on $\abstracthecke$ by the condition that $T_{\cD}$
for $\cD$ as above belongs to the homogeneous component of the image of
$$
[\cD]^{-1}=\prod_{i=1}^d[\fa_i]^{-1}\in \Cl_F(\fN\infty)
$$
in the quotient $C$, and being induced on $W$ by the homomorphism $\nu$.

\begin{definition}\label{def:eigensystem}
  Let $L$ be a field.
  A \emph{Hecke eigensystem over $L$} is an $L$-algebra homomorphism
  $a = (a_\cD)_{\cD}\colon L\otimes \abstracthecke\to L$, $1\otimes T_{\cD}\mapsto a_{\cD}$.
  Let $M$ an $(L\otimes\abstracthecke)$-module.
  The \emph{multiplicity in $M$} of a Hecke eigensystem~$(a_\cD)$
  is the $L$-dimension of the space of all elements $f\in M$
  satisfying
  $$
  T_{\cD}f = a_{\cD}f
  $$
  for all $\cD\in \left(\adel_{F,f}^\times/\prod_{\fp}\Z_{\fp}^\times\right)^d$.
  A \emph{Hecke eigensystem in $M$} is a Hecke eigensystem that has multiplicity
  at least $1$ in $M$.

  Given a Hecke eigensystem~$(a_\cD)_{\cD}$ in $M$ and $\chi\in \Cdual$,
  we say that~$(a_\cD)_{\cD}$ \emph{admits a self-twist} by $\chi$ if for all
  $\cD\in \left(\adel_{F,f}^\times/\prod_{\fp}\Z_{\fp}^\times\right)^d$ satisfying
  $\chi([\cD])\neq 1$ one has $a_{\cD}=0$. In the case when $R$ contains all
  $n$-th roots of unity, where $n$ is the order of $\chi$, we may view $\chi$
  as taking values in $R^\times$, and then the condition of admitting a self-twist
  is equivalent to the condition that for all $\cD$ one has $\chi([\cD])\cdot a_{\cD}=a_{\cD}$.
\end{definition}

The following two lemmas will link the self-twist condition with Proposition \ref{prop:littleExercise}.

\begin{lemma}\label{lem:SubmodEigensystem}
  Let $R$ be a domain, let $\hecke=R\otimes \abstracthecke$, let $M$ be a $\hecke$-module that is finitely
  generated over $R$, let $\fp\in \MaxSpec(R)$, and let $E$ be a simple $\hecke$-submodule
  of $R/\fp\otimes_R M$. Then there exists a finite field extension $S$ of $R/\fp$ and a
  Hecke eigensystem $(a_{\cD})_{\cD}$ over~$S$ in $S\otimes_{R} M$ such that whenever
  $T_{\cD}\in \abstracthecke$ annihilates $E$, one has $a_{\cD}=0$.
\end{lemma}
\begin{proof}
  Let $\fm$ be the annihilator of $E$ in $R/\fp\otimes \hecke$, and let
  $S=(R/\fp\otimes \hecke)/\fm$ be the residue field of $\fm$. Since $M$ is finitely generated over $R$,
  the field $S$ is a finite extension of $R/\fp$. The module $S\otimes_{R/\fp} E$
  is a $1$-dimensional $S$-vector subspace of $S\otimes_{R}M$. For every $\cD$,
  define $a_{\cD}$ to be the image of $T_{\cD}$ in $S$. Then $(a_{\cD})_{\cD}$ is a Hecke eigensystem
  with the claimed property.
\end{proof}

\begin{lemma}\label{lem:nonInvertibleHecke}
  Let $M$ be a graded $\abstracthecke[W]$-module,
  and let $c\in C$. Suppose that there exists a field $L$,
  a character $\chi\in \Cisodual\setminus c^\perp$, and a Hecke
  eigensystem over~$L$ in~$L\otimes M$ admitting a self-twist by~$\chi$. Then
  no element of $\abstracthecke[W]_c$ acts invertibly on~$M$.
\end{lemma}
\begin{proof}
  Let $f\in L\otimes M$ be an eigenvector whose eigensystem has a self-twist by $\chi$. We
  will show that every $T\in \abstracthecke[W]_c$ annihilates $f$. First, let $T=tw\in \abstracthecke[W]_c$,
  where $w\in W$ and $t\in \abstracthecke_{c\nu(w)^{-1}}$. Since $\chi(c\nu(w)^{-1})=\chi(c)\neq 1$,
  we have $tf=0$, and therefore $Tf=wtf=0$. Since $\abstracthecke[W]_c$ is spanned by elements
  $T$ of this form, the proof of the claim is complete.
  Finally, if an element of $\abstracthecke[W]_c$ acted invertibly on~$M$, it would also
  act invertibly on~$L\otimes M$; this proves the lemma.
\end{proof}

Fix an integer $i\in \{0,\ldots,d\}$. As in the introduction,
let $\Delta$ denote the Laplace operator. If $Y$ is either $\cY$ or
$Y_c$ for some $c\in C$, and $\lambda\in \R$,
we will consider the following $R$-modules $\cF(Y)$ attached to $Y$ for
suitable rings $R$:
\begin{itemize}[leftmargin=*]
  \item $\cF(Y)=\Omega^i_{\Delta=\lambda}(Y)$,
the space of real differential $i$-forms on $Y$ on which $\Delta$ acts
by multiplication by $\lambda$, with $R=\R$,
\item $\cF(Y)=\cH^i(Y)=\Omega^i_{\Delta=0}$,
the space of real harmonic $i$-forms, with $R=\R$,
\item the homology group $\cF(Y) = H_i(Y,R)$, where $R$ is any domain, e.g. $\Z$ or $\Z_p$.
\end{itemize}

For each of these collections $\cF$ of $R$-modules one has a direct sum decomposition
$$
\cF(\cY) = \bigoplus_{c\in C}\cF(Y_c),
$$
which defines a $C$-grading on $\cF(\cY)$. We will apply the formalism of Section \ref{sec:gradedAlg}
with $\hecke=R\otimes \abstracthecke$. The Hecke algebra
$\hecke[W]$ naturally acts on~$M=\cF(\cY)$,
and the image~$\heckebig$ of $\hecke[W]$ in $\End(\cF(\cY))$
is an~$R$-algebra that is finitely generated as an $R$-module.
The algebra $\heckebig$ inherits the $C$-grading and the involution~$\iota$ from $\abstracthecke[W]$, and
$\cF(\cY)$ is a graded $\heckebig$-module in the sense of Section \ref{sec:gradedregconst}.

\begin{thm}\label{thm:HeckeIsosp}
  Let~$c\in C$, and let~$i\ge 0$ be an integer.
  \begin{enumerate}[leftmargin=*,label={\upshape(\arabic*)}]
    \item\label{item:lambdaiso} Let~$\lambda\in\R$. Then
      exactly one of the following two statements is true:
      \begin{enumerate}[label={\upshape(\roman*)}]
        \item there exist~$\chi\in \Cisodual\setminus c^\perp$ and a Hecke
          eigensystem over~$\CC$ in~$\CC\otimes_{\R}\Omega^i_{\Delta=\lambda}(\cY)$ admitting a self-twist
          by~$\chi$;
        \item there exists~$T\in\abstracthecke[W]_c$ inducing, for all $b\in C$, an isomorphism
          of~$\abstracthecke_1$-modules
          \[
            T \colon \Omega^i_{\Delta=\lambda}(Y_{b})
            \to \Omega^i_{\Delta=\lambda}(Y_{cb}).
          \]
      \end{enumerate}
    \item\label{item:allLambdasIso}
      At least one of the following two statements is true:
        \begin{enumerate}[label={\upshape(\roman*)}]
          \item there exist~$\chi\in \Cisodual\setminus c^\perp$ and a Hecke
            eigensystem over~$\CC$ in the module
            $\CC\otimes_{\R}\bigoplus_{\lambda\in\R} \Omega^i_{\Delta=\lambda}(\cY)$ 
            admitting a self-twist by~$\chi$;
          \item for all $b\in C$, the manifolds~$Y_{b}$ and~$Y_{cb}$ are
            $i$-isospectral.
        \end{enumerate} 
    \item \label{item:allpIso}
      Let~$p$ be a prime number. Then exactly one of the
      following two statements is true:
      \begin{enumerate}[label={\upshape(\roman*)}]
        \item there exist~$\chi\in \Cisodual\setminus c^\perp$ and a Hecke
          eigensystem over~$\Fpbar$ in~$\Fpbar\otimes H_i(\cY,\Z)$ admitting a self-twist by~$\chi$;
        \item there exists~$T\in\abstracthecke[W]_c$ inducing, for all $b\in C$, an isomorphism
          of~$\abstracthecke_1$-modules
          \[
            T \colon H_i(Y_{b},\Z_{(p)})
            \to H_i(Y_{cb},\Z_{(p)}).
          \]
      \end{enumerate}
    \item\label{item:localpIso} Let $\fn$ be a maximal ideal of $\abstracthecke_1$,
      let~$p$ be the characteristic of $\abstracthecke_1/\fn$. Then exactly one of the
      following two statements is true:
      \begin{enumerate}[label={\upshape(\roman*)}]
        \item there exist~$\chi\in \Cisodual\setminus c^\perp$ and a Hecke
          eigensystem over~$\Fpbar$
          in~$\Fpbar\otimes H_i(\cY,\Z)_{\fn}$ admitting a self-twist by~$\chi$;
        \item there exists~$T\in\abstracthecke[W]_c$ inducing, for all $b\in C$, an isomorphism
          of~$\abstracthecke_1$-modules
          \[
            T \colon H_i(Y_{b},\Z)_{\fn}
            \to H_i(Y_{cb},\Z)_{\fn}.
          \]
      \end{enumerate}
    \end{enumerate}
\end{thm}
\begin{proof}~
  \begin{enumerate}[leftmargin=*,label={\upshape(\arabic*)}]
    \item Take~$R = \R$. Suppose that there are no $\chi$ and Hecke eigensystem
      as in the statement. By Lemma \ref{lem:SubmodEigensystem}, the
      graded~$\heckebig$-module~$M = \Omega_{\Delta=\lambda}^i(\cY)$
      satisfies the hypotheses of Proposition~\ref{prop:littleExercise}, so
      there exists~$T'\in \heckebig_c$ that acts invertibly on~$M$.
      The image of $\abstracthecke[W]_c$ in $\heckebig$ generates the real vector space $\heckebig_c$,
      and being an isomorphism is a non-empty Zariski-open condition,
      so there also exists $T\in
      \abstracthecke[W]_c$ realising an isomorphism as claimed.
      The converse follows from Lemma \ref{lem:nonInvertibleHecke}.
    \item This follows immediately from~\ref{item:lambdaiso}.
    \item Take~$R = \Z_{(p)}$. Suppose, once again, that there are no
      $\chi$ and Hecke eigensystem as in the statement.
      By Lemma \ref{lem:SubmodEigensystem}, the
      graded~$\heckebig$-module~$M = H_i(\cY,\Z_{(p)})$ satisfies the hypotheses of
      Proposition~\ref{prop:littleExercise}, so
      there exists~$T'\in \heckebig_c$ that acts invertibly on~$M$.
      Clearing denominators and lifting to $\abstracthecke[W]$, we obtain a $T\in
      \abstracthecke[W]_c$ as required.
      The converse follows from Lemma \ref{lem:nonInvertibleHecke}.
    \item The proof is identical to that of the previous part,
      with~$R = \Z_{p}$ and~$M = H_i(\cY,\Z)_{\fn}$.\qedhere
  \end{enumerate}
\end{proof}

\begin{remark}
  Notice that in part \ref{item:allLambdasIso} there is no claim that the isospectrality
  is realised by a Hecke operator, and we do not get an equivalence but only one implication.
\end{remark}

The group $G_{\infty}$ acts by unitary operators on the complex Hilbert space
$$
  \rL^2(\G(F)^+\lquo\G(\adel_{F})/Z_{\infty} K_f)
$$
of square-integrable functions
$\G(F)^+\lquo\G(\adel_{F})/Z_{\infty} K_f\to \CC$, and also, for each $c\in C$,
on the analogously defined space $\rL^2(\Gamma_c\lquo G_{\infty}/Z_{\infty})$.
\begin{definition}
  Let $c$, $c'\in C$. We say that the groups $\Gamma_{c}$ and $\Gamma_{c'}$
  are \emph{representation equivalent} if there is an isomorphism of unitary
  $G_{\infty}$-representations
  \[
    \rL^2(\Gamma_{c}\lquo G_\infty / Z_{\infty}) \cong \rL^2(\Gamma_{c'}\lquo
    G_\infty / Z_{\infty}).
  \]
\end{definition}

The Hecke algebra $\abstracthecke[W]$ naturally acts on $\rL^2(\G(F)^+\lquo\G(\adel_{F})/Z_{\infty} K_f)$,
making the $\rL^2$-space a $C$-graded $\abstracthecke[W]$-module.

\begin{thm}\label{thm:repEquiv}
  Let~$c \in C$. Then
  at least one of the following two statements is true:
  \begin{enumerate}[leftmargin=*,label={\upshape(\roman*)}]
    \item there exist~$\chi\in \Cisodual\setminus c^\perp$ and a Hecke
      eigensystem over~$\CC$ in the module~$\rL^2(\G(F)^+\lquo\G(\adel_{F})/Z_{\infty} K_f)$
      admitting a self-twist by~$\chi$;
    \item for all $b\in C$ the groups~$\Gamma_{b}$ and~$\Gamma_{cb}$ are representation equivalent.
  \end{enumerate}
  \end{thm}
\begin{proof}
  Suppose that there are no $\chi$ and Hecke eigensystem as in the statement.
  Fix $b\in C$.
  By \cite{GelPS} (see also \cite{Godement}), there are decompositions
  into Hilbert orthogonal direct sums of isotypical unitary
  representations of~$G_\infty$ with finite multiplicities
  \[
    \rL^2(\G(F)^+\lquo\G(\adel_{F})/Z_{\infty} K_f) \cong \widehat{\bigoplus}_\automrep V(\automrep),
  \] 
  and, for all~$c'\in C$,
  \[
    \rL^2(\Gamma_{c'}\lquo G_\infty / Z_{\infty}) \cong \widehat{\bigoplus}_\automrep
    V(\automrep)_{c'},
  \]
  both sums indexed by pair-wise non-isomorphic irreducible unitary 
  representations $\automrep$ of $\G(\adel_F)$.
  The action of $\normaliser(K_f)/K_f\times Z_{\infty}$ commutes with that
  of $G_{\infty}$, and $K_{\infty}$ is contained in $G_{\infty}$, so
  each~$V(\automrep)$ is also a~$\abstracthecke[W]$-module, and is
  graded, with~$c'$-component~$V(\automrep)_{c'}$.
  
  Let~$\automrep$ be an irreducible unitary representation of~$G_\infty$.
  We claim that~$\automrep$ occurs with the same multiplicities in $V(\automrep)_{b}$
  and in~$V(\automrep)_{cb}$.
  We have a Hilbert direct sum decomposition of isotypical unitary
  finite-dimensional representations of~$K_\infty$
  \[
    V(\automrep)|_{K_\infty} \cong \widehat{\bigoplus}_\sigma V(\automrep)(\sigma),
  \]
  with the sum running over pair-wise non-isomorphic irreducible unitary representations $\sigma$
  of $K_{\infty}$, again compatible with the grading, and where each summand is
  preserved by $\abstracthecke[W]$. Let~$\sigma$ be an
  irreducible unitary representation of~$K_\infty$ such that~$V(\automrep)(\sigma)\neq
  0$, so that it is enough to prove that~$\dim V(\automrep)(\sigma)_{b} =
  \dim V(\automrep)(\sigma)_{cb}$. Let~$R = \CC$ and let~$\heckebig$ be the image
  of~$\CC\otimes\abstracthecke[W]$ in the endomorphism algebra of the
  finite-dimensional~$\CC$-vector space~$V(\automrep)(\sigma)$, so that~$\heckebig$ is a
  finite-dimensional $\CC$-algebra. By Lemma \ref{lem:SubmodEigensystem}
  the graded~$\heckebig$-module~$M = V(\automrep)(\sigma)$
  satisfies the hypotheses of Proposition~\ref{prop:littleExercise}, so there
  exists~$T\in \heckebig_c$ that acts invertibly on~$M$, therefore realising an
  isomorphism
  \[
    T \colon V(\automrep)(\sigma)_{b} \to V(\automrep)(\sigma)_{cb}.
  \]
  This proves the desired equality of dimensions, and therefore the desired
  equality of multiplicities for~$\automrep$.
  Since~$\automrep$ was arbitrary, we obtain an isomorphism as claimed.
\end{proof}

For the rest of the section, fix a degree $i\in \Z_{\geq 0}$ and a prime number $p$,
and let~$\heckebig$ be the image of $\abstracthecke[W]$
in $\End_{\Z}H_i(\cY,\Z)_{\free}$, so that $H_i(\cY,\Z)_{\free}$ is a graded
$\heckebig$-module. We will apply the results of Section \ref{sec:gradedAlg} with
$R=\Z$ and $\hecke=\abstracthecke$.
For every~$c\in C$ there is a canonical positive definite pairing on~$\cH^i(Y_c)$
\cite[Notation 3.2 and Lemma 3.3]{us1}, \cite{Laplace}.
By the Hodge and de Rham
theorems, this harmonic forms pairing induces a non-degenerate $\R$-valued pairing
on~$H_i(Y_c,\R)$. The \emph{$i$-th regulator $\Reg_i(Y_c)$ of $Y_c$} is defined to be
the covolume of the lattice~$H_i(Y_c,\Z)_{\free}$ in~$H_i(Y_c,\R)$.
\begin{definition}\label{def:curlyN}
Define~$\cN=\{\fn\cap\iota(\fn):\fn\in
\MaxSpec(\Z_{p}\otimes \abstracthecke_1)\}$.
\end{definition}

\begin{lemma}\label{lem:harmonPolar}
  Define a pairing on $M=H_i(\cY,\Z)_{\free}$ induced by the harmonic forms pairing
  on the summands $H_i(Y_c,\Z)_{\free}$ for $c\in C$ and by making the distinct
  summands pairwise orthogonal. Then this defines a polarisation on
  $M$ in the sense of Definition \ref{def:pol}.
  Moreover, for every $\fn'\in \cN$ the graded $\Z_p\otimes \heckebig$-module $M_{\fn'}$
  is polarisable.
\end{lemma}
\begin{proof}
  The adjoint of each~$T\in \abstracthecke$ with respect to the harmonic forms pairing is~$\iota(T)$,
  and $W$ acts by isometries (see also Remark
\ref{rmrk:Wadjoint}), so the pairing defines a polarisation on the graded module~$\cH^i(\cY)$.
  Moreover, the Hodge--de Rham isomorphism between $\cH^i(\cY)$ and $H_i(\cY,\R)$ is
  $\heckebig$-equivariant, so the harmonic forms pairing is an $\R$-valued polarisation
  on the $\heckebig$-module $M$. The last assertion follows from Proposition \ref{prop:HeckeProd}.
\end{proof}

\begin{lemma}\label{lem:RegConstFormalism}
  Let $c$, $c'\in C$, let $M=H_i(\cY,\Z)_{\free}$,
  and suppose that $M_{c}$ and~$M_{c'}$
  are linked in the sense of Definition \ref{def:linked}. Then:
  \begin{enumerate}[leftmargin=*,label={\upshape(\arabic*)}]
    \item\label{item:regseqregconst} we have
      \[
        \cC_{c,c'}(M) = \frac{\Reg_i(Y_{c})^2}{\Reg_i(Y_{c'})^2},
      \]
      where recall that the invariant $\cC_{c,c'}(M)$ was defined in Definition \ref{def:gradedregconst};
    \item\label{item:rationalregs} we have
      \[
        \frac{\Reg_i(Y_{c})^2}{\Reg_i(Y_{c'})^2} \in \Q^\times;
      \]
    \item\label{item:HeckeProd} 
      for all~$\fn'\in \cN$, 
      the homogeneous components~$(M_{\fn'})_{c}$ and $(M_{\fn'})_{c'}$ are linked,
      and we have
      \[
        \cC_{c,c'}(M) \equiv \prod_{\fn'\in \cN} \cC_{c,c'}(M_{\fn'})
        \mod{(\Z_p^\times)^2},
      \]
      where~$\cC_{c,c'}(M)\in \Q^\times$ is viewed as an element of~$\Q_p^\times$.
  \end{enumerate}
\end{lemma}
\begin{proof}
  \begin{enumerate}[leftmargin=*,label={\upshape(\arabic*)}]
    \item By definition, $\Reg_i(Y_{c})^2$ is the determinant of the Gram matrix of the
      harmonic forms pairing with respect to any $\Z$-basis of $H_i(Y_{c},\Z)_{\free}$, and
      similarly for $Y_{c'}$. The claim therefore follows from Lemma \ref{lem:harmonPolar}
      and Proposition~\ref{prop:indepT}.
    \item The assertion follows from combining part \ref{item:regseqregconst} and Corollary \ref{cor:rational}.
    \item The assertion follows from Lemma \ref{lem:extendring},
      Lemma \ref{lem:harmonPolar}, and Proposition \ref{prop:HeckeProd}.\qedhere
  \end{enumerate}
\end{proof}

\begin{thm}\label{thm:regQuosRatl}
  Let~$c\in C$.
  Then exactly one of the following two statements is true:
  \begin{enumerate}[leftmargin=*,label={\upshape(\roman*)}]
    \item there exist~$\chi\in \Cisodual\setminus c^\perp$ and a Hecke
      eigensystem over~$\CC$
    in~$\CC\otimes_{\R}\cH^i(\cY)$ admitting a self-twist by~$\chi$;
  \item for all $b\in C$ the homogeneous components $H_i(Y_{b},\Z)_{\free}$ and $H_i(Y_{cb},\Z)_{\free}$
    are linked.
  \end{enumerate}
\end{thm}
\begin{proof}
  Suppose that there are no $\chi$ and Hecke eigensystem
  as in the statement. The isomorphism between $\cH^i(\cY)$ and $H^i(\cY,\R)$ is
  $\heckebig$-equivariant. Lemma \ref{lem:SubmodEigensystem}, applied
  to $M=H_i(\cY,\Z)_{\free}$, and Corollary \ref{cor:linked}
  imply that $M_{c}$ and $M_{c'}$ are linked.
  The converse follows from Lemma \ref{lem:nonInvertibleHecke}.
\end{proof}

\begin{thm}\label{thm:regQuosValues}
  Let $c\in C$, and let~$p$ be a prime number, let $M=H_i(\cY,\Z)_{\free}$, and let $\cN$ be as
  in Definition \ref{def:curlyN}.
  \begin{enumerate}[leftmargin=*,label={\upshape(\arabic*)}]
    \item\label{item:TrivialLocalRegconst}
      Let~$\fn'\in \cN$.
      Then at least one of the following two statements is true:
      \begin{enumerate}[label={\upshape(\roman*)}]
        \item there exist~$\chi\in \Cisodual\setminus c^\perp$
          and a Hecke eigensystem over~$\Fpbar$ in~$\Fpbar\otimes (M/\fn' M)$ admitting a self-twist by~$\chi$;
        \item for all $b\in C$, the homogeneous components $(M_{\fn'})_b$ and $(M_{\fn'})_{cb}$ are linked, and we have
        \[
          \cC_{b,cb}(M_{\fn'}) \in \Z_p^\times;
        \]
      \end{enumerate}
    \item\label{item:TrivialGlobalRegconst}
      At least one of the following two statements is true:
      \begin{enumerate}[label={\upshape(\roman*)}]
        \item there exist~$\chi\in \Cisodual \setminus c^\perp$ and a Hecke
          eigensystem over~$\Fpbar$ in~$\Fpbar\otimes M$ admitting a self-twist by~$\chi$;
        \item 
          for all~$b\in C$, the homogeneous components $M_b$ and $M_{cb}$ are linked,
          and we have
          \[
            \frac{\Reg_i(Y_{b})^2}{\Reg_i(Y_{cb})^2}\in \Z_{(p)}^\times.
          \]
      \end{enumerate}
  \end{enumerate}
\end{thm}
\begin{proof}
  \begin{enumerate}[leftmargin=*,label={\upshape(\arabic*)}]
    \item Suppose that there are no $\chi$ and Hecke eigensystem
      in~$\Fpbar\otimes (M/\fn' M)$ as in the statement.
      By Lemma \ref{lem:SubmodEigensystem} the assumptions of Proposition
      \ref{prop:littleExercise} are satisfied for $M_{\fn'}$,
      so there exists $T\in \heckebig_c$ that acts invertibly on $M_{\fn'}$.
      The result follows from Lemma \ref{lem:triv}.
    \item Suppose that there are no $\chi$ and Hecke eigensystem
      in~$\Fpbar\otimes M$ as in the statement, and let~$b\in C$.
      By the same argument as in part \ref{item:TrivialLocalRegconst}, the homogeneous components 
      $M_b$ and $M_{cb}$ are linked.
      The claimed equality then follows by combining
      Lemma~\ref{lem:RegConstFormalism} and part~\ref{item:TrivialLocalRegconst}
      of the present theorem.\qedhere
  \end{enumerate}
\end{proof}
%

%
%
\section{\texorpdfstring{Self-twists in characteristic~$0$ and automorphic induction}{Self-twists in characteristic 0 and automorphic induction}}\label{sec:isoconds}

\begin{assumption}\label{ass:quaternion}
We keep the notation of Section~\ref{sec:Vigneras}, but we
assume that~$d=2$, i.e. that~$D$ is a quaternion algebra.
In particular we have~$Z_\infty = Z_\infty^+$.
\end{assumption}

In this section we prove Theorem \ref{thm:IntroProtoMain}
for representation equivalence, $\Omega^\bullet$-, $\Omega^0$-, and $\cH^\bullet$-isospectrality,
Theorem~\ref{thm:IntroRegLL} parts~\ref{item:IntroRegLLRegConst} and~\ref{item:IntroRegLLProd},
and Theorem \ref{thm:IntroKelmerConverse}. To prove Theorem~\ref{thm:IntroProtoMain}
we will use the results of Section \ref{sec:Vigneras} and make the self-twist conditions
in those results explicit. To that end, we will relate the various $\abstracthecke$-modules
to automorphic representations, directly or using Matsushima's formula. By
an automorphic induction theorem of Langlands, the self-twist condition yields the
existence of Hecke characters of quadratic extensions of $F$ with specific properties.
To make these properties explicit, we inspect closely irreducible representations
of $\GL_2(\R)$ and $\GL_2(\CC)$ and how they interact with automorphic induction,
Jacquet--Langlands transfer, and Matsushima's formula.
Theorem~\ref{thm:IntroRegLL} parts~\ref{item:IntroRegLLRegConst}
and~\ref{item:IntroRegLLProd} constitute a refinement of Theorem
\ref{thm:IntroProtoMain} for $\cH^\bullet$-isospectrality.
To deduce Theorem \ref{thm:IntroKelmerConverse},
we count the Hecke characters that contribute to the various Casimir eigenspaces.
This is achieved by proving local bounds, Lemmas \ref{lem:localboundarch} and \ref{lem:localboundfin},
and then counting points in lattices.

The content of Sections \ref{sec:GL1}--\ref{sec:automGL2} is well-known.
\begin{itemize}[leftmargin=*]
  \item In Section \ref{sec:GL1} we fix our notation and recall basic facts about Hecke characters.
  \item In Section \ref{subsec:archi} we recall the classification of irreducible representations
    of $\GL_2(E)$ for an Archimedean local field $E$, including their $K$-type,
    Casimir eigenvalue, and central character, and how they arise from automorphic induction and
    Jacquet--Langlands transfer.
  \item In Section \ref{subsec:nonarchi} we recall the classification of smooth irreducible representations
    of~$\GL_2(E)$ for a non-Archimedean local field $E$, and we state which ones arise
    as automorphic inductions, as well as some properties of conductors.
  \item In Section \ref{sec:automGL2} we recall the results on automorphic representations of $\GL_2$
    that we need, namely strong multiplicity $1$, automorphic induction, Jacquet--Langlands correspondence,
    and Matsushima's formula.
\end{itemize}
General references for these facts are \cite[Appendix B]{GerardinLabesse}
\cite{GetzHahn}, \cite{JL}.

In Section \ref{sec:selfTwistConds} we prove Theorem \ref{thm:IntroProtoMain} for the isospectralities
listed above and  Theorem~\ref{thm:IntroRegLL} parts~\ref{item:IntroRegLLRegConst} and~\ref{item:IntroRegLLProd}.
That section is divided into subsections, one for each type of isospectrality.
Finally, in Section \ref{sec:KelmerConverse} we prove Theorem \ref{thm:IntroKelmerConverse}.

Throughout Section \ref{sec:isoconds} all group representations will be on complex Hilbert spaces.

\subsection{\texorpdfstring{Automorphic representations of~$\GL_1$}{Automorphic representations of GL1}}\label{sec:GL1}

In this section, we fix some notation for Hecke characters. For a general
reference, see e.g. \cite[Ch. XIV]{LangANT}.

Let~$G$ be a locally compact group. Recall that a \emph{quasi-character} of~$G$ is a
continuous homomorphism~$G \to \CC^\times$, and a \emph{character} of~$G$ is a
quasi-character that is unitary.

Every quasi-character of~$\R^\times$ is of the form
\[
  \automchar_\R(k,s) \colon x \mapsto \sign(x)^k |x|^s
\]
for a unique pair~$(k,s)\in \Z/2\Z\times\CC$.

Every quasi-character of~$\CC^\times$ is of the form
\[
  \automchar_\CC(k,s) \colon z \mapsto \left(\frac{z}{|z|}\right)^k |z|^{2s}
\]
for a unique pair~$(k,s)\in \Z\times\CC$.

In both cases, such a quasi-character is unitary if and only if one has~$s\in i\R$.

If $E$ is a non-Achimedean local field, $\fp$ is the maximal ideal of the ring
of integers~$\Z_E$ of $E$, and $\automchar$ is a quasi-character of~$E^\times$, then
there exists $m\in \Z_{\geq 0}$ such that $\automchar$ is trivial on
$(1+\fp^m)\cap \Z_E^\times$.
The \emph{conductor} of $\automchar$ is $\fp^m$, where $m$ is the smallest such integer.

A \emph{Hecke quasi-character} of $F$ is a quasi-character of~$F^\times \lquo
\adel_F^\times$, and a \emph{Hecke character} of $F$ is a character of~$F^\times \lquo
\adel_F^\times$.

Every Hecke quasi-character~$\automchar$ of~$F$ is of the form
\[
  \automchar = \prod_{v} \automchar_v
\]
for quasi-characters~$\automchar_v$ of~$F_v^\times$, where~$v$ runs over the places
of~$F$. For every infinite place~$v$, we define~$k_v$ and~$s_v$ by
writing~$\automchar_v = \automchar_{F_v}(k_v,s_v)$. Note that this definition depends on
identifications between $F_v$ and $\CC$ for every complex place $v$.
Replacing a chosen identification with its complex conjugate does not
change $s_v$ but negates~$k_v$. Everything that we write in this section
is insensitive to this ambiguity, as long as we impose the following convention:
whenever $L/F$ is a quadratic extension of number fields, $v$ is a complex place of $F$,
and $w$, $w'$ are places of $L$ extending $v$, we always choose isomorphisms
$L_w\cong \CC$ and $L_{w'}\cong \CC$ that extend \emph{the same} arbitrarily chosen
isomorphism $F_v\cong \CC$.

A Hecke quasi-character~$\automchar$ is called \emph{algebraic} if for every real
place $v$ we have $s_v\in \Z$ and for every complex place $v$ we have~$s_v+k_v/2\in \Z$.
If $\automchar$ is an algebraic Hecke character, then for every
embedding~$\tau\colon F\to \CC$ there exists a uniquely determined~$q_\tau\in \Z$
such that for all~$\alpha\in F^\times$ that are positive at all real places we have
$$
\prod_{v|\infty}\automchar_v(\alpha) = \prod_{\tau\colon F\to \CC}\tau(\alpha)^{q_{\tau}},
$$
and the products run over all infinite places $v$, respectively embeddings~$\tau\colon F\to \CC$.
We refer to the collection~$(q_{\tau})_{\tau \in \Hom(F,\CC)}$ as the \emph{type} of
an algebraic Hecke character.

\subsection{\texorpdfstring{Representations of~$\GL_2$ over local fields}{Representations of GL2 over local fields}}\label{sec:GL2local}

Throughout this subsection, let $E$ be a local field. We will fix the required
notation and briefly recall the relevant facts for irreducible representations of $\GL_2(E)$.
General references for this subsection are \cite[Appendix B]{GerardinLabesse} and 
\cite[\S 5-6, 14-15]{JL}.

Let~$L$ be a quadratic \'etale $E$-algebra, i.e. either a quadratic field
extension of $E$ or a direct product $E\times E$. There is an \emph{automorphic
induction} functor $\AI_E^L$ from the category of irreducible representations
of $\GL_1(L)$ to that of irreducible representations of $\GL_2(E)$. If $\sigma$
is the non-trivial $E$-linear automorphism of $L$, then for all characters
$\automchar$ of $L^\times$ we have $\AI_E^L(\automchar^\sigma)=\AI_E^L(\automchar)$.

Let~$A$ be a quaternion algebra over $E$. The local Jacquet--Langlands correspondence \cite{JL}
attaches to every irreducible representation $\automrep$ of $A^\times$ a
representation~$\JL_{A}(\automrep)$ of $\GL_2(E)$, well-defined up to isomorphism,
sometimes referred to as the \emph{Jacquet--Langlands transfer} of $\automrep$ from
$A^\times$ to $\GL_2$.
If $A$ is split, then $\JL_{A}(\automrep)$ is isomorphic to $\automrep$. We will
say more about the properties of $\JL_{A}(\automrep)$ in some special cases below.

\subsubsection{Archimedean fields}\label{subsec:archi}
In this subsection, assume that $E$ is Archimedean.

Given an irreducible representation $\automrep$ of $\GL_2(E)$ on a vector
space $V$, its \emph{central character} is the group homomorphism
$\zeta_\automrep\colon E^\times\to \CC^\times$ characterised by the property
that for all $x\in E^\times\subset \GL_2(E)$
and all $v\in V$ one has $\automrep(x)(v) = \zeta_\automrep(x)v$.

Let $K$ be a compact subgroup of~$\GL_2(E)$. The restriction of every representation of $\GL_2(E)$ to $K$
decomposes as a Hilbert direct sum of irreducible representations.
Fix a full set of representatives~$u$ of isomorphism classes of
irreducible representations of~$K$. The \emph{$K$-type} of a
representation~$\automrep$ of~$\GL_2(E)$ is the multiset of representatives~$u$
with multiplicities~$\dim \Hom_{K}(u,\automrep)$.

Given an irreducible representation~$\automrep$ of $\GL_2(E)$ or of $\Hamil^\times$, the
\emph{Casimir operator}, which we take to be the negation of~\cite[Chapters 1 and 2]{BorelWallach} (see
also~\cite[Theorem~2.2.1]{Bump} in the case of~$\GL_2(\R)$), acts as a
scalar on a dense subspace of the underlying
vector space. We call this scalar the \emph{Casimir eigenvalue} of $\automrep$.
The definition of the Casimir operator depends on a choice of scalar
multiple of the Killing form, and we choose the normalisation for which,
in the case of $\GL_2(E)$, the quotient by a maximal compact subgroup has
constant curvature $-1$, and in the case of $\Hamil^\times$, the form
agrees with that for $\GL_2(\R)$ after complexification.

Isomorphism classes of irreducible representations of~$\SO_2(\R)$ have
dimension~$1$ and are parametrised
by integers, with~$k\in\Z$ corresponding to
\[
  r(k) \colon
  \begin{pmatrix}
    \cos \theta & \sin \theta \\
    -\sin \theta & \cos \theta
  \end{pmatrix}
  \mapsto e^{ik\theta}.
\]

Isomorphism classes of irreducible representations of~$\SU_2(\CC)$ are
parametrised by non-negative integers, with~$k\in \Z_{\ge 0}$ corresponding to
the representation
\[
  s(k) = \Sym^k \CC^2,
\]
which has dimension~$k+1$, where~$\CC^2$ is equipped with the standard action
of~$\SU_2(\CC)$.

Isomorphism classes of irreducible representations of~$\Hamil^\times$ are parametrised by
elements of $\Z_{\ge 2}\times \CC$, with~$(k,\mu)\in \Z_{\ge 2}\times \CC$
corresponding to the representation
\[
  t(k,\mu) = \Sym^{k-2}\CC^2 \otimes_{\CC} \nrd^{\mu/2},
\]
which has dimension~$k-1$, where~$\CC^2$ is equipped with the action
of~$\Hamil^\times$ induced by the map~$\Hamil^\times \to \SL_2(\CC)$ defined
by~$h\mapsto \nrd(h)^{-1/2}j(h)$ for some isomorphism~$j\colon \CC\otimes_\R \Hamil \cong \Mat_2(\CC)$.
The eigenvalue of the Casimir operator on~$t(k,\mu)$
is~$\frac{k}{2}(1-\frac{k}{2})$ and the central character of~$t(k,\mu)$
is~$\automchar_\R(k,\mu)$.

Every irreducible representation of~$\GL_2(\R)$ is isomorphic to one
from exactly one of the following families.
\begin{itemize}[leftmargin=*]
  \item Irreducible principal series~$\PS_\R(\automchar_1,\automchar_2)$, where~$\automchar_1,\automchar_2$
    are quasi-characters~$\automchar_i = \automchar_\R(k_i,s_i)$ such that~$s = \frac{1}{2}(s_1-s_2+1)$ and~$k=k_1-k_2$
    satisfy~$2s\notin k+2\Z$.
%
%
  \item Discrete series and limits of discrete series~$\DS(k,\mu)$, where~$k\ge
    1$ is an integer and~$\mu\in \CC$.
%
  \item Finite-dimensional
    representations~$\FD_\R(k,\automchar)=\Sym^{k-2}\CC^2\otimes_{\CC} (\automchar\circ \det)$, where $k\ge 2$ is an
    integer and~$\automchar$ is a quasi-character of~$\R^\times$.
%
\end{itemize}


\begin{center}
\begin{tabular}[h]{l|c|c|c}
  Representation & $\SO_2(\R)$-type & Casimir & central \\
   &  & eigenvalue & character \\
  \hline
  $\PS_\R(\automchar_1,\automchar_2)$ & $\{r(m) : m \in k+2\Z\}$ & $s(1-s)$ & $\automchar_1\automchar_2$ \\
  $\DS(k,\mu)$ & $\{r(m) : m \in k+2\Z, |m|\ge k\}$ &
  $\frac{k}{2}(1-\frac{k}{2})$ & $\automchar_\R(k,\mu)$ \\
  $\FD_\R(k,\automchar)$ & $\{r(m) : m \in k+2\Z, |m|< k\}$& $\frac{k}{2}(1-\frac{k}{2})$
   & $\automchar_{\R}(k,k-2)\automchar^2$
\end{tabular}
\end{center}


Every irreducible representation of~$\GL_2(\CC)$ is isomorphic to one
from exactly one of the following families.
\begin{itemize}[leftmargin=*]
  \item Irreducible principal series~$\PS_\CC(\automchar_1,\automchar_2)$, where~$\automchar_1,\automchar_2$
    are quasi-characters~$\automchar_i = \automchar_\CC(k_i,s_i)$ such that~$s = s_1-s_2$ and~$k
    = k_1-k_2$ satisfy $2s\notin k+2\Z$
    or~$|s|<1+\frac{|k|}{2}$.

%
  \item Finite-dimensional representations~$\FD_\CC(k,k',\automchar)=\Sym^{k-2}\CC^2\otimes \Sym^{k'-2}\bar{\CC}^2\otimes (\automchar\circ\det)$, where~$k,k'\ge 2$
    are integers, $\automchar$ is a quasi-character of~$\CC^\times$, and $\CC^2$ denotes the standard representation
    of $\GL_2(\CC)$, while $\bar{\CC}^2$ denotes its composition with complex conjugation.
%
\end{itemize}

\begin{center}
\begin{tabular}[h]{l|c|c|c}
  Repre- & $\SU_2(\CC)$-type & Casimir & central \\
   sentation &  & eigenvalue & character \\
  \hline
  $\PS_\CC(\automchar_1,\automchar_2)$ & $\{s(m) : m \in k+2\Z, m\ge |k|\}$ &
  $1-\frac{k^2}{4}-s^2$ & $\automchar_1\automchar_2$ \\
  $\FD_\CC(k,k',\automchar)$ & $\{s(m) : m \in k+k'+2\Z,$&
  $k(1-\frac{k}{2}) + k'(1-\frac{k'}{2})$
   & $\automchar_{\CC}(k-k',$\\
  &  $|k-k'|\le m \le k+k'-4\}$ & & \quad\quad$\tfrac{k+k'-4}{2})\automchar^2$
\end{tabular}
\end{center}

\begin{lemma}\label{lem:localboundarch}
  Let~$V$ be a finite-dimensional representation of~$K_\infty$. Then
  there exists~$\kappa_\infty>0$ such that for every irreducible
  representation~$\automrep_\infty$ of~$G_\infty$ one has
  \[
    \dim \Hom_{K_\infty}(V,\automrep_\infty) \le \kappa_\infty.
  \]
\end{lemma}
\begin{proof}
  The assertion follows by inspection of the above lists of the
  irreducible representations of~$\GL_2(E)$ for~$E=\R$ and~$E = \CC$.
\end{proof}

The automorphic inductions~$\AI_E^L$ from $\GL_1(L)$ to $\GL_2(E)$ are as follows.
\begin{itemize}[leftmargin=*]
  \item Extension~$\R\subset \CC$: $\AI_{\R}^{\CC}(\automchar_\CC(k,s)) = \DS(1+|k|,2s)$.
  \item Extension~$\R \subset \R\times\R$:
    let~$\automchar_i = \automchar_\R(k_i,s_i)$ for~$i=1,2$ be ordered such
    that~$\Re(s_1)\ge\Re(s_2)$. Then:
    \begin{itemize}[leftmargin=*]
      \item $\AI_{\R}^{\R\times\R}(\automchar_2,\automchar_1)=\AI_{\R}^{\R\times\R}(\automchar_1,\automchar_2) =
        \PS_\R(\automchar_1,\automchar_2)$ if~$s_1-s_2+1\notin k_1+k_2 + 2\Z$;
      \item $\AI_{\R}^{\R\times\R}(\automchar_2,\automchar_1)=\AI_{\R}^{\R\times\R}(\automchar_1,\automchar_2) =
        \FD_\R(k,\automchar_\R(k_2,s_2+\frac{1}{2}))$, where~$k=s_1-s_2+1$, if~$k\in
        k_1+k_2+2\Z$ and~$k\neq 1$;
      \item $\AI_{\R}^{\R\times\R}(\automchar_2,\automchar_1)=\AI_{\R}^{\R\times\R}(\automchar_1,\automchar_2) =
        \DS(1,2s_1) = \DS(1,2s_2)$ if~$s_1=s_2$ and~$k_1\not\equiv k_2\mod{2}$.
    \end{itemize}
  \item Extension~$\CC\subset\CC\times\CC$:
    let~$\automchar_i = \automchar_\CC(k_i,s_i)$ for~$i=1,2$ be ordered such that~$\Re(s_1)\ge \Re(s_2)$,
    $s = s_1-s_2$, and~$k = k_1-k_2$. Then:
    \begin{itemize}[leftmargin=*]
      \item $\AI_{\CC}^{\CC\times\CC}(\automchar_2,\automchar_1) =\AI_{\CC}^{\CC\times\CC}(\automchar_1,\automchar_2) = \PS_\CC(\automchar_1,\automchar_2)$
        if~$2s \notin k+2\Z$ or~$|s|< \frac{|k|}{2}+1$;
      \item $\AI_{\CC}^{\CC\times\CC}(\automchar_2,\automchar_1) =\AI_{\CC}^{\CC\times\CC}(\automchar_1,\automchar_2) = \FD_\CC(s+\frac{k}{2}+1,
        s-\frac{k}{2}+1, \automchar_\CC(k_2,s_2+\frac{1}{2}))$ otherwise.
    \end{itemize}
\end{itemize}

The Jacquet--Langlands transfer~$\JL_{\Hamil}$ from~$\Hamil^\times$ to~$\GL_2(\R)$ is
\[
  \JL_{\Hamil}(t(k,\mu)) = \DS(k,\mu) \text{ for }(k,\mu)\in \Z_{\ge 2}\times\CC.
\]

We now record the isomorphism classes of some representations of various compact groups.

Let $\lieso_2(\R)$ and $\liegl_2(\R)$ be the Lie algebras of $\SO_2(\R)$, respectively of $\GL_2(\R)$,
both equipped with the adjoint action of $\SO_2(\R)$.
We have the following isomorphisms of $\SO_2(\R)$-representations.\\[0.5em]
\begin{center}
\begin{tabular}{l|l}
  $i$ & \quad$\CC\otimes_{\R}\Lambda^i(\liegl_2(\R)/\lieso_2(\R))$\\\hline
  $0$ & \quad $r(0)$\\
  $1$ & \quad $r(-2)\oplus r(2)$\\
  $2$ & \quad $r(0)$\\
\end{tabular}\\[0.5em]
\end{center}

Now, let $\liesu_2(\CC)$ and $\liegl_2(\CC)$ be the Lie algebras of $\SU_2(\CC)$, respectively of $\GL_2(\CC)$,
both equipped with the adjoint action of $\SU_2(\CC)$.
We have the following isomorphisms of $\SU_2(\CC)$-representations.\\[0.5em]
\begin{center}
\begin{tabular}{l|l}
  $i$ & \quad$\CC\otimes_{\R}\Lambda^i(\liegl_2(\CC)/\liesu_2(\CC))$\\\hline
  $0$ & \quad $s(0)$\\
  $1$ & \quad $s(2)$\\
  $2$ & \quad $s(2)$\\
  $3$ & \quad $s(0)$\\
\end{tabular}
\end{center}

\subsubsection{Non-Archimedean fields}\label{subsec:nonarchi}
%
%
%
An additional general reference for this subsection is \cite{BushnellHenniart}, particularly Chapters 8 and 13.

Let~$E$ be a $p$-adic field, let~$\fp$ be its maximal ideal.
Let~$A$ be the unique quaternion division algebra over~$E$.
A representation of $A^\times$ or of $\GL_2(E)$ is called \emph{smooth} if
every vector in the underlying vector space is fixed by some compact open subgroup.
We will assume that all our representations are smooth without repeating it.

Every irreducible representation of~$A^\times$ is finite-dimensional.
Every irreducible representation of $\GL_2(E)$ is of exactly one of the
following types:
\begin{itemize}[leftmargin=*]
  \item finite-dimensional representations: they are all of the form~$g\mapsto
    \phi\circ\det$ for a quasi-character~$\phi\colon E^\times \to \CC^\times$,
  \item irreducible principal series~$\PS_E(\automchar_1,\automchar_2)$,
    where~$\automchar_1,\automchar_2\colon E^\times \to \CC^\times$ are quasi-characters such
    that~$\automchar_1\automchar_2^{-1}\notin \{|\cdot|^{\pm 1}\}$,
  \item special representations,
  \item supercuspidal representations.
\end{itemize}

The automorphic induction~$\AI_E^L$ from~$\GL_1(L)$ to~$\GL_2(E)$ satisfies the following.
\begin{itemize}[leftmargin=*]
  \item An automorphic induction is never a special representation.
  \item Extension~$E\subset E\times E$: let~$\automchar_1,\automchar_2 \colon E^\times \to
    \CC^\times$ be quasi-characters.
    \begin{itemize}[leftmargin=*]
      \item $\AI_E^{E\times E}(\automchar_1,\automchar_2) = \PS_E(\automchar_1,\automchar_2)$
        if~$\automchar_1\automchar_2^{-1}\notin \{|\cdot|^{\pm 1}\}$.
      \item $\AI_E^{E\times E}(\automchar_1,\automchar_2) = \phi\circ \det$ for some
        quasi-character~$\phi\colon E^\times \to \CC^\times$ otherwise.
    \end{itemize}
  \item Quadratic field extension~$E\subset L$: let~$\sigma\in\Gal(L/E)$ be the
    non-trivial element, let~$\chi\colon \Gal(L/E)\to \CC^\times$ be the
    non-trivial quadratic character, and let~$\automchar\colon L^\times \to \CC^\times$
    be a quasi-character.
    \begin{itemize}[leftmargin=*]
      \item $\AI_E^L(\automchar)$ is supercuspidal if~$\automchar \neq \automchar^\sigma$.
      \item $\AI_E^L(\automchar) = \PS_E(\phi,\phi\chi)$ for some
        quasi-character~$\phi\colon E^\times \to \CC^\times$ otherwise.
    \end{itemize}
\end{itemize}

For $r\in \Z_{\geq 0}$, let
\[
  K_0(\fp^r)=\left\{\left(\begin{smallmatrix}a&b\\c&d\end{smallmatrix}\right): c\equiv 0\bmod \fp^r\right\}\subset \GL_2(\Z_E)
\]
and
\[
  K_1(\fp^r)=\left\{\left(\begin{smallmatrix}a&b\\c&d\end{smallmatrix}\right):
    d\equiv 1\bmod \fp^r,c\equiv 0\bmod \fp^r\right\}\subset K_0(\fp^r).
\]
If $\Pi$ is an irreducible representation of $\GL_2(E)$,
then its \emph{conductor} is $\fp^n$ where $n\in \Z_{\geq 0}\cup \{\infty\}$ is minimal
subject to the condition~$\Pi^{K_1(\fp^n)}\neq 0$.
\begin{lemma}\label{lem:dimfixedpoints}
Let $\Pi$ be an irreducible representation, let $\fp^n$ be its conductor, and let $i\in \Z_{\geq n}$.
Then the space $\Pi^{K_1(\fp^i)}$ of fixed points has dimension
$i-n+1$.
\end{lemma}
\begin{proof}
See \cite{Casselman}, see also \cite[Theorem~1.2.1 (ii)]{Schmidt}.
\end{proof}

In~$A^\times$, we define~$K_0(\fp^0)$ to be the unit group of the unique maximal
order in~$A^\times$. When we use the notation~$K_0(\fp^0)$, it will be clear
from the context whether it refers to the subgroup of~$A^\times$
or to~$\GL_2(\Z_E)$.

\begin{lemma}\label{lem:levelAI}
    Let~$L/E$ be an \'etale quadratic extension, let~$\delta$ be its discriminant,
    let~$\automchar\colon L^\times\to \CC^\times$ be a quasi-character such that
    $\AI_E^L(\automchar)$ is infinite-dimensional, and let $\fF$ be the conductor
    of $\automchar$. Then the conductor of $\AI_E^L(\automchar)$ is~$\delta \cdot
    \norm_{L/E}(\fF)$.
\end{lemma}
\begin{proof}
  See \cite[Table at the end of \S 1]{Schmidt}.
%
%
%
\end{proof}

The Jacquet--Langlands transfer~$\JL_A$ from~$A^\times$ to~$\GL_2(E)$ has the following property:
  an irreducible representation~$\automrep$ of~$\GL_2(E)$ is in the image of~$\JL_A$ if and only if~$\automrep$
    is special or supercuspidal.

%

\begin{lemma}\label{lem:localboundfin}
  There exists~$\kappa_f>0$ such that for every irreducible
  representation~$\automrep_f$ of~$\G(\adel_{F,f})$ we have
  \[
    \dim \automrep_f^{K_f} \le \kappa_f.
  \]
\end{lemma}
\begin{proof}
  Recall that we have~$K(\fN) \subset K_f$, so that for every irreducible
  representation~$\automrep_f$ of~$\G(\adel_{F,f})$ we
  have~$\dim \automrep_f^{K_f} \le \dim
  \automrep_f^{K(\fN)}$. Write~$K(\fN) = \prod_\fp K(\fp^{m_\fp})$ and~$\automrep_f =
  \otimes_\fp ' \automrep_\fp$, where $\otimes'$ denotes a restricted tensor product.
  Note that for every prime~$\fp$ we have~$\dim \automrep_\fp^{K(\fp^0)} \le 1$, and we
  have~$\dim \automrep_f^{K(\fN)} = \prod_\fp \dim \automrep_\fp^{K(\fp^{m_\fp})}$.
  Let~$\fp$ be such that~$m_\fp>0$.
  \begin{itemize}[leftmargin=*]
    \item If~$\fp$ is split in~$D$, then by Lemma \ref{lem:dimfixedpoints} we have
      $\dim\automrep_\fp^{K(\fp^{m_\fp})} \le
      \dim\automrep_\fp^{K_1(\fp^{2m_\fp})} \le 2m_{\fp}+1$, which depends only
      on~$K_f$.
    \item If~$\fp$ is ramified in~$D$, then consider the finite group~$G =
      D_\fp^\times/(F_\fp^\times K(\fp^{m_\fp}))$, and let~$d(G)$ be the maximal
      dimension of an irreducible representation of~$G$.
      Since~$\automrep_\fp$ is irreducible, $F_\fp^\times$ acts via a
      character~$\zeta$ on~$\automrep_\fp$.
      Therefore~$\automrep_\fp^{K(\fp^{m_\fp})}\zeta^{-1}$ is an irreducible
      representation of~$G$, hence of dimension at most~$d(G)$, which also
      depends only on~$K_f$.
  \end{itemize}
  This proves the existence of~$\kappa_f$.
\end{proof}

\subsection{\texorpdfstring{Automorphic representations of~$\GL_2$}{Automorphic representations of GL2}}\label{sec:automGL2}
For the duration of the subsection, let $\bfG$ be one of the algebraic groups $\G$ or $\GL_2$ over $F$.

For a general definition of an automorphic representation of $\bfG(\adel_F)$ and of cuspidality
see \cite[Definitions 6.3.5 and 6.5.1]{GetzHahn}. For our purposes, simpler definitions in special
cases will be sufficient.

A \emph{discrete automorphic representation} of $\bfG(\adel_F)$ is an irreducible subrepresentation of
$\rL^2(\bfG(F)\lquo \bfG(\adel_F)/\R_{>0})$, where
$\R_{>0}\subset \centre(\bfG(F_{\R}))=\prod_{v|\infty}\centre(\bfG(F_v))$ is embedded diagonally.
For the relation with automorphic representations in the usual sense
see \cite[Theorem 6.6.4]{GetzHahn}. Every automorphic representation of $\G(\adel_F)$ is discrete
and cuspidal, and every cuspidal automorphic representation of $\GL_2(\adel_F)$ can be identified
with a discrete automorphic representation, see \cite[Theorem 6.5.3]{GetzHahn}.

Let $\Pi$ be a discrete automorphic representation of $\bfG(\adel_F)$. Then $\Pi$ can be
written as a restricted tensor product $\Pi_\infty\otimes \Pi_f = \bigotimes'_v\Pi_v$
over the places $v$ of $F$, where $\Pi_{\infty}$ is a representation of $\bfG(F_{\R})$,
$\Pi_{f}$ is a representation of $\bfG(\adel_{F,f})$, and for each place $v$, the
factor $\Pi_v$ is a representation of $\bfG(F_v)$. The fixed point space
$$
\Big(\mathop{\bigotimes\nolimits^{\prime}}_{\fp\nmid \delta_D\fN}\Pi_{\fp}\Big)^{K_f}=\bigotimes_{\fp\nmid\delta_D\fN}\Pi_{\fp}^{K_1(\fp^0)}
$$
is $1$-dimensional, and admits an action of $\abstracthecke$, thus giving rise to a Hecke eigensystem.

If $\Pi$ is a discrete automorphic representation of $\bfG(\adel_F)$ and $\chi$
is a Hecke character of $F$, we write $\Pi\otimes \chi$ as shorthand for the automorphic
representation $\Pi\otimes_{\CC}(\chi\circ\nrd)$.

\begin{theorem}[Strong multiplicity $1$]\label{thm:MultOne}
  Every cuspidal automorphic representation occurs with multiplicity $1$ in 
  $\rL^2(\bfG(F)\lquo \bfG(\adel_F)/\R_{>0})$. Moreover,
  if $\automrep$ and $\automrep'$ are two cuspidal automorphic representations of $\bfG$
  such that $\automrep_v$ is isomorphic to $\automrep_v'$ for all but finitely
  many places $v$ of $F$, then $\automrep$ and $\automrep'$ are isomorphic.
\end{theorem}
\begin{proof}
  See  \cite{JL}, \cite[Theorem 5.1 (b), (c)]{Badulescu}, see also \cite[Theorems 11.4.3 and 11.7.2]{GetzHahn}.
\end{proof}

\begin{thm}[Automorphic induction]\label{thm:AI}
  Let~$L/F$ be a quadratic extension of number fields and let~$\sigma$ be the
  generator of its Galois group, and let~$\chi$ be the quadratic Hecke character
  corresponding to the extension~$L/F$ by class field theory.
  Let~$\automchar$ be a Hecke character of~$L$.
  The automorphic induction~$\AI_{F}^L(\automchar)=\otimes'_v \AI_{F_v}^{F_v\otimes_FL}(\automchar_v)$ is an automorphic
  representation of~$\GL_2(\adel_F)$.
  The representation~$\AI_{F}^L(\automchar)$ is cuspidal if and only
  if~$\automchar^\sigma\neq\automchar$.
  A cuspidal automorphic representation~$\automrep$ of~$\GL_2(\adel_F)$ is the automorphic
  induction of a Hecke character of~$L$ if and only if~$\automrep \otimes \chi\cong
  \automrep$.
\end{thm}
\begin{proof}
  See \cite{BaseChange}, see also \cite[Theorem 13.4.2]{GetzHahn}.
\end{proof}

\begin{thm}[Jacquet--Langlands correspondence]\label{thm:JL}
  There exists a unique injection~$\JL_D$ from the class of isomorphism classes of
  automorphic representations of~$\G(\adel_F)$ to the class of isomorphism classes
  of discrete automorphic 
  representations of~$\GL_2(\adel_F)$ satisfying~$\JL_{D_v}(\automrep_v) = \JL_D(\automrep)_v$ for
  all places~$v$.
  The representation $\JL_D(\automrep)$ is cuspidal if and only if $\automrep$ is infinite-dimensional.
  When restricting to the class of infinite-dimensional representations,
  the image is the class of cuspidal automorphic representations $\automrep$
  such that for all places $v$, the local representation $\automrep_v$ is in the
  image of~$\JL_{D_v}$.
\end{thm}
\begin{proof}
  See \cite{JL}, see also \cite[Theorem 19.4.3]{GetzHahn}.
\end{proof}

\begin{cor}\label{cor:JLAI}
  Keep the notation as in Theorem~\ref{thm:AI}.
  Let~$\automrep$ be an automorphic representation of~$\G$ such
  that~$\automrep\otimes\chi \cong \automrep$.
  Then $\automrep$ is infinite-dimensional, and there exists a Hecke character~$\automchar$ of~$L$ such that~$\JL_D(\automrep) =
  \AI_{F}^L(\automchar)$.
\end{cor}
\begin{proof}
  If $\automrep$ is finite-dimensional, then it is $1$-dimensional, and the hypothesis is never satisfied,
  therefore $\automrep$ is infinite-dimensional.
  By Multiplicity $1$, Theorem~\ref{thm:MultOne}, the Jacquet--Langlands correspondence, being local-global
  compatible, is also compatible with twisting. Applying Theorems~\ref{thm:AI}
  and~\ref{thm:JL} therefore gives the statement.
\end{proof}



Let $\lieg$ and $\liek$ denote the Lie algebras of $G_{\infty}$ and $K_\infty$,
respectively, both equipped with the adjoint action of $K_\infty^+$.
\begin{thm}[Matsushima's formula]\label{thm:Matsushima}
  We have
  \[
    \Omega^i(\cY)_{\CC}\cong \bigoplus_{\automrep}
    \Hom_{K_\infty^+}(\CC\otimes_{\R}\Lambda^i(\lieg/\liek), \automrep_\infty) 
      \otimes_{\CC} \automrep_f^{K_f},
  \]
  where the sum ranges over automorphic representations~$\automrep$ of~$\G(\adel_F)$
  such that~$\automrep_\infty$ is trivial on~$Z_\infty$.
  Moreover, the isomorphism is equivariant with respect to:
  \begin{itemize}[leftmargin=*]
    \item the Laplace operator on the left hand side and the Casimir operator
      acting on~$\automrep_\infty$ on the right hand side, and
    \item the action of Hecke operators, acting on~$\automrep_f^{K_f}$ on the right
      hand side.
  \end{itemize}
\end{thm}
\begin{proof}
  This is a well-known ad\'elic reformulation of the classical Matsushima formula \cite{Matsushima}, see \cite[\S 1.2]{BergeronClozel}.
  The equivariance with respect to Hecke operators follows from the canonicity
  of the correspondence in ibid.
  Note that the statement simplifies in our situation due to Multiplicity $1$, Theorem \ref{thm:MultOne}.
\end{proof}
\begin{definition}\label{def:vects}
  Given a collection $\vecti=(i_v)_v$ of non-negative integers indexed
  by the infinite places $v$ of $F$, we define
  \[
    \Omega^{\vecti}(\cY)_{\CC} = \bigoplus_{\automrep}
    \Hom_{K_\infty^+}\Big(\bigotimes_{v}\CC\otimes_{\R}\Lambda^{i_v}(\lieg_v/\liek_v), \automrep_\infty\Big) 
    \otimes_{\CC} \automrep_f^{K_f}.
  \]
  Moreover, given a collection $\vectlambda=(\lambda_v)_v$ or real numbers
  indexed by the infinite places $v$ of $F$, define
  $\Omega_{\Delta=\vectlambda}^{\vecti}(\cY)_{\CC}\subset \Omega^{\vecti}(\cY)_{\CC}$
  to be the subspace consisting of Laplace eigenvectors
  on which the Laplace operator at each infinite place $v$ has
  eigenvalue $\lambda_v$. Let
  $\harmon^{\vecti}(\cY)_{\CC}=\Omega^{\vecti}_{\Delta=\vectzero}(\cY)_{\CC}$ be the space
  of harmonic forms in $\Omega^{\vecti}(\cY)_{\CC}$, where $\vectzero$ denotes the
  zero vector.
\end{definition}
\begin{corollary}\label{cor:Matsushima}
  For every $i\in \Z_{\geq 0}$ and $\lambda\in\R$, we have
  \[
    \Omega^i_{\Delta=\lambda}(\cY)_{\CC} = \bigoplus_{\vectlambda}\bigoplus_{\vecti}\Omega^{\vecti}_{\Delta=\vectlambda}(\cY)_{\CC},\quad\text{ and }\quad
    \harmon^i(\cY)_{\CC} = \bigoplus_{\vecti}\harmon^{\vecti}(\cY)_{\CC},
  \]
  where the sums run over all collections $\vectlambda=(\lambda_v)_v$ satisfying $\sum_v \lambda_v=\lambda$, respectively
  all collections $\vecti=(i_v)_v$ satisfying $\sum_vi_v=i$.
\end{corollary}

\subsection{Self-twist conditions}\label{sec:selfTwistConds}
In this subsection we prove Theorem \ref{thm:IntroProtoMain} for
representation equivalence, $\Omega^\bullet$-, $\Omega^0$-, and
$\cH^\bullet$-isospectrality, and
Theorem~\ref{thm:IntroRegLL} parts~\ref{item:IntroRegLLRegConst}
and~\ref{item:IntroRegLLProd}.
\subsubsection{Representation equivalence}
\begin{prop}\label{prop:L2psi}
  Let~$\chi$ be an order~$2$ Hecke character, let~$L/F$ be the corresponding
  quadratic extension, and let~$\sigma$ denote the non-trivial
  automorphism of~$L/F$. Let $V$ be a $K_\infty^+$-representation.
  Then every Hecke eigensystem over~$\CC$ in the
  module~$M=\Hom_{K_{\infty}^+}(V,\rL^2(\G(F)\lquo\G(\adel_{F})/Z_{\infty}
  K_f))$ that admits a self-twist by~$\chi$ is attached to some automorphic
  representation~$\automrep = \automrep_\infty\otimes\automrep_f$ of
  $\G(\adel_F)$ such that $\JL_D(\automrep)=\AI_{F}^{L}(\automchar)$ for a
  unitary Hecke character~$\automchar$ of~$L$ satisfying all of the following:
  \begin{enumerate}[leftmargin=*,label=\upshape{(\arabic*)}]
    \item\label{item:JLpsi} for every place~$v$ of~$F$ that ramifies in~$D$, there is a single
      place~$w$ of~$L$ above~$v$, and we have~$\automchar_w^\sigma \neq \automchar_w$;
    \item\label{item:centralchar} the parameters~$(k_w,s_w)$ of the
      character~$\automchar$ satisfy the following conditions:
      \begin{itemize}[leftmargin=*] 
        \item for every real place~$v$ of~$F$ that extends to a complex place~$w$
          of~$L$, we have~$s_w = 0$ and $k_w$ is odd;
        \item for every real place~$v$ of~$F$ that extends to two real places~$w,w'$
          of~$L$, we have~$s_w+s_{w'}=0$ and $k_w+k_{w'}=0\mod{2}$;
        \item for every complex place~$v$ of~$F$ that extends to two complex
          places~$w,w'$ of~$L$, we have~$s_w+s_{w'}=0$ and~$k_w+k_{w'}=0$.
  \end{itemize}
  \end{enumerate}
  Each such Hecke eigensystem has multiplicity~$\dim \automrep_f^{K_f}\cdot\dim\Hom_{K_{\infty}^+}(V,\automrep_{\infty})$ (which may be $0$).
\end{prop}
\begin{proof}
  Hecke eigensystems over~$\CC$ in~$M$ are attached to automorphic representations~$\automrep =
  \automrep_\infty\otimes\automrep_f$ of~$\G(\adel_F)$ such that the central character
  of~$\automrep_\infty$ vanishes on~$Z_\infty$, and each such eigensystem has
  multiplicity~$\dim \automrep_f^{K_f}\cdot\dim\Hom_{K_{\infty}^+}(V,\automrep_{\infty})$.
  Hecke eigensystems over~$\CC$
  in~$\rL^2(\G(F)\lquo\G(\adel_{F})/Z_{\infty} K_f)$ that admit a self-twist
  by~$\chi$ are exactly those attached to automorphic representations~$\automrep =
  \automrep_\infty\otimes\automrep_f$ of~$\G(\adel_F)$ such that $\automrep\otimes\chi \cong \automrep$
  and the central character of~$\automrep_\infty$ vanishes on~$Z_\infty$.
  By Corollary~\ref{cor:JLAI}, such~$\automrep$ are exactly the ones
  satisfying
  \begin{itemize}
    \item $\automrep$ is infinite-dimensional;
    \item $\JL_D(\automrep) = \AI_{L/F}(\automchar)$ for some Hecke character~$\automchar$ of~$L$;
    \item the central character of~$\automrep_\infty$ vanishes on~$Z_\infty$.
  \end{itemize}
  By Theorem~\ref{thm:JL} and Theorem~\ref{thm:AI}, the first two conditions are
  equivalent to the existence of a Hecke character $\automchar$ of $L$ such that we have
  $\JL_D(\automrep) = \AI_{F}^L(\automchar)$ and~$\automchar^\sigma \neq \automchar$.
  Let~$\automchar$ be a Hecke character of~$L$ such that~$\automchar^\sigma\neq\automchar$.
  By Theorem~\ref{thm:JL} there exists~$\automrep$ such
  that~$\AI_F^L(\automchar)=\JL_D(\automrep)$ if and only if~$\AI_F^L(\automchar)$ is discrete
  series at the infinite places that are ramified in $D$,
  and special or supercuspidal at the finite places that are ramified in $D$.
  Let~$v$ be a place of~$F$ that ramifies in~$D$.
  If~$v$ is an infinite place: by Subsection~\ref{subsec:archi},
  the local component~$\AI_F^L(\automchar)_v$ is a discrete series representation if
  and only if~$v$ extends to a complex place~$w$ in~$L$ and~$k_w \neq 0$,
  equivalently there is a single place~$w$ of~$L$ above~$v$
  and~$\automchar_w^\sigma\neq\automchar_w$.
  If~$v$ is a finite place: by Subsection~\ref{subsec:nonarchi}, the
  local component~$\AI_F^L(\automchar)_v$ is special or supercuspidal if and only
  there is a single place~$w$ of~$L$ above~$v$
  and~$\automchar_w^\sigma \neq \automchar_w$. This proves that condition~\ref{item:JLpsi}
    of the conclusion of the proposition holds.
  By Subsection~\ref{subsec:archi}, the central character of~$\AI_F^L(\automchar)$
  vanishing on~$Z_\infty$ is equivalent to
  condition~\ref{item:centralchar} in the conclusion.
  Since $\automchar$ is trivial on $L^\times$, all the $s_w$ have the same real part.
  Therefore condition \ref{item:centralchar} implies that this real part is $0$,
  so that $\automchar$ is unitary.
%
\end{proof}

\begin{definition}\label{def:repequivObstruction}
  If $L/F$ is a quadratic extension, then an \emph{$\rL^2$-shady character} of~$L$
  is a unitary Hecke character $\automchar$ of $L$ such that $L$ and $\automchar$ have all of the following properties: 
  \begin{itemize}[leftmargin=*]
    \item the field $L$ and the character~$\automchar$ satisfy the conditions~\ref{item:JLpsi}
      and~\ref{item:centralchar} in Proposition~\ref{prop:L2psi}, so that
      in particular, by Theorems~\ref{thm:AI} and~\ref{thm:JL},
      there exists a unique automorphic representation~$\automrep=\automrep_{\infty}\otimes\automrep_f$
      of~$\G(\adel_F)$ satisfying~$\JL_D(\automrep)=\AI_F^L(\automchar)$;
    \item we have~$\automrep_f^{K_f}\neq 0$.
  \end{itemize}
\end{definition}
%
%
The following is Theorem~\ref{thm:IntroProtoMain} for representation equivalence.

\begin{thm}\label{thm:repEquivPsi}
  Let~$c\in C$. Then at least one of the following statements is true:
  \begin{enumerate}[leftmargin=*,label={\upshape(\roman*)}]
    \item there exist a character~$\chi\in \Cisodual\setminus c^\perp$,
      with corresponding quadratic extension~$L/F$, 
      and an $\rL^2$-shady character of~$L$;
    \item
      for all $b\in C$ the groups~$\Gamma_{b}$
    and~$\Gamma_{cb}$ are representation equivalent.
  \end{enumerate}
\end{thm}
\begin{proof}
  This is an immediate consequence of Theorem~\ref{thm:repEquiv} and
  Proposition~\ref{prop:L2psi}.
\end{proof}

\begin{rmk}
  In Proposition~\ref{prop:L2psi}, condition~\ref{item:JLpsi} at a real
  place~$v$ of~$F$ is equivalent to the following: if~$v$ ramifies in~$D$,
  then~$v$ extends to a complex place~$w$ of~$L$ and we have~$k_w\neq 0$.
\end{rmk}

It will sometimes be useful to reformulate the conditions of
Proposition~\ref{prop:L2psi} as follows.

\begin{lemma}\label{lem:shadyGaloisaction}
  Let~$L/F$ be a quadratic extension, let~$\sigma$ denote the non-trivial
  automorphism of~$L/F$, and let~$\automchar$ be a Hecke character of~$L$.
  Then $\Psi^\sigma\Psi$ has finite order if and only if all of the following
  conditions hold:
  \begin{itemize}[leftmargin=*]
    \item for every real place~$v$ of~$F$ that extends to a complex place~$w$
      of~$L$, we have~$s_w = 0$;
    \item for every real place~$v$ of~$F$ that extends to two real places~$w,w'$
      of~$L$, we have~$s_w+s_{w'}=0$;
    \item for every complex place~$v$ of~$F$ that extends to two complex
      places~$w,w'$ of~$L$, we have~$s_w+s_{w'}=0$ and~$k_w+k_{w'}=0$.
  \end{itemize}
\end{lemma}
\begin{proof}
  At a place of~$F$ that extends to two places of~$L$, the action of~$\sigma$
  on~$\Psi$ swaps the parameters~$(k,s)$. At a real place~$v$ of~$F$ that
  extends to a complex place~$w$ of~$L$, the action of~$\sigma$ on~$\Psi$
  negates~$k_w$ and leaves~$s_w$ unchanged. Finally, a Hecke character of~$L$
  has finite order if and only if all its~$s$ parameters are~$0$ and all
  its~$k$ parameters at complex places are~$0$. Putting these together gives the
  lemma.
\end{proof}

\begin{remark}
  When at least one real place of~$F$ ramifies in~$L$, the Hecke characters
  appearing in Theorem~\ref{thm:repEquivPsi} are \emph{partially algebraic Hecke
  characters} in the sense of~\cite[Section 5.5]{MolinPage}. In fact, they must
  come from the construction given in the proof of \emph{ibid.}, Proposition 41.
\end{remark}

By omitting conditions from Theorem~\ref{thm:repEquivPsi}, we recover previously known results.

\begin{corollary}\label{cor:repEquivNoPsi}
  Let~$c\in C$. Then at least one of the following statements is true:
  \begin{enumerate}[leftmargin=*,label={\upshape(\roman*)}]
    \item there exists a character~$\chi\in \Cisodual$, with
      corresponding quadratic extension~$L/F$,
  such that
  \begin{enumerate}[leftmargin=*,label=\upshape{(\alph*)}]
    \item\label{item:odd} $\chi(c) = -1$, i.e. $\chi\not\in c^\perp$, and
    \item\label{item:singlePlace} for every place~$v$ of~$F$ that ramifies in~$D$, there is a single
      place of~$L$ above~$v$;
  \end{enumerate}
    \item for all $b\in C$ the groups~$\Gamma_{b}$ and~$\Gamma_{cb}$ are
  representation equivalent.
  \end{enumerate}
\end{corollary}

\begin{remark}
  Corollary~\ref{cor:repEquivNoPsi} seems much weaker than
  Theorem~\ref{thm:repEquivPsi}, but is sufficient to recover previously known
  results.
  It would be interesting to know how much weaker it actually is, specifically
  whether, given $L/F$ as in Corollary~\ref{cor:repEquivNoPsi},
  there always exists a $\automchar$ as in Theorem \ref{thm:repEquivPsi}.
\end{remark}

\begin{cor}\label{cor:LinowitzVoight}
  Let~$\fN$ be a non-zero ideal of~$\Z_F$ coprime to~$\delta_D$, and let~$K_f =
  \prod_{\fp^i\parallel\fN}K_0(\fp^i)$.
  Let~$c\in C$. Then at least one of the following statements is true:
  \begin{enumerate}[leftmargin=*,label={\upshape(\roman*)}]
    \item we have that
      \begin{enumerate}[leftmargin=*,label=\upshape{(\alph*)}]
        \item $D$ is unramified at all finite places of~$F$, and
        \item there exists a quadratic extension~$L/F$ that is
          ramified at exactly the same set of real places of~$F$
          as~$D$, such that all primes of~$F$
          dividing~$\fN$ with odd exponent are split in~$L$,
          and such that all primes of~$F$ 
          whose class in~$C_{\iso}$ is~$\bar{c}$ are inert in~$L$;
      \end{enumerate}
    \item for all $b\in C$ the groups~$\Gamma_{b}$ and~$\Gamma_{cb}$
      are representation equivalent.
  \end{enumerate}
\end{cor}
\begin{proof}
  By Equation~(\ref{eq:Ciso}) and a local computation \cite[eqns. (23.2.4) and~(23.2.8), Proposition 23.4.14]{Voight},
  we have~$C_{\iso} = \Cl_F(\cV_{\Hamil}(D))/\langle\fa^2, \fp \text{ ramified in }D,
  \fp^e\|\fN\rangle$.
  Assume that for some $b\in C$ the groups~$\Gamma_{b}$ and~$\Gamma_{cb}$ are not
  representation equivalent.
  Let~$\fp$ be a finite place of~$F$ that ramifies in~$D$. Then the class of~$\fp$ is
  trivial in~$C_{\iso}$, so~$\fp$ splits in~$L$; but this
  contradicts Corollary~\ref{cor:repEquivNoPsi}~\ref{item:singlePlace}. Therefore, no
  finite place ramifies in~$D$.
  Let~$\chi$ and~$L/F$ be as in Corollary~\ref{cor:repEquivNoPsi}. By the
  expression of~$C_{\iso}$ as a class group, $L/F$ is unramified at all finite
  places of~$F$, split at all prime ideals dividing~$\fN$ with odd exponent, and the
  set of real places of~$F$ that ramify in~$L$ is a subset of those that ramify
  in~$D$.
  Corollary~\ref{cor:repEquivNoPsi}~\ref{item:singlePlace} gives the other inclusion
  between those sets of ramified real places.
  Let~$\fp$ be a finite place of~$F$ whose class in~$C_{\iso}$ equals~$\bar{c}$.
  Then~$\chi(\fp) \neq 1$ by Corollary \ref{cor:repEquivNoPsi}~\ref{item:odd},
  so that~$\fp$ is inert in~$L$.
\end{proof}

\begin{rmk}
  The groups of the form~$\Gamma_c$ for~$c\in C$ are exactly the ones of the form~$\cO^\times$,
  where~$\cO$ is an Eichler order of level~$\fN$.
  Corollary~\ref{cor:LinowitzVoight} is essentially the criterion used
  in~\cite{VoightLinowitz} to investigate isospectrality of a large number of
  pairs of hyperbolic orbifolds of dimension~$2$ and~$3$,
  using the selectivity method~\cite[Theorem~2.17 and Theorem~2.19]{VoightLinowitz}.
\end{rmk}

\begin{cor}\label{cor:Rajan}
  Assume that there is a maximal ideal~$\fp$ of $\Z_F$ that is ramified in~$D$, and that
  there exists an element~$g\in \normaliser(K_f)$ such that the valuation~$\ord_\fq(\nrd(g))$ is
  even for every finite place~$\fq\neq \fp$ and such that~$\ord_{\fp}(\nrd(g))$ is odd.
  Then for all~$c,c'\in C$ the groups~$\Gamma_{c}$
  and~$\Gamma_{c'}$ are representation equivalent.
\end{cor}
\begin{proof}
  Let $c$, $c'\in C$ and assume that~$\Gamma_{c}$ and~$\Gamma_{c'}$ are not
  representation equivalent, and let~$L/F$ be as in
  Corollary~\ref{cor:repEquivNoPsi}.
  By definition of~$C_{\iso}$, the class of~$\nrd(g)$ in~$C_{\iso}$ is
  trivial. On the other hand, since~$C_{\iso}$ has exponent dividing~$2$, the
  valuation assumptions imply that the class of~$\nrd(g)$ equals the class
  of~$\fp$. This implies that~$\fp$ splits in~$L$, contradicting
  Corollary~\ref{cor:repEquivNoPsi}~\ref{item:singlePlace}.
  Therefore, the groups~$\Gamma_{c}$
  and~$\Gamma_{c'}$ are representation equivalent.
\end{proof}

\begin{rmk}\label{rk:Rajan}
  Corollary~\ref{cor:Rajan} is essentially the same criterion as
  in~\cite[Section~3]{Rajan}: the condition~(H3) in that paper (with~$v_0=\fp$) implies
  that we can take~$g$ to be a uniformiser of~$D_\fp$ at~$v_0$ and~$g_v=1$
  for~$v\neq v_0$. One superficial difference is that the representation equivalent groups
  in~\cite{Rajan} are arithmetic subgroups of~$\SL_1(D)$, whereas ours are
  subgroups of~$\GL_1(D)$; by choosing appropriate levels $K_f$, our method can be adapted to
  the~$\SL_1(D)$ setting, but we chose to stick to~$\GL_1(D)$ for simplicity of
  the exposition. The similarity is not surprising: our Hecke eigenvalue systems
  with a self-twist exactly correspond to the endoscopic representations studied
  in~\cite{LabesseLanglands} and playing an important role in~\cite{Rajan}. However,
  our method proves representation equivalence
  without an analysis of the Labesse--Langlands multiplicity formula.
\end{rmk}

\subsubsection{Differential forms}
\begin{proposition}\label{prop:diffFormsPsi}
  Let~$\chi$ be an order~$2$ Hecke character, with corresponding
  quadratic extension~$L/F$, let~$\sigma$ denote the non-trivial
  automorphism of~$L/F$, and let $\vecti=(i_v)_v$ and $\vectlambda=(\lambda_v)_v$
  be as in Definition \ref{def:vects}. Then the Hecke eigensystems over~$\CC$
  in~$\Omega^{\vecti}_{\Delta=\vectlambda}(\cY)_{\CC}$ that admit a self-twist
  by~$\chi$ are the Hecke eigenvalue systems attached to an automorphic representation~$\automrep= \automrep_\infty\otimes\automrep_f$
  of $\G(\adel_F)$ such that $\JL_D(\automrep)=\AI_{F}^{L}(\automchar)$ for some unitary Hecke character~$\automchar$ of~$L$ satisfying
  all of the following conditions:
  \begin{enumerate}[leftmargin=*,label=\upshape{(\arabic*)}]
    \item\label{item:diffFormsPsi1} for every place~$v$ of~$F$ that ramifies in~$D$, there
      is a single place~$w$ of~$L$ above~$v$, and we have~$\automchar_w^\sigma \neq \automchar_w$;
    \item for every real place $v$ of $F$ that extends to a complex place $w$ of $L$
      we have
      \begin{itemize}
        \item $i_v=0$ if~$v$ is ramified in~$D$;
        \item $i_v=1$ otherwise;
      \end{itemize}
      $\lambda_v=0$, and $\automchar_w = \automchar_{\CC}(k,0)$ with $k\in \{\pm 1\}$;
    \item for every real place of $F$ that extends to two real places $w$,
      $w'$ of $L$ we have $i_v\in \{0,1,2\}$, $\automchar_{w} = \automchar_{\R}(k,s)$ and
      $\automchar_{w'} = \automchar_{\R}(k',s')$ where $k\equiv k'\mod 2$, and $s=it=-s'$ for $t\in \R$
      satisfying $\tfrac14+t^2=\lambda_v$;
    \item\label{item:diffFormsPsi4} for every complex place $v$ of $F$, extending to places $w$ and $w'$ of $L$,
      we have $\automchar_w=\automchar_{\CC}(k,s)$ and $\automchar_{w'} = \automchar_{\CC}(k',s')$ where either
      \begin{itemize}
        \item $i_v\in \{0,1,2,3\}$, $k=k'=0$, and $s=it=-s'$ with $t\in \R$ satisfying $1+4t^2=\lambda_v$; or
        \item $i_v\in \{1,2\}$, $k=-k'\in \{\pm1\}$, and $s=it=-s'$ with $t\in \R$ satisfying $4t^2=\lambda_v$;
      \end{itemize}
  \end{enumerate}
  each such Hecke eigenvalue system occurring with multiplicity $2^r\dim \automrep_f^{K_f}$,
  where $r$ is the number of real places $v$ of $F$ for which one has $i_v=1$.
\end{proposition}
\begin{proof}
  We apply Proposition \ref{prop:L2psi} with $V=\Omega^{\vecti}_{\Delta=\vectlambda}(\cY)_{\CC}$.
  Let $\automchar$ and $\automrep$ be as in Proposition \ref{prop:L2psi}.
  For every infinite place $w$ of $L$, write $s_w=it_w$ with~$t_w\in\R$.

  Write $K_{\infty}^+=\prod_v K_v$, with the product running over the infinite places of $F$,
  and for each such place $v$, let $H_v=\Hom_{K_v}(\Lambda^{i_v}(\lieg_v/\liek_v),\automrep_v)$.

  Let $v$ be a real place of $F$ that extends to a complex place $w$ of $L$.
  Then we have $\AI_F^L(\automchar)_v=\DS(1+|k_w|,0)$. By the local computation in
  Section \ref{subsec:archi} the space $H_v$ is non-trivial if and only if
  either
  \begin{itemize}
    \item the place $v$ is unramified in $D$, $i_v=1$, $k_w\in \{\pm 1\}$, and $\lambda_v=0$,
      in which case $\dim H_v=2$; or
    \item the place $v$ is ramified in $D$, $i_v=0$, $k_w\in \{\pm 1\}$, and $\lambda_v=0$,
      in which case $\dim H_v=1$.
  \end{itemize}

  Let $v$ be a real place of $F$ that extends to two real places $w$, $w'$ of $L$.
  Then we necessarily have $\AI_F^L(\automchar)_v=\PS(\automchar_w,\automchar_{w'})$. Moreover, 
  the space $H_v$ is non-trivial if and only if $i_v\in\{0,1,2\}$ and $\lambda_v=\tfrac14+t_w^2$.
  In this case, $\dim H_v=2$ if $i_v=1$, and $\dim H_v=1$ otherwise.

  Let $v$ be a complex place of $F$, extending to two places $w$ and $w'$ of $L$. 
  Then we necessarily have $\AI_F^L(\automchar)_v=\PS(\automchar_w,\automchar_{w'})$. Moreover, 
  the space $H_v$ is non-trivial if and only if one of the following holds:
  \begin{itemize}
    \item $k_w=0$, $i_v\in \{0,1,2,3\}$, $\lambda_v=1+4t_w^2$;
    \item $k_w\in \{\pm1\}$, $i_v\in \{1,2\}$, $\lambda_v=4t_w^2$.
  \end{itemize}
  When $H_v$ is non-trivial, we have $\dim H_v=1$.
  Taking the tensor product of the spaces $H_v$ over all places $v\mid\infty$ gives the result.
\end{proof}

\begin{definition}
  If $L/F$ is a quadratic extension, and $\vecti$ and $\vectlambda$ are as in
  Proposition~\ref{prop:diffFormsPsi}, then an \emph{$(\Omega^{\vecti}_{\Delta=\vectlambda})$-shady character}
  of~$L$ is a unitary Hecke character $\automchar$ of~$L$ such that $L$ and $\automchar$ have
  all of the following properties:
  \begin{itemize}[leftmargin=*]
    \item the field $L$ and the character~$\automchar$ satisfy the conditions~\ref{item:diffFormsPsi1}--\ref{item:diffFormsPsi4}
      in Proposition~\ref{prop:diffFormsPsi}, so that
      in particular, by Theorems~\ref{thm:AI} and~\ref{thm:JL},
      there exists a unique automorphic representation~$\automrep=\automrep_{\infty}\otimes\automrep_f$
      of~$\G(\adel_F)$ satisfying~$\JL_D(\automrep)=\AI_F^L(\automchar)$;
    \item we have $\automrep_f^{K_f}\neq 0$.
  \end{itemize}
\end{definition}

We now prove a general version of Theorem~\ref{thm:IntroProtoMain}, which we
will then specialise.

\begin{theorem}\label{thm:isospectralPsi}
  Let~$c\in C$, let $i\in \Z_{\geq 0}$, and let
  $\lambda\in \R_{\geq 0}$. Then exactly one of the following statements is true:
  \begin{enumerate}[leftmargin=*,label={\upshape(\roman*)}] 
    \item\label{item:isospectralPsiObstruction} there exist a character~$\chi\in \Cisodual\setminus c^\perp$, with
      corresponding quadratic extension~$L/F$,
      collections $\vectlambda=(\lambda_v)_v$ and $\vecti=(i_v)_v$ of real numbers, 
      respectively non-negative integers, indexed by the infinite places $v$ of $F$,
      such that~$\sum_v \lambda_v = \lambda$ and~$\sum_v i_v = i$,
      and an $(\Omega^{\vecti}_{\Delta=\vectlambda})$-shady character of $L$;
    \item\label{item:isospectralPsiConclusion} there exists $T\in \abstracthecke[W]_c$ inducing, for all $b\in C$, an
      isomorphism of $\abstracthecke_1$-modules
      $$
        T\colon \Omega^i_{\Delta=\lambda}(Y_{b})\to \Omega^i_{\Delta=\lambda}(Y_{cb}).
      $$
  \end{enumerate}
\end{theorem}

\begin{proof}
  This is an immediate consequence of Theorem~\ref{thm:HeckeIsosp}~\ref{item:lambdaiso},
  Corollary \ref{cor:Matsushima}, and Proposition \ref{prop:diffFormsPsi}.
\end{proof}

\begin{definition}
  If $L/F$ is a quadratic extension, with non-trivial automorphism~$\sigma$,
  then an \emph{$\Omega^{\bullet}$-shady character} of $L$ is a unitary Hecke
  character $\automchar$ of $L$ such that~$L$ and $\automchar$ have all of the following properties:
  \begin{enumerate}[leftmargin=*]
    \item for every place~$v$ of~$F$ that ramifies in~$D$, there
      is a single place~$w$ of~$L$ above~$v$, and we have~$\automchar_w^\sigma \neq \automchar_w$;
    \item for every real place $v$ of $F$ that
      extends to a complex place $w$ of $L$ we have~$\automchar_w = \automchar_{\CC}(k,0)$
      with $k\in \{\pm 1\}$;
    \item for every real place of $F$ that extends to two real places $w$,
      $w'$ of $L$ we have $\automchar_{w} = \automchar_{\R}(k,s)$ and
      $\automchar_{w'} = \automchar_{\R}(k',s')$ where $k\equiv k'\mod 2$ and $s=-s'\in i\R$;
    \item for every complex place $v$ of $F$, extending to places $w$ and $w'$ of $L$,
      we have $\automchar_w=\automchar_{\CC}(k,s)$ and $\automchar_{w'} = \automchar_{\CC}(k',s')$ where
      $k=-k'\in \{-1,0,1\}$ and $s=-s'\in i\R$;
    \item the unique automorphic representation~$\automrep$ of~$\G(\adel_F)$ such
      that $\JL_D(\automrep) = \AI_F^L(\automchar)$ satisfies~$\automrep_f^{K_f}\neq 0$.
  \end{enumerate}
\end{definition}

The following is Theorem~\ref{thm:IntroProtoMain}
for~$\Omega^\bullet$-isospectrality.

\begin{corollary}\label{cor:allisospectralPsi}
  Let~$c\in C$. Then at least one of the following statements is true:
  \begin{enumerate}[leftmargin=*,label={\upshape(\roman*)}] 
    \item\label{item:allisospectralPsiObstruction} there exist a character~$\chi\in \Cisodual\setminus c^\perp$,
      with corresponding quadratic extension~$L/F$ with non-trivial
      automorphism~$\sigma$, and an $\Omega^\bullet$-shady character of~$L$;
    \item\label{item:allisospectralPsiConclusion} for all $b\in C$ the manifolds
      $Y_{b}$ and $Y_{cb}$ are $i$-isospectral for all~$i\ge 0$.
  \end{enumerate}
\end{corollary}
\begin{proof}
  Suppose that part \ref{item:allisospectralPsiObstruction} does not hold. Then nor
  does part \ref{item:isospectralPsiObstruction} of Theorem \ref{thm:isospectralPsi} for any $i$ and $\lambda$.
    Thus part \ref{item:isospectralPsiConclusion} of Theorem \ref{thm:isospectralPsi} holds 
    for every $i$ and $\lambda$, whence we deduce that part \ref{item:allisospectralPsiConclusion}
    of the corollary holds.
\end{proof}

\begin{definition}
  If $L/F$ is a quadratic extension, with non-trivial automorphism~$\sigma$,
  then an \emph{$\Omega^0$-shady character} of $L$ is a unitary Hecke
  character $\automchar$ of $L$ such that~$L$ and $\automchar$ have all of the following properties:
  \begin{enumerate}[leftmargin=*]
    \item for every place~$v$ of~$F$ that ramifies in~$D$, there
      is a single place~$w$ of~$L$ above~$v$, and we have~$\automchar_w^\sigma \neq \automchar_w$;
    \item for every real place $v$ of $F$ that extends to a complex place $w$ of
      $L$, the algebra~$D$ is ramified at~$v$ and
      $\automchar_w = \automchar_{\CC}(k,0)$ with $k\in \{\pm 1\}$;
    \item for every real place of $F$ that extends to two real places $w$,
      $w'$ of $L$ we have $\automchar_{w} = \automchar_{\R}(k,s)$ and
      $\automchar_{w'} = \automchar_{\R}(k',s')$ where $k\equiv k'\mod 2$ and $s=-s'\in i\R$;
    \item for every complex place $v$ of $F$, extending to places $w$ and $w'$ of $L$,
      we have $\automchar_w=\automchar_{\CC}(0,s)$ and $\automchar_{w'} = \automchar_{\CC}(0,s')$
      where~$s=-s'\in i\R$;
    \item the unique automorphic representation~$\automrep$ of~$\G(\adel_F)$ such
      that $\JL_D(\automrep) = \AI_F^L(\automchar)$ satisfies~$\automrep_f^{K_f}\neq 0$.
  \end{enumerate}
\end{definition}

The following is Theorem~\ref{thm:IntroProtoMain} for~$\Omega^0$-isospectrality.

\begin{corollary}\label{cor:zeroisospectralPsi}
  Let~$c\in C$. Then at least one of the following statements is true:
  \begin{enumerate}[leftmargin=*,label={\upshape(\roman*)}] 
    \item there exist a character~$\chi\in \Cisodual\setminus c^\perp$, with
      corresponding quadratic extension~$L/F$,
      and an $\Omega^0$-shady character~$\automchar$ of~$L$;
    \item for all $b\in C$ the manifolds $Y_{b}$ and $Y_{cb}$ are $0$-isospectral.
  \end{enumerate}
\end{corollary}
\begin{proof}
  The proof is completely analogous to that of Corollary \ref{cor:allisospectralPsi}.
\end{proof}
%
%
\subsubsection{Cohomology}
\begin{proposition}\label{prop:harmonPsi}
  Let~$\chi$ be an order~$2$ Hecke character, with corresponding quadratic
  extension~$L/F$, let~$\sigma$ denote the non-trivial automorphism
  of~$L/F$, and let $\vecti=(i_v)_v$ be as in Definition \ref{def:vects}.
  Then the Hecke eigensystems over~$\CC$ in~$\harmon^{\vecti}(\cY)_{\CC}$
  that admit a self-twist by~$\chi$ are the Hecke eigenvalue systems attached
  to an automorphic representation~$\automrep= \automrep_\infty\otimes\automrep_f$ of $\G(\adel_F)$
  such that $\JL_D(\automrep)=\AI_{F}^{L}(\automchar)$ for some unitary Hecke
  character~$\automchar$ of~$L$ satisfying all of the following:
  \begin{enumerate}[leftmargin=*,label=\upshape{(\arabic*)}]
    \item for every place~$v$ of~$F$ that ramifies in~$D$, there is a single
      place~$w$ of~$L$ above~$v$, and we have~$\automchar_w^\sigma \neq \automchar_w$;
    \item every real place $v$ of $F$ extends to a complex place $w$ of $L$, and for all
      such places we have
      \begin{itemize}
        \item $i_v=0$ if~$v$ is ramified in~$D$;
        \item $i_v=1$ otherwise;
      \end{itemize}
      and $\automchar_w = \automchar_{\CC}(k,0)$ with $k\in \{\pm 1\}$;
    \item for every complex place $v$ of $F$, extending to places $w$ and $w'$ of $L$,
      we have $\automchar_w=\automchar_{\CC}(k,0)$ and $\automchar_{w'} = \automchar_{\CC}(k',0)$, where 
      $i_v\in \{1,2\}$, $k=-k'\in \{\pm1\}$;
  \end{enumerate}
  each such Hecke eigenvalue system occurring with multiplicity $2^r\dim \automrep_f^{K_f}$,
  where $r$ is the number of real places $v$ of $F$ for which one has $i_v=1$.
\end{proposition}
\begin{proof}
  The result is obtained by specialising Proposition \ref{prop:diffFormsPsi} to
  $\vectlambda=\vectzero$.
\end{proof}

\begin{definition}\label{def:Hshady}
  If $L/F$ is a quadratic extension, with non-trivial automorphism~$\sigma$,
  then an \emph{$\cH^\bullet$-shady character} of $L$ is a unitary Hecke
  character $\automchar$ of $L$ such that~$L$ and $\automchar$ have all of the following properties:
  \begin{enumerate}[leftmargin=*,label={\upshape(\arabic*)}]
    \item\label{item:Hshadyramified} for every place~$v$ of~$F$ that ramifies in~$D$, there is a single
      place~$w$ of~$L$ above~$v$, and we have~$\automchar_w^\sigma \neq \automchar_w$;
    \item\label{item:HshadyR} every real place $v$ of $F$ extends to a complex place $w$ of $L$, and for all
      such places we have
      $\automchar_w = \automchar_{\CC}(k,0)$ with $k\in \{\pm 1\}$;
    \item\label{item:HshadyC} for every complex place $v$ of $F$, extending to places $w$ and $w'$ of $L$,
      we have $\automchar_w=\automchar_{\CC}(k,0)$ and $\automchar_{w'} = \automchar_{\CC}(k',0)$, where 
      $k=-k'\in \{\pm1\}$;
    \item\label{item:Hshadylevel} the unique automorphic representation~$\automrep$ of~$\G(\adel_F)$ such
      that $\JL_D(\automrep) = \AI_F^L(\automchar)$ satisfies~$\automrep_f^{K_f}\neq 0$.
  \end{enumerate}
\end{definition}

The following is Theorem~\ref{thm:IntroProtoMain}
for~$\cH^\bullet$-isospectrality and, together with
Lemma~\ref{lem:RegConstFormalism}, implies Theorem~\ref{thm:IntroRegLL}
parts~\ref{item:IntroRegLLRegConst} and~\ref{item:IntroRegLLProd}.

\begin{theorem}\label{thm:ratlRegPsi}
  Let $c\in C$, and let $i\in \Z_{\geq 0}$.
  Then exactly one of the following statements is true:
  \begin{enumerate}[leftmargin=*,label={\upshape(\roman*)}] 
    \item we have $\#\cV_{\R}+\#\cV_{\CC}\leq i\leq \#\cV_{\R}+ 2\#\cV_{\CC}$ and
      there exist a character~$\chi\in \Cisodual\setminus c^\perp$, with
      corresponding quadratic extension~$L/F$, and an $\cH^\bullet$-shady character of $L$;
    \item for all $b\in C$ the homogeneous components~$H_i(Y_b,\Z)_{\free}$ and
      $H_i(Y_{cb},\Z)_{\free}$ are linked in the sense of Definition
      \ref{def:linked}, and we have
      \[
        \frac{\Reg_i(Y_{b})^2}{\Reg_i(Y_{cb})^2}\in \Q^{\times}.
      \]
  \end{enumerate}
\end{theorem}
\begin{proof}
  This is an immediate consequence of Theorem \ref{thm:regQuosRatl}, Corollary \ref{cor:Matsushima},
  Proposition \ref{prop:harmonPsi}, and Lemma \ref{lem:RegConstFormalism}.
\end{proof}

\begin{remark}
  The range of~$i$ appearing in Theorem~\ref{thm:ratlRegPsi} is sometimes called the
  cuspidal range or the tempered range. Outside that range, the only automorphic
  representations contributing to harmonic forms are non-cuspidal, and moreover
  by Theorem~\ref{thm:ratlRegPsi} we get automatic rationality for the regulator quotient.
\end{remark}

\begin{definition}\label{def:balanced}
  Let~$\mYu$ and~$\mYa$ be fields and~$\mE/\mYu$ a quadratic \'etale algebra.
  A collection~$(q_\tau)_\tau \in \{0,1\}^{\Hom(\mE,\mYa)}$ is \emph{balanced} if
  for every~$\tau\neq \tau' \in \Hom(\mE,\mYa)$ such that~$\tau|_{\mYu} =
  \tau'|_{\mYu}$, we have~$\{q_\tau,q_{\tau'}\} = \{0,1\}$.
\end{definition}

\begin{proposition}\label{prop:algchar}
  Let~$L/F$ be a quadratic extension with~$L$ totally complex, let~$\automchar$ be a Hecke character of~$L$,
  and let~$\automchar^{\alg} = \automchar \|\cdot\|^{-1/2}$.
  Then~$\automchar$ satisfies \ref{item:HshadyR} and \ref{item:HshadyC} of
  Definition~\ref{def:Hshady} if and only if~$\automchar^{\alg}$ is an algebraic Hecke
  character whose type is balanced.
\end{proposition}
\begin{proof}
  This is well-known (see e.g.~\cite[Lemma 4]{MolinPage}).
  The type~$(q_\tau)_\tau$ of~$\automchar^{\alg}$ is related to the parameters~$(k_w)_w$
  of~$\automchar$ by the relations~$q_{\tau}+q_{\bar\tau} = 1$
  and~$q_{\bar\tau}-q_{\tau} = k_\tau$.
\end{proof}

\begin{remark}
  The conditions on~$\automchar^{\alg}$ also imply that for every complex
  embedding~$\tau$ of~$L$ we have~$\{q_{\tau},q_{\bar\tau}\} = \{0,1\}$.
\end{remark}

The following corollary shows that~$\cH^\bullet$-isospectrality is in fact very
frequent.

\begin{cor}
  Let $c\in C$, and let $i\in \Z_{\geq 0}$.
  Then at least one of the following statements is true:
  \begin{enumerate}[leftmargin=*,label={\upshape(\roman*)}]
  \item
    there exist a character~$\chi\in \Cisodual\setminus c^\perp$
    corresponding to a quadratic extension~$L/F$ such that
    $L$ contains a CM subfield
    not contained in~$F$;
  \item for all $b\in C$ we have
    \[
      \frac{\Reg_i(Y_{c})^2}{\Reg_i(Y_{c'})^2}\in \Q^{\times}.
    \]
  \end{enumerate}
\end{cor}

\begin{proof}
  Assume that the second statement does not hold, and
  apply Theorem~\ref{thm:ratlRegPsi} and Proposition~\ref{prop:algchar} to get
  corresponding characters~$\automchar$ and~$\automchar^{\alg}$. By the
  Artin--Weil theorem~(see \cite{Weil} for the statement and, for example,
  \cite[Lemma 2.3.4]{Patrikis} for several proofs), the existence of an algebraic Hecke character
  of non-parallel type forces $L$ to contain a CM subfield, and therefore a
  maximal one,~$M$, say. The type
  of~$\automchar^{\alg}$ descends to~$M$ but not to~$F$, so that~$F$ cannot
  contain~$M$.
\end{proof}

\begin{remark}
  As we consider weaker and weaker notions of isospectrality, there are fewer
  and fewer Hecke characters that can prevent us from proving the respective isospectrality.
  We find it instructive to list the conditions on the shady characters
  incrementally, as follows.

  The shady characters~$\automchar$ appearing in Corollary~\ref{cor:allisospectralPsi}
  ($i$-isospectrality for all~$i$)
  are exactly the characters appearing in Theorem~\ref{thm:repEquivPsi}
  (representation equivalence) that
  satisfy the additional conditions:
  \begin{itemize}[leftmargin=*]
    \item for every real place $v$ of $F$ that
      extends to a complex place $w$ of $L$ we have~$k_w\in \{\pm 1\}$, and
    \item for every complex place $v$ of $F$, extending to places $w$ and $w'$ of $L$,
      we have $k_w\in \{-1,0,1\}$.
  \end{itemize}

  Assuming that~$L$ and~$D$ are ramified at the same set of real places of~$F$,
  the shady characters~$\automchar$ appearing in
  Corollary~\ref{cor:zeroisospectralPsi} ($0$-isospectrality)
  are exactly the characters appearing in Corollary~\ref{cor:allisospectralPsi}
  ($i$-isospectrality for all~$i$) that
  satisfy the additional condition:
  \begin{itemize}[leftmargin=*]
    \item for every complex place $v$ of $F$, extending to places $w$ and $w'$ of $L$,
      we have~$k_w = 0$.
  \end{itemize}

  Assuming that~$L$ is totally complex,
  the shady characters~$\automchar$ appearing in Theorem~\ref{thm:ratlRegPsi}
  (rationality of regulator quotients) are exactly the characters appearing in
  Corollary~\ref{cor:allisospectralPsi} ($i$-isospectrality for all~$i$) that
  satisfy the additional condition:
  \begin{itemize}[leftmargin=*]
    \item for every complex place $v$ of $F$, extending to places $w$ and $w'$ of $L$,
      we have~$k_w \in \{\pm 1\}$ and~$s_w=0$.
  \end{itemize}
\end{remark}

\subsection{Sparsity of non-matching Laplace eigenvalues}\label{sec:KelmerConverse}
In this subsection we show that Vign\'eras pairs of manifolds are always
``almost isospectral''. The main result is Theorem \ref{thm:KelmerConverse},
which in particular implies Theorem \ref{thm:IntroKelmerConverse} from the Introduction.
\begin{notation}
  Let $V$ be a finite-dimensional representation of $K_{\infty}$ and $\Gamma$
  a discrete cocompact subgroup of $G_{\infty}$. Then we set
\[
  \Omega(V,\Gamma) = \Hom_{K_\infty}(V, \rL^2(\Gamma\lquo G_\infty / Z_\infty)).
\]
By $\Delta$ we denote the operator on $\Omega(V,\Gamma)$ induced by
the Casimir operator on the $\rL^2$-space.
If $A$ is a commutative ring, $a\colon A\to \CC$ is a ring homomorphism, and $\Omega$
is an $A_{\CC}$-module, we write $\Omega_{A=a}$ for the subspace of $\Omega$
consisting of all elements $\omega\in\Omega$ such that one has $t\omega=a(t)\omega$ for all $t\in A$.
\end{notation}

\begin{lemma}\label{lem:globalbound}
  Let~$V$ be a finite-dimensional representation of~$K_\infty$.
  Then there exists~$\kappa>0$ such that for every ring homomorphism~$a\colon
  \abstracthecke_1 \to \CC$, we have
  \[
    \sum_{c\in C}\dim \Omega(V,\Gamma_{c})_{\abstracthecke_1=a} \le \kappa.
  \]
\end{lemma}
\begin{proof}
  Let
  $M = \Hom_{K_\infty}(V,\rL^2(\G(F)^+\lquo\G(\adel_{F})/Z_{\infty} K_f))_{\abstracthecke_1=a}$,
  which is a $\abstracthecke$-module that is finite-dimensional over~$\CC$, and satisfies
  \[
    \dim M =
    \sum_{c\in C}\dim \Omega(V, \Gamma_{c})_{\abstracthecke_1=a}.
  \]
  The conclusion of the lemma is trivial if~$M = 0$, so assume otherwise.
  Let~$\heckebig$ be the image of~$\abstracthecke_{\CC}$ in~$\End_{\CC}(M)$, which inherits the
  grading from~$\abstracthecke$. Then~$\heckebig_1$ is a field, and $\heckebig$ is a
  finite \'etale algebra over~$\heckebig_1$ of dimension at most~$\# C$, so
  there are at most~$\# C$ Hecke eigensystems in~$M$.
  We have
  \[
    \dim M
    = \sum_b \dim M_{\abstracthecke=b},
  \]
  where~$b$ ranges over Hecke eigensystems in~$M$.
  Let~$b = (b_\cD)_{\cD}$ be a Hecke eigensystem in~$M$.
  By Strong Multiplicity $1$, Theorem \ref{thm:MultOne}, there exists a unique automorphic
  representation~$\automrep = \automrep_\infty \otimes \automrep_f$
  of~$\G(\adel_F)$ such that $\automrep_f$ corresponds to~$b$, and~$\automrep$
  has multiplicity~$1$ in~$\rL^2(\G(F)^+\lquo\G(\adel_{F})/Z_{\infty})$.
  We have
  \[
    \dim M_{\abstracthecke=b} = \dim \Hom_{K_\infty}(V,\automrep_\infty) \cdot \dim
    \automrep_f^{K_f} \le \kappa_\infty \cdot \kappa_f,
  \]
  where~$\kappa_\infty$ and~$\kappa_f$ are as in Lemmas~\ref{lem:localboundarch}
  and~\ref{lem:localboundfin}.
  Putting everything together, we see that~$\kappa = (\# C) \cdot \kappa_\infty
  \cdot \kappa_f$ satisfies the conclusion of the lemma.
\end{proof}


We will need the following version of Theorem~\ref{thm:IntroProtoMain}.

\begin{lemma}\label{lem:shadyKelmer}
  Let~$V$ be a finite-dimensional representation of~$K_\infty$.
  Let~$\lambda\geq 0$ and let~$a\colon \abstracthecke_1 \to \CC$ be a ring
  homomorphism.
  Then at least one of the following statements holds:
  \begin{enumerate}[leftmargin=*,label={\upshape(\roman*)}]
    \item there exist a character~$\chi\in \Cisodual$,
      with corresponding quadratic extension~$L/F$,
      and an $\rL^2$-shady character of~$L$, such that the automorphic
      representation~$\automrep = \automrep_\infty \otimes \automrep_f$ of~$\G(\adel_F)$ with~$\JL_D(\automrep) =
      \AI_{L/F}(\automchar)$ satisfies
      \begin{itemize}
        \item $\Hom_{K_\infty}(V,\automrep_\infty) \neq 0$,
        \item the Casimir eigenvalue of~$\automrep_\infty$ is~$\lambda$, and
        \item $\abstracthecke_1$ acts on~$\automrep_f$ as~$a$;
      \end{itemize}
    \item all spaces
      $
        \Omega(V, \Gamma_{c})_{\Delta=\lambda, \abstracthecke_1=a}
      $
      for~$c\in C$ have the same dimension.
  \end{enumerate}
\end{lemma}
\begin{proof}
  Apply Proposition \ref{prop:littleExercise}, analogously to e.g. the proof of Theorem
  \ref{thm:HeckeIsosp}, and then apply Proposition \ref{prop:L2psi}.
\end{proof}

We now count the relevant shady characters by reducing to lattice point
counting.

\begin{lemma}\label{lem:counting}
  Let~$\chi\in\Cisodual$, let~$L/F$ be the corresponding quadratic extension, and
  assume that there exists an $\rL^2$-shady character of~$L$.
  Let~$\Lambda$ be the group of unitary Hecke characters~$\automchar$ of~$L$ such that
  \begin{itemize}[leftmargin=*]
    \item for every real place~$v$ of~$F$ that extends to a complex place~$w$
      of~$L$, we have~$s_w = 0$ and~$k_w = 0$;
    \item for every real place~$v$ of~$F$ that extends to two real places~$w,w'$
      of~$L$, we have~$s_w+s_{w'}=0$;
    \item for every complex place~$v$ of~$F$, extending to two complex
      places~$w,w'$ of~$L$, we have~$s_w+s_{w'}=0$ and~$k_w=k_{w'}=0$;
  \end{itemize}
  For a Hecke character~$\automchar$ of~$L$, if
  there exists an automorphic representation~$\automrep =
  \automrep_\infty \otimes \automrep_f$ of~$\G(\adel_F)$ such
  that~$\JL_D(\automrep) = \AI_{L/F}(\automchar)$,
  let~$\lambda(\automchar)$ be the Casimir
  eigenvalue of~$\automrep_\infty$; otherwise let~$\lambda(\automchar)=\infty$.

  Let~$\automchar_0$ be a unitary Hecke character of~$L$.
  Then as~$T\to\infty$, we have
  \[
    \#\{\automchar\in \Lambda \colon \lambda(\automchar_0\automchar)\le T\} = O(T^{r/2}),
  \]
  where~$r=\#\cV_{\R}(D)+\#\cV_{\CC}(D)$.
\end{lemma}
\begin{proof}
  Note that the definition of $\rL^2$-shady characters, specifically Proposition \ref{prop:L2psi},
  property \ref{item:JLpsi}, and equation \eqref{eq:Ciso}, imply 
  that $L$ and $D$ are ramified at the same real places of~$F$.
  By Lemma~\ref{lem:shadyGaloisaction}, a Hecke character~$\automchar$ of~$L$ is
  in~$\Lambda$ if and only if~$\automchar^\sigma\automchar$ has finite order and
  one has~$k_w=0$ for all complex places~$w$ of~$L$.
  Therefore the rank of~$\Lambda$ is $\rk\Z_L^\times-\rk\Z_F^\times=r$.

  Let~$\automchar_0$ be a unitary Hecke character of~$L$.
  By the formulas in Section \ref{subsec:archi}, there exists
  a quadratic form~$Q\colon \Lambda \to \R$ that
  gives~$\Lambda/\Lambda_{\tors}$ the structure of a Euclidean lattice and such
  that for all~$\automchar\in \Lambda$, either~$\lambda(\automchar_0 \automchar) = \infty$ or
  \[
    \lambda(\automchar_0\automchar) = Q(\automchar) + O(Q(\automchar)^{1/2}).
  \]
  In particular, since $\Lambda_{\tors}$ is finite, we have
  \[
    \#\{\automchar\in \Lambda \colon Q(\automchar)\le T\} = O(T^{r/2})
  \]
  in the limit as~$T\to\infty$. This proves the lemma.
\end{proof}

We can finally prove the following result, of which
Theorem~\ref{thm:IntroKelmerConverse} is a special case.

\begin{theorem}\label{thm:KelmerConverse}
  Let~$V$ be a finite-dimensional representation of~$K_\infty$, and let~$c,c'\in C$. Then
  as $T\to \infty$, we have
  \[
    \sum_{\lambda \le T}
    \bigl|
    \dim \Omega(V,\Gamma_{c})_{\Delta=\lambda}
    -
    \dim \Omega(V,\Gamma_{c'})_{\Delta=\lambda}
    \bigr|
    = O(T^{r/2}),
  \]
  where~$r=\#\cV_{\R}(D)+\#\cV_{\CC}(D)$.
\end{theorem}
\begin{proof}
  Let $\kappa$ be as in Lemma~\ref{lem:globalbound}. Then by Lemma~\ref{lem:shadyKelmer} we have
  \begin{eqnarray}
     \lefteqn{\sum_{\lambda \le T}
    \bigl|
    \dim \Omega(V,\Gamma_{c})_{\Delta=\lambda}
    -
    \dim \Omega(V,\Gamma_{c'})_{\Delta=\lambda}
  \bigr| \nonumber}\\
    &=& \sum_{\lambda \le T} \sum_{a\colon \abstracthecke_1 \to \CC}
    \bigl|
    \dim \Omega(V,\Gamma_{c})_{\Delta=\lambda, \abstracthecke_1=a}
    -
    \dim \Omega(V,\Gamma_{c'})_{\Delta=\lambda, \abstracthecke_1=a}
    \bigr| \nonumber\\
    &\le& \sum_{\chi\in \Cisodual}\sum_{\automchar}\kappa, \label{eq:Kelmer}
  \end{eqnarray}
  where for each $\chi \in \Cisodual$, with corresponding quadratic extension $L/F$, the last sum
  runs over $\rL^2$-shady characters~$\automchar$ of~$L$ such that the automorphic
  representation~$\automrep = \automrep_\infty \otimes \automrep_f$ of~$\G(\adel_F)$ with~$\JL_D(\automrep) =
  \AI_{L/F}(\automchar)$ satisfies
  \begin{itemize}[leftmargin=*]
    \item $\Hom_{K_\infty}(V,\automrep_\infty) \neq 0$ and
    \item the Casimir eigenvalue of~$\automrep_\infty$ is at most~$T$.
  \end{itemize}
  Let~$\chi$ and~$L/F$ be as above and let~$\sigma$ be the non-trivial automorphism
  of~$L/F$, and assume that there exists an $\rL^2$-shady character of~$L$.
  By inspection of the tables in Section \ref{subsec:archi}, we see that the 
  condition $\Hom_{K_\infty}(V,\automrep_\infty)
  \neq 0$ implies that the~$k_w$ parameters at complex places~$w$
  of~$L$ of all~$\automchar$ in the sum must belong to a finite set depending
  only on~$V$.
  Let~$\Lambda$ be as in Lemma~\ref{lem:counting}.
  The set of Hecke characters~$\automchar$ of~$L$ that appear in~\eqref{eq:Kelmer} is contained in
  a finite union of cosets of~$\Lambda$, so Lemma~\ref{lem:counting} gives the desired bound.
\end{proof}

\begin{remark}
  Weyl's law implies that for all~$c\in C$ there exists~$a>0$ such that
      as~$T\to\infty$ we have
      \[
        \sum_{\lambda\le T}\dim \Omega(V,\Gamma_c)_{\Delta=\lambda} \sim a
        T^{d/2},
      \]
      where~$d=2\#\cV_{\R}(D)+3\#\cV_{\CC}(D)$. In particular, it follows from
      Theorem \ref{thm:KelmerConverse} that the eigenvalues in the Laplace
      spectra of any two $\Omega(V,\Gamma_{c})$ and $\Omega(V,\Gamma_{c'})$ that 
      do not match up are sparse -- have density $0$ -- among all eigenvalues.
\end{remark}

\begin{remark}\label{rmrk:compareKelmer}
  It is instructive to compare Theorem \ref{thm:KelmerConverse} in the special case
  $r=1$ to the results of \cite{Kelmer}. In this case, \cite[Theorem 1(2)]{Kelmer} shows
  that if $\Gamma_{c}$ and $\Gamma_{c'}$ are torsion-free, and for
  all finite-dimensional representations $V$ of $K_{\infty}$ one has
  \[
    \sum_{\lambda \le T}
    \bigl|
    \dim \Omega(V,\Gamma_{c})_{\Delta=\lambda}
    -
    \dim \Omega(V,\Gamma_{c'})_{\Delta=\lambda}
    \bigr|
    = o(T^{r/2}),
  \]
  then $\Gamma_{c}$ and $\Gamma_{c'}$ are representation equivalent.
  On the other hand, there do exist torsion-free $\Gamma_{c}$ and $\Gamma_{c'}$ 
  arising from the Vign\'eras construction
  that are not representation equivalent -- see \cite{VoightLinowitz} for
  examples and further references.
  This shows that the bound in \cite[Theorem 1(2)]{Kelmer} is sharp.
\end{remark}

%

%
%
\section{\texorpdfstring{Self-twists in characteristic~$p$ and Galois representations}{Self-twists in characteristic p and Galois representations}}\label{sec:galreps}

\begin{assumption}\label{ass:gamma0level}
  We keep Assumption \ref{ass:quaternion}.
In addition, we assume that we have $K_f = K_0(\fN)$ for some non-zero ideal~$\fN$ coprime to~$\delta_D$,
where $K_0(\fN) = \prod_{\fp^i\parallel\fN}K_0(\fp^i)$.
\end{assumption}

In this section we prove Theorem~\ref{thm:IntroProtoMain} for $\Z_p$-isospectrality,
Theorem~\ref{thm:IntroTorsionLL}, Theorem~\ref{thm:IntroRegLL} part~\ref{item:IntroRegLLTriv},
and Theorem~\ref{thm:IntroFiniteSetS} with a specified set $S$,
all conditional on a conjecture on the existence of Galois representations with prescribed
local behaviour attached to mod~$p$ Hecke eigenvalue systems, Conjecture \ref{conj:attachedgalreps}.
The basic idea for proving Theorem \ref{thm:IntroProtoMain} is the same as in Section
\ref{sec:isoconds}, but with the difference that we cannot use automorphic
induction in the mod $p$ setting. Instead, we go via Galois representations,
where the self-twist condition implies that the representation is induced
from a $1$-dimensional Galois character. We then analyse the local properties
of these inductions. Theorem \ref{thm:IntroTorsionLL} and Theorem \ref{thm:IntroRegLL}
part \ref{item:IntroRegLLTriv} are refinements of Theorem \ref{thm:IntroProtoMain}
for $\Z_p$-isospectrality. The proof of Theorem \ref{thm:finiteSetS} is an effective version 
of an argument of Serre, \cite[Proposition 20]{SerreEllCurves}, showing that if
one has mod $p$ Hecke characters with bounded weights for lots of primes $p$, then one has
a suitable characteristic $0$ Hecke character.

In Section \ref{sec:galrepprelim} we recall some basic definitions for Galois representations
and recall a result on conductors of induced representations. In Section \ref{sec:Zpisospectrality}
we state Conjecture \ref{conj:attachedgalreps} and prove Theorem \ref{thm:IntroProtoMain} for $\Z_p$-isospectrality,
Theorem \ref{thm:IntroTorsionLL}, and Theorem \ref{thm:IntroRegLL} part \ref{item:IntroRegLLTriv}.
In Section \ref{sec:localrep} we collect local properties of Galois representations,
which are then used in Section \ref{sec:globalrep} to prove Theorem \ref{thm:IntroFiniteSetS}.

%

\subsection{Basics on Galois representations}\label{sec:galrepprelim}
All representations
of Galois groups will be finite-dimensional and continuous with respect to the discrete
topology on the target.

If $L$ is a perfect field, then $\overline{L}$ denotes a fixed algebraic
closure, $G_L$ denotes the Galois group of $\overline{L}/L$, and, if $p$ is a prime number,
$\eps_{L,p}$ denotes the mod~$p$ cyclotomic character $\eps_{L,p}\colon G_L\to \F_p^\times$.

\begin{defn}
  If $E$ is a local field, we will write $\bE$ for its residue field.

  Let $E$ be a local field.
  We denote by $\Frob_E$ an arbitrary choice of an element of~$G_E$
  that acts as $x\mapsto x^{\#\bE}$ on the residue field of~$\overline{E}$,
  and call such an element a \emph{Frobenius element}.

  Let $p$ be a prime number. We say that a representation~$(V,\galrep)$ of $G_E$ over $\F_p$
  is \emph{finite flat} if there
  exists a finite flat group scheme~$\cG$ over~$\Z_{E}$ such that one has~$V \cong
  \cG(\overline{E})$ as representations of~$G_E$.

  Let $q$ be a power of $p$. We say that a representation~$(V,\galrep)$ of $G_E$
  over~$\F_q$ is \emph{finite flat} if the
  restriction of scalars~$\WRes_{\F_p}^{\F_q}V$ from $\F_q$ to $\F_p$, which is
  a representation of~$G_E$ over~$\F_p$ of dimension~$[\F_q:\F_p]\cdot\dim V$, is finite flat.

  We say that a representation $\galrep\colon G_E\to \GL_2(\F_q)$ is \emph{ordinary}
  if it has a non-zero unramified quotient.
\end{defn}

\begin{defn}
  Let~$L$ be a number field, let $p$ be a prime number, and~$\fp$ a maximal ideal
  of~$\Z_L$.
  We abbreviate $G_{L_{\fp}}$ to $G_{\fp}$.
  We denote by $\Frob_{\fp}$ a choice of Frobenius element in $G_{\fp}$, which, in turn, determines a 
  well-defined conjugacy class in $G_L$.

  Let $q$ be a power of $p$. We say that a representation $(V,\galrep)$ of $G_L$
  over $\F_q$ is \emph{finite flat at~$\fp$} if the restriction~$\galrep|_{G_{\fp}}$ is finite
  flat.

  We say that a representation $\galrep\colon G_L\to \GL_2(\F_q)$ is \emph{ordinary at $\fp$}
  if its restriction to $G_{\fp}$ is ordinary.

  The \emph{prime-to-$p$-conductor} of a representation $\galrep$ of $G_L$ is the
  largest divisor of the Artin conductor that has no prime divisors lying above $p$.
\end{defn}

\begin{lem}\label{lem:Taguchi}
  Let~$p$ be a prime number.
  Let~$\galchar$ be a character of $G_L$, and let~$\fM$ be its prime-to-$p$-conductor.
  Then the representation $\galrep = \Ind_{G_F/G_L}\galchar$ has
  prime-to-$p$-conductor $\delta_{L/F}\norm_{L/F}(\fM)$,
  where $\delta_{L/F}$ denotes the relative discriminant of~$L/F$.
\end{lem}
\begin{proof}
  See \cite{Taguchi}.
\end{proof}

\subsection{Self-twist conditions}\label{sec:Zpisospectrality}

The following is a variant on conjectures formulated by Ash, Calegari--Geraghty, Calegari--Venkatesh,
and others on Galois representations attached to cohomology classes \cite{Ash}, \cite[Conjecture B]{CalegariGeraghty},
\cite[Conjecture 2.2.5]{torsionJL}. We claim no originality, except for
possible mistakes.

\begin{conj}\label{conj:attachedgalreps}
  Under Assumption \ref{ass:gamma0level},
  let~$i\ge 0$ be an integer, let~$\fm$ be a maximal ideal of~$\abstracthecke$, and
  let~$p$ be the residue characteristic of~$\fm$.
  If~$H_i(\cY,\Z)_\fm$ is non-zero, then there exists a semisimple Galois representation
  \[
    \galrep\colon G_F \to \GL_2(\abstracthecke/\fm)
  \]
  with the following properties:
  \begin{enumerate}[leftmargin=*,label=\upshape{(\arabic*)}]
    \item\label{item:condunram} $\galrep$ is unramified outside~$p\delta_D\fN$, and for all~$\fq\nmid
      p\delta_D\fN$ we have
      \[
        \Tr\galrep(\Frob_\fq) \equiv T_\fq \bmod \fm;
      \]
  \item\label{item:conddet} $\det\galrep = \zeta\cdot\eps_{F,p}$, where $\zeta$ is an unramified character;
  \item\label{item:condramif} either $\rho$ is decomposable or for all~$\fq \mid \delta_D$ with $\fq\nmid p$ we have 
      \[
        \galrep|_{G_\fq} \cong \begin{pmatrix}\theta\eps_{F,p} & * \\ 0  & \theta\end{pmatrix}
      \]
      with~$\theta$ unramified;
    \item\label{item:condatp} for all~$\fp\nmid \delta_D\fN$ with~$\fp\mid p$, either
      \begin{itemize}
        \item $\galrep$ is finite flat at~$\fp$, or
        \item $\galrep$ is ordinary at~$\fp$ but not finite flat;
      \end{itemize}
    \item\label{item:condlevel} the prime-to-$p$-conductor of~$\galrep$
      divides~$\delta_D\fN$.
  \end{enumerate}
\end{conj}

Note that we do not make any conjectures about the properties of $\rho$ at primes
that divide both $\delta_D\fN$ and $p$, nor about the precise shape of
$\rho$ at $\fp\mid \delta_D$ when $\rho$ is decomposable.
The conjecture as stated will be sufficient for our purposes.

Versions of the conjecture are known in some special cases, see \cite{Scholze} for
the first results of this nature and \cite{CaraianiNewton} for the most
up-to-date results and overview of the state of knowledge.

\begin{remark}
  In the literature, the most precise local properties of the conjectural Galois
  representations are typically formulated only for irreducible representations,
  and the most subtle properties are those for primes above $p$.
  Following a suggestion of Jack Thorne, we have performed the following sanity check:
  according to \cite{Harder} the boundary contribution to the cohomology of arithmetic
  subgroups of $\GL_2(F)$ yields reducible Galois representations whose reduction
  modulo $\fp$ satisfies condition \ref{item:condatp} of the conjecture.
\end{remark}
%

\begin{lem}\label{lem:induced}
  Let~$p>2$ be a prime number.
  Let~$G$ be a group and~$H$ an index~$2$ subgroup. Let~$\chi\colon G
  \to \{\pm 1\}$ be the corresponding quadratic character.
  Let~$\galrep\colon G \to \GL_2(\Fpbar)$ be a semisimple representation such
  that
  \[
    \galrep\otimes\chi \cong \galrep.
  \]
  Then there exists~$\galchar\colon H \to \Fpbar^\times$ such that
  \[
    \galrep \cong \Ind_{G/H}\galchar.
  \]
\end{lem}
\begin{proof}
  First suppose that $\rho$ is decomposable, so that $\galrep \cong \theta\oplus \theta\chi$
  for some character $\theta$ of $G$. Since~$p\neq 2$, we
  have~$\Fpbar[G/H] \cong \triv\oplus \chi$, and therefore~$\Ind_{G/H}(\theta|_H) \cong
  \theta\otimes_{\Fpbar} \Fpbar[G/H] \cong \galrep$, as claimed.

  Now suppose that $\rho$ is irreducible. Then
  we have
  \[
    \Hom_H(\galrep|_H,\galrep|_H)
    \cong \Hom_G(\galrep, \Ind_{G/H}(\galrep|_H))
    \cong \Hom_G(\galrep, \galrep\otimes_{\Fpbar}\Fpbar[G/H]).
  \]
  Since $\galrep$ is irreducible, we have $\dim\Hom_G(\galrep,\galrep)=1$, and therefore
  \[
    \dim\Hom_H(\galrep|_H,\galrep|_H) = 2.
  \]
  The restriction~$\galrep|_H$ is therefore reducible, so there exists
  a $1$-dimensional subrepresentation~$\galchar$ of~$\galrep|_H$.
  We then have 
  \[
    \Hom_G(\Ind_{G/H}\galchar,\galrep)
    \cong \Hom_H(\galchar,\galrep|_H) \neq 0,
  \]
  so there exists a non-zero homomorphism~$f\colon \Ind_{G/H}\galchar \to \galrep$,
  which must be an isomorphism by irreducibility of~$\galrep$ and dimension
  comparison.
\end{proof}

\begin{lemma}\label{lem:pisenough}
  Let $L$ be a number field, let $p$ be a prime number, let $\fF$ be an ideal
  of~$\Z_L$ supported at the primes above $p$, let $\fM$ be an ideal of $\Z_L$
  coprime to $p$, and let~$\cV$ be a set of real places of $L$. Then every
  Hecke character $\automchar\colon \Cl_{L}(\fF\fM\cV)\to \Fpbar^\times$
  has modulus dividing $p\fM\cV$.
\end{lemma}
\begin{proof}
  The character $\automchar$ is at most tamely ramified at all primes of $\Z_L$
  above $p$, hence the result follows.
\end{proof}

\begin{definition}\label{def:pshady}
  Let~$L/F$ be a quadratic extension.
  A \emph{$\Z_p$-shady character of~$L$} is a Hecke character
  \[
    \automchar \colon \Cl_L(p\delta_D\fN\Z_L\infty) \to \Fpbar^\times
  \]
  such that $\galrep = \Ind_{G_F/G_L}\galchar$ satisfies conditions
  \ref{item:conddet}--\ref{item:condlevel} from Conjecture~\ref{conj:attachedgalreps},
  where $\galchar\colon G_L \to \Fpbar^\times$ is the corresponding Galois
  character.
\end{definition}
We now prove Theorem \ref{thm:IntroProtoMain} for $\Z_p$-isospectrality.

%

\begin{theorem}\label{thm:pshady}
  Let $p>2$ be a prime number. Assume that Conjecture~\ref{conj:attachedgalreps}
  holds for all $i\geq 0$ and all maximal ideals $\fm$ of $\abstracthecke$ with residue
  characteristic $p$.
  Let~$c\in C$.
  Then at least one of the following two statements is true:
  \begin{enumerate}[leftmargin=*,label={\upshape(\roman*)}]
    \item\label{item:pshady} there exist~$\chi\in \Cisodual\setminus c^\perp$ and a $\Z_p$-shady
      character of the quadratic extension determined by~$\chi$;
    \item\label{item:pAllthegoodstuff} there exists~$T\in\abstracthecke[W]_c$ such that for every integer~$i\ge 0$
      and every~$b\in C$, all of the following hold:
    \begin{enumerate}[leftmargin=*,label={\upshape(\alph*)}]
      \item\label{item:6Tisom} $T$ induces an isomorphism
        of~$\abstracthecke_1$-modules
        \[
          T \colon H_i(Y_{b},\Z_{(p)})
          \to H_i(Y_{cb},\Z_{(p)});
        \]
      \item\label{item:6nopregquo} one has
        \[
          \frac{\Reg_i(Y_{cb})^2}{\Reg_i(Y_{b})^2}\in \Z_{(p)}^\times.
        \]
    \end{enumerate}
  \end{enumerate}
\end{theorem}
\begin{proof}
  Suppose that \ref{item:pAllthegoodstuff} does not hold, so that at least one of \ref{item:6Tisom},
  \ref{item:6nopregquo} does not hold. Then by combining Theorem \ref{thm:HeckeIsosp}\ref{item:allpIso}
  and Theorem \ref{thm:regQuosValues}\ref{item:TrivialGlobalRegconst} we deduce that there exists
  $\chi\in \Cisodual\setminus c^{\perp}$ and a Hecke eigensystem $a=(T_{\cD}\mapsto a_{\cD})_{\cD}$
  over $\Fpbar$ in $\Fpbar\otimes H_i(\cY,\Z)$ admitting a self-twist by $\chi$.

  Let $\fm\subset \abstracthecke$ be the kernel of~$a\colon \abstracthecke \to
  \Fpbar$.
  Since the image of $\abstracthecke$ under $a$ is a subring of $\Fpbar$, and hence a field,
  the ideal $\fm$ is maximal, and by assumption we have~$H_i(\cY,\Z)_{\fm}\neq 0$.
  Let $\galrep$ be a Galois representation as in the conclusion of Conjecture
  \ref{conj:attachedgalreps}.
  By property \ref{item:condunram} and Chebotarev's density theorem we have
  $\galrep\otimes \chi \cong \galrep$. By Lemma \ref{lem:induced} we have
  $\galrep \cong \Ind_{G_F/G_L}\galchar$, where $L/F$ is the quadratic extension
  cut out by $\chi$. By combining properties \ref{item:condunram} and \ref{item:condlevel}
  of $\galrep$ with Lemmas \ref{lem:Taguchi} and \ref{lem:pisenough} we deduce that $\galchar$
  has conductor dividing $p\delta_D\fN\Z_L\infty$, so that by class field theory
  it corresponds to a $\Z_p$-shady character of $L$.
\end{proof}

\begin{lemma}\label{lem:iotatriv}
  For all $T\in \abstracthecke_1$, the endomorphisms $T$ and $\iota(T)$
  of $H_i(\cY,\Z)$ are the same.
\end{lemma}
\begin{proof}
  Let~$T = \prod_{j=1}^k T_{\cD_j}\in\abstracthecke_1$ where the~$\cD_j = (\fa_j,\fb_j)$ are such that
  the class of~$\fd = \prod_{j=1}^{k}\fa_j\fb_j$ in~$C$ is trivial.
  We have
  \[
    \iota(T) = \prod_{j=1}^k T_{\cD_j^{-1}} = T_{(\fd,\fd)}^{-1} \cdot T.
  \]
  Since the class of~$\fd$ in~$C$ is trivial, by
  Proposition~\ref{prop:plustransfer} there exists~$\alpha\in F_+^\times$ such
  that we have~$\fd = \alpha\Z_F$. But the corresponding Hecke
  operator~$T_{(\alpha,\alpha)}$ acts trivially on~$H_i(\cY,\Z)$.
  This proves the conclusion of the lemma for~$T$.
  Since Hecke operators of this form generate~$\abstracthecke_1$, this proves
  the lemma.
\end{proof}

If $p$ is a prime number, $\fn$ is a maximal ideal of $\Z_p\otimes \abstracthecke_1$ and
$L/F$ is a quadratic extension, then we say that~$\fn$ \emph{corresponds to} a $\Z_p$-shady
character~$\automchar$ of~$L$ if for every prime ideal~$\fq$ of~$\Z_F$ that splits in~$L$,
we have~$T_{\fq} \bmod\fn =  \automchar(\fQ)+\automchar(\fQ')$, where~$\fq\Z_L =
\fQ\fQ'$.
We have the following local version of Theorem \ref{thm:pshady}, proving Theorem \ref{thm:IntroTorsionLL}
and Theorem \ref{thm:IntroRegLL} part \ref{item:IntroRegLLTriv}.

\begin{theorem}\label{thm:pshadylocal}
  Let $p>2$ be a prime number, and let $\fn$ be a maximal ideal of $\Z_p\otimes \abstracthecke_1$.
  Assume that Conjecture \ref{conj:attachedgalreps} holds for all $i\geq 0$ and all maximal ideals
  $\fm$ of $\abstracthecke$ above $\fn$. Let~$c\in C$.
  Then at least one of the following two statements is true:
  \begin{enumerate}[leftmargin=*,label={\upshape(\roman*)}]
    \item there exists~$\chi\in \Cisodual\setminus c^\perp$, with corresponding quadratic extension $L/F$,
      such that $\fn$ corresponds to a $\Z_p$-shady character $\automchar$
      of $L$;
    \item\label{item:localAllTheGoodStuff} there exists~$T\in\abstracthecke[W]_c$ such that for every integer~$i\ge 0$
      and every~$b\in C$, all of the following hold:
    \begin{enumerate}[leftmargin=*,label={\upshape(\alph*)}]
      \item $T$ induces an isomorphism of~$\abstracthecke_1$-modules
        \[
          T \colon H_i(Y_{b},\Z)_{\fn}
          \to H_i(Y_{cb},\Z)_{\fn};
        \]
      \item we have $\cC_{b,cb}(H_i(\cY,\Z)_{\fn}) \in \Z_p^\times$.
    \end{enumerate}
  \end{enumerate}
\end{theorem}
\begin{proof}
  The proof follows the same pattern as that of Theorem \ref{thm:pshady}.

  Suppose that \ref{item:localAllTheGoodStuff} does not hold. Then in particular $H_i(\cY,\Z)_{\fn}$
  is non-zero, so that $\fn$ is lifted from a maximal ideal of the image of $\abstracthecke_1$ in
  $\End H_i(\cY,\Z_p)$. By Lemma \ref{lem:iotatriv} we have $\iota(\fn)=\fn$.
  By Theorem \ref{thm:HeckeIsosp} \ref{item:localpIso} and Theorem \ref{thm:regQuosValues} \ref{item:TrivialLocalRegconst}
  there is a character $\chi\in \Cisodual\setminus c^\perp$, and a Hecke eigensystem $a$
  with a self-twist by $\chi$.
  Let $\fm\subset \abstracthecke$ be the kernel of~$a\colon \abstracthecke \to
  \Fpbar$, let $L/F$ be the field cut out by $\chi$,
  and $\galrep=\Ind_{G_F/G_L}\galchar$ be the Galois representation attached by
  Conjecture \ref{conj:attachedgalreps}. Then an easy calculation
  (see also the proof of Theorem~\ref{thm:finiteSetS}, specifically
  equation (\ref{eq:induceddecomp})) shows that $\fn$
  corresponds to the Hecke character~$\automchar$ attached to $\galchar$ by class
  field theory.
\end{proof}

\subsection{Local representations}\label{sec:localrep}
For the duration of this subsection,
let~$p>2$ be a prime,
let~$E$ be a $p$-adic field, we write $\Z_E$ for its
ring of integers, $\pi$ for its uniformiser, $e$ for its ramification
index over $\Q_p$, so that $p\in \pi^e\Z_E^\times$, $f$ for the residue field
degree $[\resfield:\F_p]$,
and we take~$\overline\resfield$ to be the residue field of the maximal unramified extension~$E_{\nr}$
of~$E$.
For each~$n\ge 1$, let~$\resfield_n$ be the subfield of~$\overline \resfield$ of degree~$n$
over~$\F_p$ (so that~$\resfield = \resfield_f$).

Let~$I \supset P$ be the inertia group, respectively the wild inertia inside $G_E$,
and let~$I_{\tame} = I/P$.

In this subsection we abbreviate~$\eps_{E,p}$ to $\eps$ and $\Frob_E$ to $\Frob$.

Let~$n\in \Z_{\ge 1}$, let~$q = p^n$ and~$d \mid q-1$. Define a character~$\theta_d\colon
I_{\tame} \to \resfield_n^\times$ by
\[
  \frac{s(\pi^{1/d})}{\pi^{1/d}} = \theta_d(s)\in\mu_d(E_{\nr}) \cong \mu_d(\resfield_n)
  \text{ for all }s\in I_{\tame}.
\]
%
%
A \emph{fundamental character of level~$n$} is a
character~$I_{\tame}\to\overline{\F}_p^\times$ of the form
\[
  \omega_\tau = \tau \circ \theta_{q-1}
\]
for some~$\tau\in\Hom(\resfield_n,\overline{\F}_p)$.
Let $\varphi$ be the absolute Frobenius automorphism $x\mapsto x^p$ of $\Fpbar$.
If one chooses a~$\tau_0\in \Hom(\resfield_n,\F_q)$, then every
character~$\chi\colon I_{\tame} \to \F_q^\times$
is of the form
$$
\chi = \omega_{\tau_0}^a
$$
for some $a\in \Z$.

\begin{remark}\label{rmrk:htau}
  If we have an equality
$$
  \omega_{\tau_0}^a= \prod_{\tau\in \Hom(\resfield_n,\F_q)}\omega_\tau^{h_\tau}
$$
for $h_\tau\in \Z$, then the exponents satisfy
$a\equiv \sum_{i=0}^{n-1}h_{\varphi^i\circ\tau_0}p^i\mod{(q-1)}$.
Since $\omega_{\tau_0}$ has order $q-1$,
we may always choose $a\in \{0,\ldots,q-2\}$
and the $h_\tau\in \{0,\ldots,p-1\}$,
in which case the $h_\tau$ are uniquely determined by~$\chi$, except
that~$(0,\dots,0)$ and~$(p-1,\dots,p-1)$ both represent the trivial character.
\end{remark}

If~$n$, $m\in \Z_{\geq 1}$ are such that $n\mid m$, and~$\tau\in \Hom(\resfield_n,\overline{\F}_p)$, then
\[
  \omega_\tau = \prod_{\tau'\in\Hom(\resfield_m,\overline{\F}_p), \tau'|_{\resfield_n}= \tau}
  \omega_{\tau'}.
\]

By abuse of notation, whenever~$\chi$ is a character of~$I_{\tame} = I/P$, we will also
denote by~$\chi$ its inflation to~$I$ that is trivial on~$P$.

Let~$n\in \Z_{\geq 1}$, set $q=p^n$, and let $\galchar\colon G_E \to \F_q^\times$ be a character.
Then~$\galchar$ is tamely ramified since the order of~$\F_q^\times$ is coprime
to~$p$, and for every choice~$\Frob\in G_E$ of Frobenius element and for
every~$s\in I$ we have
\[
  \galchar(s) = \galchar(\Frob \cdot s\cdot \Frob^{-1}) = \galchar(s)^{\#\resfield}.
\]
Thus, we may always assume that~$n\mid f$. In particular, we can write
\[
  \galchar|_{I} = \prod_{\tau\in\Hom(\resfield,\F_{p^f})} \omega_\tau^{h_\tau}.
\]
In other words, we can decompose~$\galchar|_I$ using only fundamental characters of
level~$f$. This also implies that 
\[
  (\omega_\tau)^{\sigma} = \omega_{\tau\circ\sigma}
\]
for every~$\tau\in\Hom(\resfield,\Fpbar)$ and all $\sigma\in G_E$.


\begin{prop}\label{prop:cyclofondachar}
  The mod~$p$ cyclotomic character satisfies~$\eps|_{I} = \theta_{p-1}^e$.
\end{prop}
\begin{proof}~
  See \cite[Proposition 8]{SerreEllCurves}.
\end{proof}
We will typically use this proposition in the following form:
for every $n\in \Z_{\geq 1}$
we have~$\eps|_I=\prod_\tau \omega_\tau^e$ where the product is over~$\tau\in \Hom(\resfield_n,\F_{p^n})$.

\begin{thm}[Raynaud]\label{thm:raynaud}
  Let~$\galrep \colon G_E \to \GL_n(\F_p)$ be a representation.
  If~$\galrep$ is finite flat, then so is its semisimplification.
  If~$\galrep$ is semisimple, then it is finite flat if and only if each of its
  irreducible factors is finite flat.
  If~$\galrep$ is irreducible, then it is finite flat if and only if there exists
  a character~$\galchar\colon G_E \to \F_q^\times$,
  where~$q = p^n$,
  satisfying
  \[
    \galchar|_{I} = \prod_{\tau\in \Hom(\resfield_n,\F_q)}\omega_\tau^{h_\tau}
  \]
  with~$0\le h_\tau\le e$, such that
  \[
    \galrep \cong \WRes_{\F_p}^{\F_q}\galchar.
  \]
\end{thm}
\begin{proof}~
  [Raynaud, Corollaire 3.3.7 and Th\'eor\`eme 3.4.3]
\end{proof}

\begin{lemma}\label{lem:ord_finite_flat}
  Let $\rho\colon G_E\to \GL_2(\Fpbar)$ be an ordinary representation such that
  $\det \rho = \zeta\cdot\eps$, where $\zeta$ is an unramified character,
  and such that $\rho$ is either induced from a $1$-dimensional
  character or decomposable. Then $\rho$ is finite flat.
\end{lemma}
\begin{proof}
  If $\rho$ is induced from a $1$-dimensional character, then it is semisimple.
  Since it is also ordinary, it has a stable line, and therefore is decomposable
  into a direct sum of $1$-dimensional characters. The proof thus reduces
  to the case that $\rho$ is decomposable.

  Assume that $\rho$ is a direct sum of two $1$-dimensional characters.
  Since $\rho$ is ordinary, one of these characters is unramified. Due to the
  determinant assumption, the other one is of the form $\theta\eps$, where
  $\theta$ is an unramified character. By Theorem~\ref{thm:raynaud} and Proposition
  \ref{prop:cyclofondachar}, $\rho$ is finite flat.
\end{proof}

%
%
%

\begin{lem}\label{lem:redcase}
  Suppose there exists a finite flat character $\theta\colon G_E\to \Fpbar^\times$
  and an unramified character $\phi\colon G_E\to \Fpbar^\times$ such
  that $\theta^2\phi=\eps$. Then we have~$e>1$.
\end{lem}
\begin{proof}
  Suppose that $e=1$. Let $n\in \Z_{\geq 1}$ be such that $\theta$
  takes values in $\F_{p^n}^\times$.
  By Theorem \ref{thm:raynaud} we may write
  $$
    \theta|_{I} = \prod_{\tau\in \Hom(\resfield_n,\Fpbar)}\omega_{\tau}^{h_{\tau}}
  $$
  with $h_{\tau}\in \{0,1\}$. By assumption and by Proposition \ref{prop:cyclofondachar} we have
  $$
    \prod_{\tau\in \Hom(\resfield_n,\Fpbar)}\omega_{\tau}^{2h_{\tau}} = (\theta^2\phi)|_I
      = \eps|_I =  \prod_{\tau\in \Hom(\resfield_n,\Fpbar)}\omega_{\tau}.
  $$
  All exponents of $\omega_{\tau}$ appearing in the equality are in
  $\{0,1,\ldots,p-1\}$, and $1\not\in \{0,p-1\}$, so Remark \ref{rmrk:htau}
  implies that these exponents are uniquely determined. Thus, we have $2h_{\tau}=1$ for all $\tau$,
  which is impossible.
\end{proof}
\subsection{Global representations}\label{sec:globalrep}
If $L$ is a number field and
$\fP$ is a maximal ideal of~$\Z_L$, then we denote by $\F_{\fP}$ the residue field
$\Z_L/\fP$, and by $I_\fP\subset G_{\fP}$ the inertia subgroup at $\fP$.

In this subsection we prove the main result of the section, Theorem \ref{thm:finiteSetS},
which is a precise version of Theorem \ref{thm:IntroFiniteSetS} from the introduction.

\begin{lem}\label{lem:HTinduced}
  Let~$L/F$ be a quadratic extension of number fields, let~$p>2$ be a prime number
  that is unramified in~$L$,
  let~$\galrep\colon G_F\to \GL_2(\Fpbar)$ be a Galois representation of the
  form~$\galrep = \Ind_{G_F/G_L}\galchar$, where~$\galchar\colon G_L \to \Fpbar^\times$
  is a character, and let $\fp\mid p$. Assume that~$\galrep$ is finite flat
  at~$\fp$ and that~$\det\galrep = \zeta\cdot\eps_{F,p}$ for an unramified character $\zeta$.
  Then for every prime~$\fP\mid \fp$ of~$L$ we have
  \[
    \galchar|_{I_\fP} = \prod_{\tau\in\Hom(\F_\fP,\Fpbar)}\omega_\tau^{h_\tau}
  \]
  for some collection~$(h_\tau)_{\tau}\in\{0,1\}^{\Hom(\Z_L/\fp\Z_L,\Fpbar)}=\prod_{\fP|\fp}\{0,1\}^{\Hom(\F_{\fP},\Fpbar)}$
  that is balanced, as in Definition \ref{def:balanced}.
\end{lem}
\begin{proof}
  Let $\sigma$ be the non-trivial element of the Galois group.
  Fixing an extension of the discrete $\fp$-adic valuation on $F$ to $\overline{F}$
  identifies $G_{\fp}$ with a subgroup of $G_F$, and determines a prime $\fP$ of $L$ above $\fp$.
  We have the inclusions
  $$
  \begin{tikzcd}
    I_{\fp} \arrow[r,symbol=\subset] & G_{\fp} \arrow[r,symbol=\subset] & G_F\\
      I_{\fP} \arrow[u,symbol=\subset]\arrow[r,symbol=\subset] & G_{\fP}\arrow[u,symbol=\subset]\arrow[r,symbol=\subset]& G_L\arrow[u,symbol=\subset].
   \end{tikzcd}
   $$

  Of the vertical inclusions, the right hand one has index $2$, so that $G_L$ is a normal
  subgroup of $G_F$, while the left one is an equality, since $L/F$ is unramified at $\fp$.
  These inclusions also induce inclusions $\F_{\fp}\subset \F_{\fP}\cong\F_{\fP^\sigma}\subset \Fpbar$,
  of which the first one is either an equality or an inclusion of index $2$,
  and where the middle isomorphism is induced by $\sigma$.

  Since $\fp$ is unramified in $L/F$, a full set of double coset representatives
  for $I_\fp \lquo G_F / G_L$ is given by $\{1,\sigma\}$, and we have $I_{\fp}=I_{\fP}=I_{\fp}\cap G_L=I_{\fp}\cap G_L^{\sigma}$.
  Mackey's formula therefore implies that we have
  $$
    \galrep|_{I_\fp} \cong \galchar|_{I_\fP} \oplus \galchar^\sigma|_{I_\fP}.
  $$
  Theorem \ref{thm:raynaud} implies that $\galchar|_{I_\fP} $ is of the form
  \[
    \galchar|_{I_\fP} = \prod_{\tau\in\Hom(\F_\fP,\Fpbar)}\omega_\tau^{h_\tau}
  \]
  with $h_{\tau}\in \{0,1\}$ for all $\tau$.
  
  Since we have~$\det\galrep = \zeta\cdot\eps_{F,p}$ for an unramified character $\zeta$, we obtain,
  using Proposition \ref{prop:cyclofondachar},
      \[
        \eps_{F,p}|_{I_\fp}
        = \prod_\tau\omega_\tau
        = \prod_\tau
        \omega_\tau^{h_\tau}\prod_\tau(\omega_{\tau}^{\sigma}){}^{h_\tau}
        = \prod_\tau \omega_\tau^{h_\tau}\omega_{\tau\circ\sigma}^{h_\tau}
        = \prod_\tau \omega_\tau^{h_\tau+h_{\tau\circ\sigma}},
      \]
      where the products run over~$\tau\in \Hom(\F_\fP,\Fpbar)$.
      Since we have~$p-1\ge 2$, Remark~\ref{rmrk:htau} implies that~$h_\tau+h_{\tau\circ\sigma}=1$ for all~$\tau$,
      i.e. the collection $(h_{\tau})_{\tau}$ is balanced.
\end{proof}


\begin{definition}
Given a finite extension $L/F$ of number fields, a non-zero ideal $\fM$ of $\Z_L$, and
a set $\cV$ of real places of $L$, we denote by $U_L^{(F)}(\fM\cV)$ the group
of units of $\Z_L$ that are congruent to an element of $\Z_F$ modulo $\fM$ and
that are positive at all places in $\cV$.
\end{definition}

\begin{lemma}\label{lem:HTunits}
  Let~$L/F$ be a quadratic extension of number fields that is unramified
  at all finite places, $p$ a prime number unramified in~$L$,
  and $\fM$ a non-zero ideal of~$\Z_L$.
  Let~$(h_\tau)_\tau$ be a collection of integers indexed by~$\tau\in\Hom(\Z_L/p,\Fpbar)
  =\bigsqcup_{\fP|p}\Hom(\F_{\fP},\Fpbar)$.
  Let~$\galchar \colon G_L \to \Fpbar^\times$ be a character with prime-to-$p$-conductor
  dividing~$\fM$ such that the character $\det\Ind_{G_F/G_L}\galchar$ of $G_F$ is unramified at
  all primes $\fq$ of $\Z_F$ satisfying $\fq\mid \fM\cap\Z_F$ and $\fq\nmid p$,
  and such that for all prime ideals~$\fP\mid p$ of~$\Z_L$ we have
  \[
    \galchar|_{I_\fP} = \prod_{\tau\in\Hom(\F_\fP,\Fpbar)}\omega_\tau^{h_\tau}.
  \]
  Let~$\tilde{L}$ be the Galois closure of~$L$ over~$\Q$ and let~$\tilde{\fP}\mid p$ be a prime
  of~$\tilde{L}$, inducing a bijection~$\beta\colon \Hom(L,\tilde{L})\cong \Hom(\Z_L/p,\Fpbar)$.
  Then for all~$u\in U_L^{(F)}(\fM\infty)$, we have
  \[
    \prod_{\tau\in\Hom(L,\tilde{L})}\tau(u)^{h_{\beta(\tau)}} \equiv 1 \bmod{\tilde{\fP}}.
  \]
\end{lemma}
\begin{proof}
  Without loss of generality, we may replace $\fM$ by its coprime-to-$p$ part.
  By class field theory, $\galchar$ corresponds to a character
  \[
    \Cl_{L}(\fF\fM\infty) \to \Fpbar^\times,
  \]
  where~$\fF$ is an ideal supported at the primes above~$p$,
  and in fact by Lemma \ref{lem:pisenough} it corresponds to a character
  \[
    \automchar\colon \Cl_{L}(p\fM\infty) \to \Fpbar^\times.
  \]
  Moreover, by \cite{Gallagher} the assumption on $\det\Ind_{G_F/G_L}\galchar$ is equivalent to $\automchar$ factoring
  through $\Cl_{L}(p\fM\infty)/(\Z_F/(\fM\cap \Z_F))^\times$.
  We have an exact sequence
  \begin{equation}\label{eq:clgpsequence}
    1 \to \frac{\prod_{\fP\mid p}\F_\fP^\times}{U_L^{(F)}(\fM\infty)} \to
    \frac{\Cl_{L}(p\fM\infty)}{(\Z_F/(\fM\cap \Z_F))^\times} \to \frac{\Cl_{L}(\fM\infty)}{(\Z_F/(\fM\cap \Z_F))^\times} \to 1.
  \end{equation}
  By \cite[Section 1.5, Proposition 3]{SerreEllCurves}, the restriction of~$\automchar$ to~$\prod_{\fP\mid p}\F_\fP^\times$
  satisfies
  $$
  \automchar|_{\prod_{\fP\mid p}\F_\fP^\times} = \prod_{\fP\mid p}\prod_{\tau\in\Hom(\F_\fP,\Fpbar)}\tau^{-h_\tau}.
  $$
  This forces~$\prod_{\tau\in\Hom(\Z_F/p,\Fpbar)}\tau^{-h_\tau}$ to be trivial
  on the image of~$U_L^{(F)}(\fM\infty)$, which is equivalent to the claimed statement.
\end{proof}

We now prove Theorem \ref{thm:IntroFiniteSetS} from the introduction.
\begin{theorem}\label{thm:finiteSetS}
  Assume Conjecture \ref{conj:attachedgalreps}.
  Let $x$ be the exponent of $\Cl_F$, let~$\chi\in\Cisodual$, let~$L/F$ be the
  corresponding quadratic
  extension, 
  let~$\tilde{L}$ denote the Galois closure of~$L$ over~$\Q$.
  Let~$S$ be the union of the following sets of prime numbers: the primes that
  are ramified in~$L$; the prime numbers that divide~$2\norm_{F/\Q}(\delta_D\fN)$;
  and the prime numbers $p$ for which there exists an ideal $\fM$ of $\Z_L$
  such that $\norm_{L/F}\fM$ divides $\delta_D\fN$, and
  a balanced collection~$(h_\tau)_\tau\in\{0,1\}^{\Hom(L,\tilde{L})}$
  such that all of the following hold:
  \begin{enumerate}[leftmargin=*,label=\upshape{(\alph*)}]
    \item\label{item:divunits} for all units $u\in U_L^{(F)}(\fM\infty)$, the prime~$p$ divides
      \begin{equation}\label{eq:tau_units}
        \norm_{\tilde L/\Q}
        \left(\prod_{\tau\in\Hom(L,\tilde{L})}\tau(u)^{h_\tau} - 1\right);
      \end{equation}
    \item\label{item:dividealgens} for every prime ideal~$\fq$ of~$\Z_F$ dividing~$\delta_D$, let~$\fQ$ be a
      prime of~$\Z_L$ above~$\fq$, let~$k$ be the order of~$\fQ$
      in~$\Cl_L(\fN\Z_L\infty)/(\Z_F/\fN)^\times$, and let~$\alpha\in \Z_F+\fN\Z_L$ be a totally
      positive generator of~$\fQ^k$; then~$p$
      divides
      \begin{equation}\label{eq:tau_alpha}
        \norm_{\tilde L/\Q}\left(
        \Bigl(\prod_{\tau} \tau(\alpha)^{2h_\tau x} - \norm_{L/\Q}(\fQ)^{2kx}\Bigr)
        \Bigl(\prod_{\tau} \tau(\alpha)^{2h_\tau x} - 1\Bigr)
        \right).
      \end{equation}
  \end{enumerate}

  Then:
  \begin{enumerate}[leftmargin=*,label=\upshape{(\arabic*)}]
    \item\label{item:allpshady} every prime number~$p$ such that there exists a $\Z_p$-shady character
      of~$L$ is contained in~$S$;
    \item\label{item:manypshady} the set~$S$ is infinite if and only if there exists an $\cH^\bullet$-shady
      character of~$L$.
  \end{enumerate}
\end{theorem}
\begin{proof}
  Let $\sigma$ be the non-trivial automorphism of $L/F$.
  We first prove \ref{item:allpshady}. Let $p$ be a prime number such that there exists
  a $\Z_p$-shady character $\automchar$. If $p|2\norm_{F/\Q}(\delta_D\fN)$, then it is in $S$.
  Assume otherwise.

  Let $\galchar\colon G_L\to \Fpbar^{\times}$
  be the corresponding Galois character, and let $\galrep=\Ind_{G_F/G_L}\galchar$,
  which, in particular, satisfies the properties of Conjecture \ref{conj:attachedgalreps}.
  By property \ref{item:conddet} we have $\det\galrep = \zeta\cdot\eps_{F,p}$, where $\zeta$
  is an unramified character. Let $\fp$ be any prime ideal of~$\Z_F$
  above $p$. By Mackey's formula the restriction $\galrep|_{G_{\fp}}$ is either also induced
  from a $1$-dimensional character or decomposable.
  By property \ref{item:condatp} and Lemma \ref{lem:ord_finite_flat} the representation
  $\galrep$ is finite flat at $\fp$. Now consider two cases.

  First suppose that $\galrep$ is decomposable. Since~$p\neq 2$, we have
  $\galrep\cong \theta\oplus \theta\chi$ for some character
  $\theta\colon G_F\to \Fpbar^{\times}$ that is finite flat at $\fp$.
  By Lemma \ref{lem:redcase}, applied with $\phi=\chi\zeta^{-1}$,
  the prime $\fp$ is ramified in $F$, so that we have $p\in S$.

  Next suppose that $\galrep$ is irreducible. If $p$ ramifies in $L$, then we have $p\in S$. Assume otherwise.
  By Lemma \ref{lem:HTinduced} we may write, for every prime ideal $\fP$ of $\Z_L$ above $\fp$,
  \[
    \galchar|_{I_\fP} = \prod_{\tau\in\Hom(\F_\fP,\Fpbar)}\omega_{\tau}^{h_{\tau}'},
  \]
  where the collection $(h_{\tau}')_{\tau}\in\{0,1\}^{\Hom(\Z_L/\fp\Z_L,\Fpbar)}=\prod_{\fP|\fp}\{0,1\}^{\Hom(\F_{\fP},\Fpbar)}$
  is balanced. Let $\fM$ be the prime-to-$p$-conductor of $\automchar$. Since $L/F$ is
  unramified at all finite primes, Lemma \ref{lem:Taguchi} and property \ref{item:condlevel}
  of $\rho$ imply that we have $(\norm_{L/F}\fM)|\delta_D\fN$.
  Since $\fp|p$ was arbitrary, we may apply Lemma \ref{lem:HTunits}
  to obtain~$\tilde{\fP}$ and~$\beta$ as
  in Lemma~\ref{lem:HTunits} such that for every $u\in U_L^{(F)}(\fM\infty)$ we have
  \[
    \prod_{\tau\in\Hom(L,\tilde{L})}\tau(u)^{h_{\beta(\tau)}'} \equiv 1
    \bmod{\tilde{\fP}}.
  \]
  Reduction mod~$\tilde{\fP}$ also induces a bijection~$\beta_F\colon \Hom(F,\tilde{L})\cong
  \Hom(\Z_F/p,\Fpbar)$ such that the following square is commutative:
  $$
    \xymatrix{
      \Hom(L,\tilde{L}) \ar[r]^{\text{restriction}}\ar[d]^{\beta}& \Hom(F,\tilde{L})\ar[d]^{\beta_F}\\
      \Hom(\Z_L/p,\Fpbar) \ar[r]^{\text{res.}} & \Hom(\Z_F/p,\Fpbar).
    }
  $$
  Therefore the collection
  $(h_{\tau}=h_{\beta(\tau)}')_{\tau}$ indexed by $\tau\in\Hom(L,\tilde{L})$ is
  balanced, and for all~$u\in U_L^{(F)}(\fM\infty)$ the prime~$p$ divides
  \[ 
    \norm_{\tilde L/\Q}
    \left(\prod_{\tau\in\Hom(L,\tilde{L})}\tau(u)^{h_\tau} - 1\right).
  \]
  This proves divisibility \ref{item:divunits}, and it remains to prove divisibility \ref{item:dividealgens}.

  Let~$\fq$ be a prime ideal of~$\Z_F$ dividing~$\delta_D$.
  Since we have
  \[
    C_{\iso} = \Cl_F(\cV_{\Hamil}(D))/\langle\fa^2, \fp \text{ ramified in }D, \fp^e\|\fN\rangle,
  \]
  the prime~$\fq$ splits in~$L$.
  Let~$\fQ$ be a prime ideal of~$\Z_L$ above~$\fq$.
  By Mackey's formula, we have
  \begin{equation}\label{eq:induceddecomp}
    \galrep|_{G_\fq} \cong \galchar|_{G_\fQ} \oplus \galchar^\sigma|_{G_\fQ}.
  \end{equation}
  Since~$\fq$ does not
  divide~$p$ and~$\galrep$ is irreducible, property~\ref{item:condramif}
  of~$\galrep$ says
  \[
    \galrep|_{G_\fq} \cong \begin{pmatrix}\theta\eps_{F,p} & * \\ 0  & \theta\end{pmatrix}
  \]
  with~$\theta$ unramified. 
  But~$\galrep|_{G_\fq}$ is semisimple, so we have
  \begin{equation}\label{eq:localdecomp}
    \galrep|_{G_\fq} \cong \theta\eps_{F,p} \oplus \theta.
  \end{equation}
  In particular~$\galrep$ is unramified at~$\fq$, and therefore~$\galchar$ is
  unramified at~$\fQ$.
  Since~$\fq | \delta_D$ and~$\fQ | \fq$ were arbitrary, the conductor of~$\automchar$
  divides~$p\fN\Z_L\infty$.
  By property~\ref{item:conddet} of~$\galrep$, we have~$\theta^2=\zeta|_{G_{\fq}}$,
  so that, in particular, we have $\theta^{2x}=1$.
  By equations \eqref{eq:induceddecomp} and \eqref{eq:localdecomp} we get
  \[
    \{\galchar(\Frob_\fQ)^{2x},\galchar^\sigma(\Frob_\fQ)^{2x}\}
    = \{\eps_{L,p}(\Frob_\fQ)^{2x},1\},
  \]
  in other words,
  \begin{equation}\label{eq:pairvalue}
    \{\automchar(\fQ)^{2x},\automchar^\sigma(\fQ)^{2x}\}
    = \{\norm_{L/\Q}(\fQ)^{2x},1\}.
  \end{equation}
  Let~$k$ and~$\alpha$ be as in~\ref{item:dividealgens}.
  By the exact sequence \eqref{eq:clgpsequence}, applied with $\fM=\fN\Z_L$, we have
  \[
    \prod_{\tau\in\Hom(\Z_L/p,\Fpbar)} \tau(\alpha)^{h_\tau'}
    = \automchar(\fQ^k).
  \]
  Combining this with equation \eqref{eq:pairvalue} we deduce that one of the following
  congruences holds:
  \begin{eqnarray*}
    \prod_{\tau\in\Hom(L,\tilde L)} \tau(\alpha)^{2h_\tau x}
    &\equiv&
    \norm_{L/\Q}(\fQ)^{2kx}
    \mod \tilde{\fP} \text{, or}\\
    \prod_{\tau\in\Hom(L,\tilde L)} \tau(\alpha)^{2h_\tau x}
    &\equiv&
    1 \mod \tilde{\fP},
  \end{eqnarray*}
  which proves~\ref{item:dividealgens}.

  Now we prove~\ref{item:manypshady}. 
  For a number field $M$ and a non-zero ideal $\fA$ of $\Z_M$, define
  $$
  \hat{U}_M(\fA) = \{(u_v)_v \in \prod_{v\nmid \infty}\Z_{M,v}^\times : u_v\equiv 1\mod \fA\Z_{M,v}\}\subset \adel_M^\times.
  $$
  If $\automchar$ is a Hecke character of $L$, then the central character of
  $\AI^L_F(\automchar)$ is $\chi\cdot\automchar|_{\adel_F^\times}$ (see e.g.
  \cite{Schmidt} or \cite[Appendix B]{GerardinLabesse}).

  First suppose that there exists an $\cH^\bullet$-shady
  character $\automchar$ of~$L$. By Corollary~\ref{cor:LinowitzVoight} this implies
  that we have $\delta_D=\Z_F$, so that condition \ref{item:dividealgens} on $p$ is empty.
  By Proposition \ref{prop:algchar} the character
  $\automchar^{\alg}=\automchar\|\cdot\|^{-1/2}$ is an algebraic Hecke character
  whose type $(h_{\tau})_{\tau\in \Hom(L,\tilde{L})}$ is balanced. This implies that we
  have $(\automchar^{\alg})^\sigma\neq \automchar^{\alg}$, so by Theorem \ref{thm:AI}
  the automorphic induction $\automrep = \AI^L_F(\automchar^{\alg})$ is infinite-dimensional.
  By Lemma \ref{lem:levelAI} the conductor $\fM$ of $\automchar^{\alg}$ satisfies
  $\norm_{L/F}(\fM)|\fN$. Moreover, the central character $\chi\cdot \automchar^{\alg}|_{\adel_F^\times}$
  of $\automrep$ vanishes on $K_0(\fN)\cap \adel_F^\times=\hat{U}_F(1)$, so that
  $\automchar^{\alg}$ vanishes on $\hat{U}_F(1)\hat{U}_L(\fM)$.
  This implies that for all $u\in U_L^{(F)}(\fM\infty)$
  the infinity-component of $\automchar^{\alg}$ satisfies
  $$
    \prod_{v|\infty}\automchar^{\alg}_v(u) = 1,
  $$
  so that all prime numbers $p$ satisfy divisibility \ref{item:divunits}.
  This shows that the set $S$ is infinite.

  Conversely, suppose that $S$ if infinite. There are only finitely many
  distinct balanced collections $(h_{\tau})_{\tau}\in \{0,1\}^{\Hom(L,\tilde{L})}$,
  and only finitely many ideals $\fM$ of $\Z_L$ such that $\norm_{L/F}\fM$ divides
  $\delta_D\fN$.
  Pick a collection $(h_{\tau})_{\tau}$ and an ideal $\fM$ for which infinitely
  many primes satisfy the divisibilities \ref{item:divunits} and \ref{item:dividealgens}.
  Suppose that there exists a prime ideal $\fq|\delta_D$, and let $\fQ$, $k$
  and~$\alpha$ be as in \ref{item:dividealgens}. Since the collection~$(h_{\tau})_{\tau}$ is balanced,
  we have
  $$
  \Big(\prod_{\tau}(\tau(\alpha)\tau(\alpha)^\sigma)^{h_{\tau}}\Big)^{2x}=\left(\norm_{L/\Q}\alpha \right)^{2x} = \norm_{L/\Q}(\fQ^{2kx}).
  $$
  Since $\alpha$ generates a non-trivial ideal, we have $\prod_{\tau}(\tau(\alpha)^{2h_{\tau}x})\neq 1$,
  and similarly for~$\alpha^\sigma$ in place of $\alpha$, so that the expression in \eqref{eq:tau_alpha}
  is non-zero. This contradicts that it is divisible by infinitely many primes. Hence we have $\delta_D=\Z_F$.

  Next, for every $u\in U_L^{(F)}(\fM\infty)$, infinitely many primes dividing
  the expression~\eqref{eq:tau_units} forces it to be $0$. This implies that there exists an algebraic
  Hecke character~$\automchar^{\alg}$ of $L$ with conductor dividing $\fM$, with type
  $(h_{\tau})_{\tau}$, and such that~$\automchar^{\alg}$ vanishes on $\hat{U}_F(1)$.
  Note that we have $K_0(\fN) = K_1(\fN)\cdot (K_0(\fN)\cap \adel_F^\times)$.
  By Proposition \ref{prop:algchar} and Lemma \ref{lem:levelAI}
  the Hecke character $\automchar = \automchar^{\alg}\|\cdot\|^{1/2}$
  is an $\cH^\bullet$-shady character of~$L$, as required.
\end{proof}

\section{Curiosity cabinet}\label{sec:examples}

We continue to make Assumption \ref{ass:gamma0level}.
Given a number field~$L$ and a non-zero ideal~$\fF$ of~$\Z_L$, we
write~$\heckegroup_{L,\fF}$ for the group of Hecke characters of~$L$ of
conductor dividing~$\fF$.
If $\Gamma$ is a subgroup of $D^\times$, let $P\Gamma$ denote the image
of $\Gamma$ in $D^\times/F^\times$.

In Subsection~\ref{subsec:toolkit}, we first explain that the existence of the
various kinds of shady characters is computable, and then give a criterion to
prove non-isospectrality.
In Subsection~\ref{subsec:collection}, we give a series of examples illustrating
our methods, and we prove Theorems~\ref{thm:IntroSmallIso},
\ref{thm:IntroZeroNotOne} and~\ref{thm:IntroRatlRegQuo}.

\subsection{Toolkit}\label{subsec:toolkit}

By~\cite{MolinPage}, there exist algorithms performing the following tasks:
\begin{itemize}
  \item given~$(L,\fF)$, compute the structure of~$\heckegroup_{L,\fF}$;
  \item given~$\automchar\in \heckegroup_{L,\fF}$, compute the multiplicative
    order of~$\automchar$;
  \item given~$\automchar\in \heckegroup_{L,\fF}$, a non-zero ideal~$\fa$, and a target
    precision, compute~$\automchar(\fa)\in\CC^\times$ to the required precision;
  \item given~$\automchar\in \heckegroup_{L,\fF}$, an infinite place~$v$ of~$L$, and a target
    precision, compute the parameters~$(k_v,s_v)\in \Z\times\CC$ to the required
    precision.
\end{itemize}
This is not quite enough to decide existence of shady characters, because their
definitions involve vanishing conditions on the parameters~$s_v$, so having
access to approximations is not sufficient.
We obtain decision algorithms as follows:
\begin{itemize}
  \item For $\rL^2$-, $\Omega^\bullet$ and~$\Omega^0$-shady characters, by
    Lemma~\ref{lem:shadyGaloisaction} it is sufficient to be able to compute the
    action of a given automorphism of~$L$ on~$\heckegroup_{L,\fF}$.
    This is possible using the description of~$\heckegroup_{L,\fF}$
    in~\cite{MolinPage} as the dual of a group on which the action can clearly
    be computed.
  \item For $\cH^\bullet$-shady characters, by Proposition~\ref{prop:algchar} it
    is sufficient to be able to compute groups of algebraic characters, for
    which an algorithm is described in~\cite{MolinPage}.
  \item For $\Z_p$-shady characters, it is sufficient to be able to compute ray
    class groups, for which there are standard algorithms
    \cite[Chapter~4]{CohenAdvanced}.
\end{itemize}

\begin{definition}
  Let $L/F$ be a quadratic extension, let $i\in \Z_{\geq 0}$, and $\lambda \in \R$.
  Then an \emph{$\Omega^i_{\Delta=\lambda}$-shady character} of $L$ is a
  $(\Omega^{\vecti}_{\Delta=\vectlambda})$-shady character of $L$ for some
  $\vecti=(i_v)_v$, $\vectlambda=(\lambda_v)_v$ satisfying
  $\sum_v i_v=i$ and $\sum_v \lambda_v=\lambda$, with both sums running over
  the infinite places $v$ of $F$.
\end{definition}

The following result will be used to prove non-$i$-isospectrality.
\begin{lemma}\label{lem:negative}
  Assume that $\cV_{\R}(D)=\emptyset$, that $\delta_D=1$, and that $C$ has order $2$,
  let $c\in C$ be the non-trivial element, and $L/F$ the corresponding quadratic extension.
  Let $i\in \Z_{\geq 0}$ and $\lambda\in \R$. Suppose that there exists a unique
  pair $\{\automchar,\automchar^{-1}\}$ of~$\Omega^i_{\Delta=\lambda}$-shady characters of $L$.
  Suppose in addition that the conductor~$\fF$ of~$\automchar$ satisfies that $\fN/\norm_{L/F}(\fF)$
  is the square of an ideal of~$\Z_F$. Then the orbifolds~$Y_{1}$ and~$Y_{c}$ are not $i$-isospectral.
\end{lemma}
\begin{proof}
  It suffices to show that the multiplicity of~$\lambda$ in the~$\Delta$-spectrum on~$\Omega^i(\cY)$ is odd.
  Let $\chi\in \Cdual$ be the unique non-trivial character.

  First let $\automrep=\automrep_{\infty}\otimes \automrep_f$ be 
  an automorphic representation of $\G(\adel_F)$ such that $\automrep\otimes \chi\not\cong \automrep$.
  Then $\automrep$ and $\automrep\otimes \chi$ have the same Casimir eigenvalue. Moreover,
  since the Hecke character $\chi$ is trivial on $\nrd(K_f)$, the dimensions of $\automrep_f^{K_f}$
  and of $(\automrep_f\otimes\chi)^{K_f}$ are equal. Thus, the contribution from $\automrep$
  and from its twist by $\chi$ to the $\lambda$-eigenspace under the isomorphism of Corollary \ref{cor:Matsushima}
  is even-dimensional.

  By assumption, there exists a unique automorphic representation $\automrep$ of $\G(\adel_F)$
  such that $\automrep_f^{K_f}$ is non-trivial,
  $\automrep\otimes\chi\cong\automrep$, and~$\automrep$ contributes
  to~$\Omega^i_{\Delta=\lambda}$. We now apply Proposition \ref{prop:diffFormsPsi}
  in conjunction with Corollary \ref{cor:Matsushima}. The hypothesis on~$\cV_{\R}(D)$
  ensures that, in the notation of Proposition \ref{prop:diffFormsPsi}, we have $r=0$. Moreover,
  by Lemmas~\ref{lem:dimfixedpoints} and~\ref{lem:levelAI}, the assumption on the conductor of $\automchar$
  implies that the contribution from $\automrep$ to multiplicity of the $\lambda$-eigenspace
  is odd-dimensional.
\end{proof}

\subsection{Collection}\label{subsec:collection}


  The following example reproduces~\cite[Example~6.1]{VoightLinowitz} with our method.

\begin{ex}\label{ex:LinowitzVoight}
  Let~$F = \Q(\alpha)$ where~$\alpha^6 - 2\alpha^5 - \alpha^4 + 4\alpha^3 -
  4\alpha^2 + 1 = 0$.
  The field~$F$ is the unique number field of discriminant~$-974528$ and
  signature~$(4,1)$ (\href{http://www.lmfdb.org/NumberField/6.4.974528.1}{LMFDB \texttt{6.4.974528.1}}).
  Let~$D$ be the unique quaternion division algebra ramified at every real place
  and no finite place of~$F$. Let~$\fN = (1)$.
  We have~$C \cong \Ciso \cong C_2$, which therefore has a single non-trivial
  character~$\chi$, corresponding to the quadratic extension~$L =
  F(\sqrt{2\alpha-1})$.

  Let~$C = \{1,c\}$. We have
  \[
    \vol(Y_{1}) = \vol(Y_{c}) = \frac{974528^{3/2}\zeta_F(2)}{2^{12}\pi^{10}} = 2.834032\dots.
  \]

  The extension~$L/F$ is ramified at exactly~$2$ of the~$4$ real places of~$F$.
  By Corollary~\ref{cor:LinowitzVoight}, the groups~$\Gamma_{1}$
  and~$\Gamma_{c}$ are representation equivalent.
\end{ex}

The following example proves Theorem~\ref{thm:IntroSmallIso}.
It is the smallest example we could find of a pair of hyperbolic
$3$-orbifolds that are $i$-isospectral for all~$i$ but not representation
equivalent.

\begin{ex}\label{ex:1notL2}
  Let~$F = \Q(\alpha)$ where~$\alpha^4 - \alpha^3 + \alpha^2 + 4\alpha - 4 = 0$,
  which is also the field~$\Q(\sqrt{-10-14\sqrt{5}})$.
  The field~$F$ is the unique number field of discriminant~$-1375$ and
  signature~$(2,1)$ (\href{http://www.lmfdb.org/NumberField/4.2.1375.1}{LMFDB \texttt{4.2.1375.1}}).
  Let~$D$ be the unique quaternion division algebra ramified at every real place
  and no finite place of~$F$. Let~$\fN = (1)$.
  We have~$C \cong \Ciso \cong C_2$, which therefore has a single non-trivial
  character~$\chi$, corresponding to the quadratic extension~$L = F(\zeta_{10})$.
  Let~$\sigma$ be the non-trivial automorphism of~$L/F$.
  Let~$C = \{1,c\}$. We have
  \[
    \vol(Y_{1}) = \vol(Y_{c}) = \frac{1375^{3/2}\zeta_F(2)}{2^8\pi^6} = 0.2510654\dots.
  \]
  By \cite[Theorem C]{VoightLinowitz}, the groups~$\Gamma_{1}$
  and~$\Gamma_{c}$ are not representation equivalent.
  Take representatives~$(\tau_k)_k$ of~$\Hom(L,\CC)$ modulo complex
  conjugation such that
  \begin{eqnarray*}
    (\tau_k(\zeta_{10}))_k &=& (e^{-\frac{2\pi}{10}i}, e^{\frac{2\pi}{10}i},
    e^{3\frac{2\pi}{10}i}, e^{-3\frac{2\pi}{10}i}), \\
    (\tau_k(\alpha))_k &\approx& (0.809 - 1.607i, 0.809 - 1.607i, 0.845, -1.463).
  \end{eqnarray*}

  We compute the group~$\heckegroup_{L,(1)}$ of unitary Hecke characters. It is
  isomorphic to~$\Z^7 \times \R$. Let~$\automchar_1,\dots,\automchar_7$ denote the
  computed basis of the canonical complement of~$\|\cdot\|^{i\R}$ (see e.g. 
  \cite[Section 3.5]{MolinPage}).
  For each character~$\automchar$ and complex embedding~$\tau$, we display an
  approximation of the pair~$(k,t)\in \Z\times \R$ such that~$\automchar_{\tau}
  = \automchar_\CC(k,it)$.

  \begin{center}
  \begin{tabular}{c|ccccccc}
 & $\automchar_1$ & $\automchar_2$ & $\automchar_3$ & $\automchar_4$ & $\automchar_5$ & $\automchar_6$ & $\automchar_7$
 \\ \hline
    $\tau_1$ & \small $(2,0)$ & \small $(-2,0)$ & \small $(0,0)$ & \small $(-3,-0.898)$ & \small
    $(1,-0.367)$ & \small $(1,-0.367)$ & \small $(-2,1.149)$
 \\
 $\tau_2$ & \small $(-2,0)$ & \small $(2,0)$ & \small $(2,0)$ & \small $(3,0.898)$ & \small $(-2,-1.265)$ & \small $(0,-1.265)$ & \small $(0,-1.149)$
 \\
 $\tau_3$ & \small $(4,0)$ & \small $(-9,0)$ & \small $(-2,-0.611)$ & \small $(-10,0)$ & \small $(5,-0.015)$ & \small $(3,1.647)$ & \small $(-1,-0.525)$
 \\
 $\tau_4$ & \small $(-4,0)$ & \small $(9,0)$ & \small $(2,0.611)$ & \small $(12,0)$ & \small $(-6,1.647)$ & \small $(-4,-0.015)$ & \small $(3,0.525)$
  \end{tabular}
  \end{center}

  We then compute the subgroup~$\heckegroup_{L,(1)}^-$ of unitary Hecke
  characters~$\automchar$ such that~$\automchar^\sigma \automchar$ has finite order.
  It admits a basis~$\automchar_1',\dots,\automchar_4'$, where
  \[
    \automchar_1' = \automchar_1, \quad
    \automchar_2' = \automchar_2, \quad
    \automchar_3' = \automchar_4, \quad
    \automchar_4' = \automchar_3^{-1} \automchar_5 \automchar_6^{-1}
    \automchar_7^{-2}.
  \]

  We display the~$\automchar_j'$ as above.

  \begin{center}
  \begin{tabular}{c|cccc}
 & $\automchar_1'$ & $\automchar_2'$ & $\automchar_3'$ & $\automchar_4'$
 \\ \hline
 $\tau_1$ & $(2,0)$ & $(-2,0)$ & $(-3,-0.898)$ & $(4,-2.298)$
 \\
 $\tau_2$ & $(-2,0)$ & $(2,0)$ & $(3,0.898)$ & $(-4,2.298)$
 \\
 $\tau_3$ & $(4,0)$ & $(-9,0)$ & $(-10,0)$ & $(6,0)$
 \\
 $\tau_4$ & $(-4,0)$ & $(9,0)$ & $(12,0)$ & $(-10,0)$
  \end{tabular}
  \end{center}

  Projecting on the values~$(k_{\tau_1},k_{\tau_3},k_{\tau_4})$, we obtain a
  lattice generated by the columns of a matrix, say $B$, and with LLL-reduced basis with matrix $B'$, where
  \[
    B=
    \begin{pmatrix}
      2 & -2 & -3 & 4 \\
      4 & -9 & -10 & 6 \\
      -4 & 9 & 12 & -10
    \end{pmatrix},
    \quad\quad
    B'=
    \begin{pmatrix}
      2  & 0 & -1 \\
      -1 & 2 & -1 \\
      1  & 2 & 3
    \end{pmatrix}.
  \]
  From this it is clear that there is no~$\Omega^\bullet$-shady character of~$L$, as
  such a character would have to satisfy~$k_{\tau_1} \in \{-1,0,1\}$
  and~$k_{\tau_3},k_{\tau_4} \in \{\pm 1\}$.
  By Corollary~\ref{cor:allisospectralPsi}, the orbifolds~$Y_{1}$
  and~$Y_{c}$ are therefore $i$-isospectral for all~$i\ge 0$.



  The maximal cyclic subgroups of~$P\Gamma_{1}$ have order~$2$, $3$ or~$5$,
  the maximal cyclic subgroups of~$P\Gamma_{c}$ have order~$2$, $3$ or~$10$,
  and we have

  \[
    H_1(P\Gamma_{1},\Z) \cong (\Z/2\Z)^2\quad
    \text{ and }\quad
    H_1(P\Gamma_{c},\Z) \cong (\Z/2\Z)^2.
  \]

  \begin{figure}
    \centering
    \includegraphics[scale=0.15]{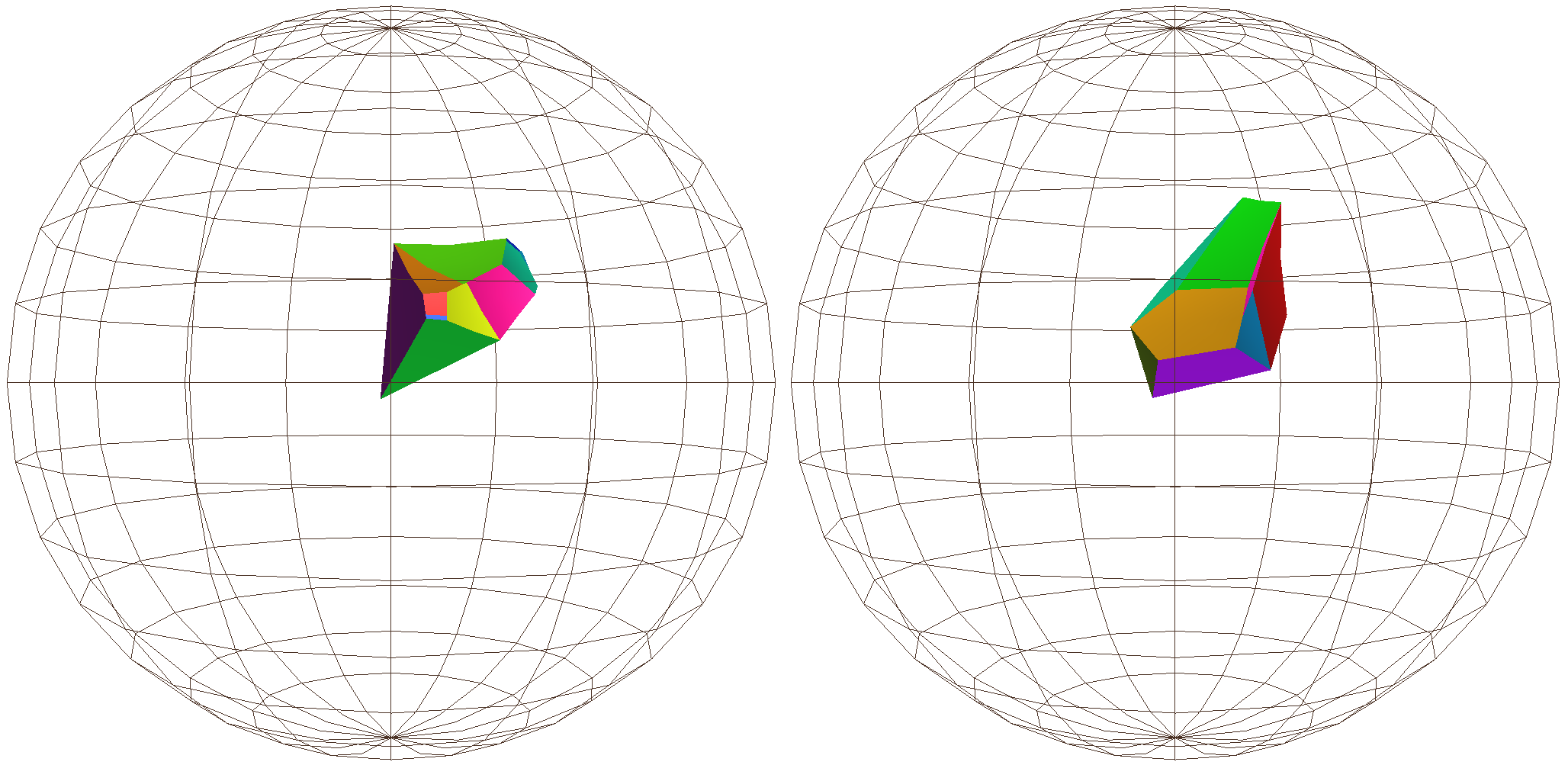}
    \caption{Isospectral and $1$-isospectral but not representation equivalent
    $3$-orbifolds (volume $\approx 0.251$)}
  \end{figure}



\end{ex}

\begin{remark}
The above example is, in particular, an example of a pair
of orbifolds that are $i$-isospectral for all $i$, but not strongly
isospectral. Wolf had asked in~\cite{Wolf} whether such examples exist, 
and their existence was first shown in \cite{LauretEtAlOnWolf}
in all dimensions greater than $4$.
\end{remark}

The following example proves Theorem~\ref{thm:IntroZeroNotOne}.
It is the smallest example we could find of a pair of hyperbolic
$3$-orbifolds that are $0$-isospectral but not not $1$-isospectral.

\begin{ex}\label{ex:0not1}
  Let~$F = \Q(\alpha)$ where~$\alpha^4 - 3\alpha^2 - 2\alpha + 1 = 0$.
  The field~$F$ is the unique number field of discriminant~$-1328$ and
  signature~$(2,1)$ (\href{http://www.lmfdb.org/NumberField/4.2.1328.1}{LMFDB \texttt{4.2.1328.1}}).
  Let~$D$ be the unique quaternion division algebra ramified at every real place
  and no finite place of~$F$. Let~$\fN = (1)$.
  We have~$C \cong \Ciso \cong C_2$, which therefore has a single non-trivial
  character~$\chi$, corresponding to the quadratic extension~$L = F(\zeta_{4})$.
  Let~$\sigma$ be the non-trivial automorphism of~$L/F$.
  Let~$C = \{1,c\}$. We have
  \[
    \vol(Y_{1}) = \vol(Y_{c}) = \frac{1328^{3/2}\zeta_F(2)}{2^8\pi^6} = 0.2461808\dots.
  \]

  Take representatives~$(\tau_k)_k$ of~$\Hom(L,\CC)$ modulo complex
  conjugation such that
  \begin{eqnarray*}
    (\tau_k(\zeta_{4}))_k &=& (i, -i, i, i), \\
    (\tau_k(\alpha))_k &\approx& (-1.138 + 0.485i, -1.138 + 0.485i, 1.940, 0.337).
  \end{eqnarray*}

  We compute the group~$\heckegroup_{L,(1)}$ of unitary Hecke characters. It is
  isomorphic to~$\Z^7 \times \R$. Let~$\automchar_1,\dots,\automchar_7$ denote the
  computed basis of the canonical complement of~$\|\cdot\|^{i\R}$. We display
  the $\automchar_j$ as in the previous example.

  \begin{center}
  \begin{tabular}{c|ccccccc}
 & $\automchar_1$ & $\automchar_2$ & $\automchar_3$ & $\automchar_4$ & $\automchar_5$ & $\automchar_6$ & $\automchar_7$
 \\ \hline
  $\tau_1$ & \small $(4,0)$ & \small $(-23,-0.651)$ & \small $(-31,0.550)$ & \small $(2,-0.453)$ & \small $(-31,0.531)$ & \small $(-6,-1.266)$ & \small $(19,1.114)$
 \\
 $\tau_2$ & \small $(-4, 0)$ & \small $(21, -0.651)$  & \small $(30, 1.266)$ & \small $(-2, -0.453)$  & \small $(30, -1.634)$ & \small $(7, -0.550)$   & \small $(-20, -2.218)$
  \\
  $\tau_3$ & \small $(4, 0)$ & \small $(-22, 0.737)$ & \small $(-31, -0.568)$ & \small $(2, -1.120)$ & \small $(-32, -0.191)$ & \small $(-7, 0.568)$ & \small $(21, -0.191)$
  \\
 $\tau_4$ & \small $(4, 0)$  & \small $(-22, 0.564)$ & \small $(-32, -1.248)$   & \small $(2, 2.026)$   & \small $(-31, 1.295)$  & \small $(-8, 1.248)$   & \small $(20, 1.295)$
  \end{tabular}
  \end{center}

  We then compute the subgroup~$\heckegroup_{L,(1)}^-$ of unitary Hecke
  characters~$\automchar$ such that~$\automchar^\sigma \automchar$ has finite order.
  It admits a basis~$\automchar_1',\dots,\automchar_4'$, where
  \[
    \automchar_1' = \automchar_1,\quad
    \automchar_2' = \automchar_3\automchar_6,\quad
    \automchar_3' = \automchar_5^{-1}\automchar_7,\quad
    \automchar_4' = \automchar_2^{-1}\automchar_4^{-1}\automchar_5\automchar_7.
  \]

  We display the~$\automchar_j'$ as above.

  \begin{center}
  \begin{tabular}{c|cccc}
 & $\automchar_1'$ & $\automchar_2'$ & $\automchar_3'$ & $\automchar_4'$
 \\ \hline
 $\tau_1$ & $(4, 0)$ & $(-37, -0.716)$ & $(50, 0.584)$ & $(9, 2.748)$
 \\
 $\tau_2$ & $(-4, 0)$ & $(37, 0.716)$ & $(-50, -0.584)$ & $(-9, -2.748)$
 \\
 $\tau_3$  & $(4, 0)$ & $(-38, 0)$ & $(53, 0)$ & $(9, 0)$
 \\
 $\tau_4$ & $(4, 0)$ & $(-40, 0)$ & $(51, 0)$ & $(9, 0)$
  \end{tabular}
  \end{center}

  Projecting on the values~$(k_{\tau_1},k_{\tau_3},k_{\tau_4})$, we obtain a
  lattice generated by the columns of a matrix, say $B$, and with
  LLL-reduced basis with matrix $B'$, where

  \[
    B=\begin{pmatrix}
      4 & -37 & 50 & 9 \\
      4 & -38 & 53 & 9 \\
      4 & -40 & 51 & 9
    \end{pmatrix},
    \quad\quad
   B'= \begin{pmatrix}
      1 & -1 & 1 \\
      1 & 2 & 0 \\
      1 & 0 & -2
    \end{pmatrix}.
  \]
  From this it is clear that there is no~$\Omega^0$-shady character of~$L$, as
  such a character would have to satisfy~$k_{\tau_1}=0$
  and~$k_{\tau_3},k_{\tau_4} \in \{\pm 1\}$.
  By Corollary~\ref{cor:zeroisospectralPsi}, the orbifolds~$Y_{1}$
  and~$Y_{c}$ are therefore $0$-isospectral.
  However, there exists an $\Omega^\bullet$-shady character,
  namely~$\automchar_\shady = (\automchar_1')^{-2}\automchar_4'$, and the set of
  $\Omega^\bullet$-shady characters is~$\automchar_\shady^{\pm
  1}\automchar_0^\Z$, where~$\automchar_0 =
  (\automchar_1')^{-9}(\automchar_4')^4$:

  \begin{center}
  \begin{tabular}{c|cc}
 & $\automchar_\shady$ & $\automchar_0$
 \\ \hline
 $\tau_1$ & $(1, 2.748)$ & $(0, 10.994)$
 \\
 $\tau_2$ & $(-1, -2.748)$ & $(0, -10.994)$
 \\
 $\tau_3$ & $(1, 0)$   & $(0, 0)$
 \\
 $\tau_4$ & $(1, 0)$ & $(0, 0)$
  \end{tabular}
  \end{center}

  For all~$n\in \Z$, the
  character~$\automchar_\shady\automchar_0^n$ contributes an
  eigenvalue~$\lambda(n) \approx 4(2.748+10.994 n)^2$ to the Laplace spectrum
  of~$\Omega^1(\cY)$ by Proposition~\ref{prop:diffFormsPsi}. The first few
  corresponding eigenvalues are
  \[
    30.2167\dots,\quad 271.9505\dots,\quad 755.4182\dots,\quad 1480.6196\dots,\quad 2447.5549\dots.
  \]

  The quadratic growth of these eigenvalues is an example
  of the phenomenon described in Theorem~\ref{thm:IntroKelmerConverse}. 
  Only~$\automchar_\shady^{\pm 1}$ contribute to the eigenvalue~$\lambda(0)$, so
  by Lemma~\ref{lem:negative}, the orbifolds~$Y_{1}$ and~$Y_{c}$ are
  not~$1$-isospectral.


  The maximal cyclic subgroups of both~$P\Gamma_{1}$ and~$P\Gamma_{c}$ have order~$2$, $3$ or~$4$,
  and we~have
  \[
     H_1(P\Gamma_{1},\Z) \cong H_1(P\Gamma_{c},\Z) \cong (\Z/2\Z)^2.
  \]

  \begin{figure}
    \centering
    \includegraphics[scale=0.15]{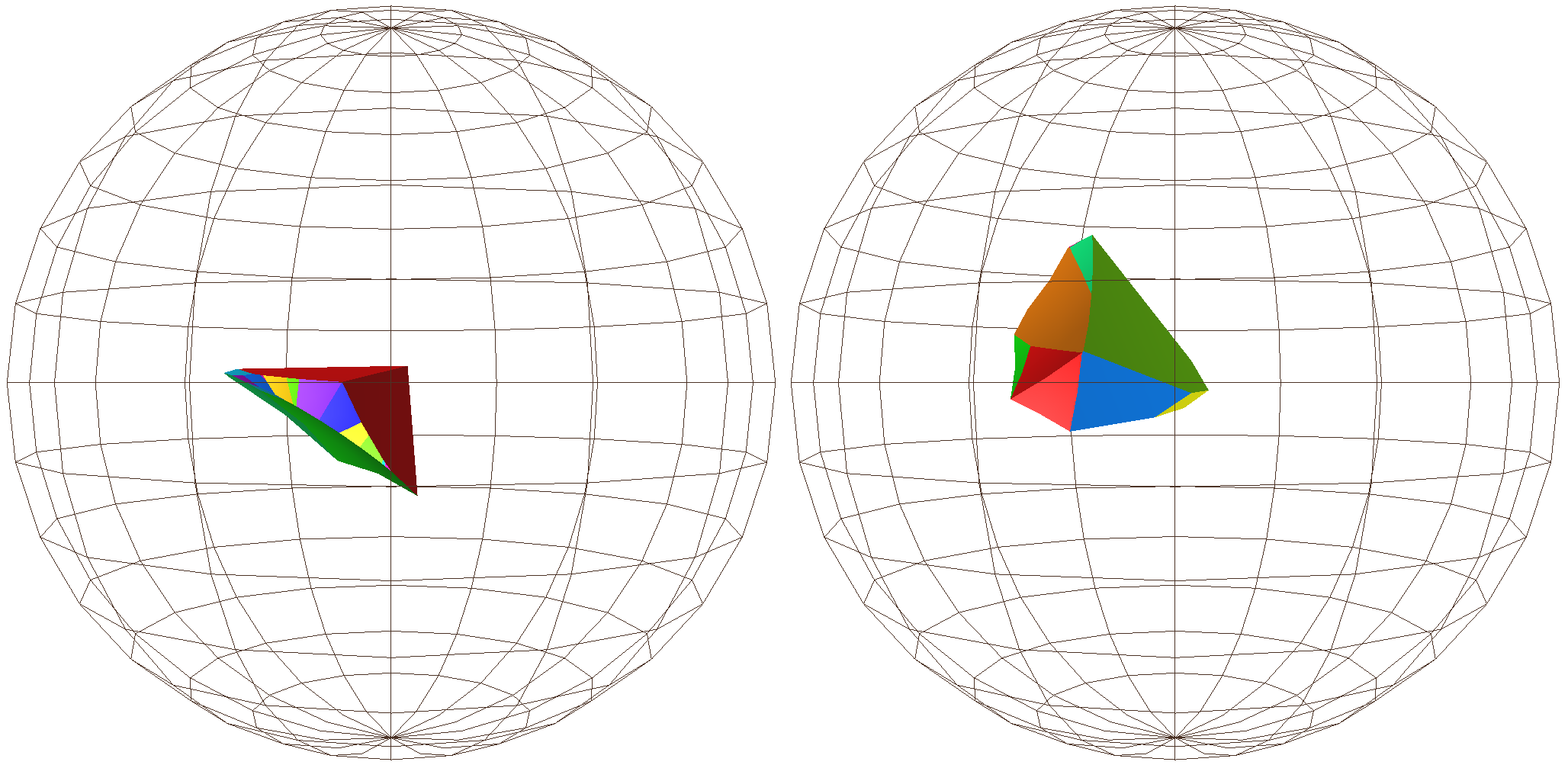}
    \caption{Isospectral but not $1$-isospectral $3$-orbifolds (volume $\approx
    0.246$)}
  \end{figure}
\end{ex}

The following is an example of $0$-isospectral $3$-orbifolds with distinct
Betti numbers (in particular, they are not $1$-isospectral).
Note that Tenie \cite{Tenie} has constructed a pair of isospectral $3$-manifolds
that have non-isomorphic rational cohomology rings.

\begin{ex}\label{ex:0betti}
  Let~$F = \Q(\alpha)$ where~$\alpha^6 - \alpha^5 - 3\alpha^4 + 2\alpha^2 + 4\alpha + 1 = 0$.
  The field~$F$ is the unique number field of discriminant~$-958527$ and
  signature~$(4,1)$ (\href{http://www.lmfdb.org/NumberField/6.4.958527.1}{LMFDB \texttt{6.4.958527.1}}).
  Let~$D$ be the unique quaternion division algebra ramified at every real place
  and no finite place of~$F$. Let~$\fN = (1)$.
  We have~$C \cong \Ciso \cong C_2$, which therefore has a single non-trivial
  character~$\chi$, corresponding to the quadratic extension~$L = F(\zeta_{6})$.
  Let~$\sigma$ be the non-trivial automorphism of~$L/F$.

  Let~$C = \{1,c\}$. We have
  \[
    \vol(Y_{1}) = \vol(Y_{c}) = \frac{958527^{3/2}\zeta_F(2)}{2^{12}\pi^{10}} = 3.397413\dots.
  \]

  Take representatives~$(\tau_k)_k$ of~$\Hom(L,\CC)$ modulo complex
  conjugation such that

  \begin{eqnarray*}
    (\tau_k(\zeta_{6}))_k &=& (e^{-\frac{2\pi}{6}i},e^{\frac{2\pi}{6}i},e^{\frac{2\pi}{6}i},e^{-\frac{2\pi}{6}i},e^{-\frac{2\pi}{6}i},e^{\frac{2\pi}{6}i}), \\
    (\tau_k(\alpha))_k &\approx& (1.959, -0.411 + 0.835i, -0.411 - 0.835i, -0.287, 1.511, -1.361).
  \end{eqnarray*}

  We compute the group~$\heckegroup_{L,(1)}$ of unitary Hecke characters. It is
  isomorphic to~$\Z^{11} \times \R$. Let~$\automchar_1,\dots,\automchar_{11}$ denote the
  computed basis of the canonical complement of~$\|\cdot\|^{i\R}$. We display
  the $\automchar_j$ as above.

  \begin{center}
  \begin{tabular}{c|cccccc}
 & $\automchar_1$ & $\automchar_2$ & $\automchar_3$ & $\automchar_4$ &
    $\automchar_5$ & $\automchar_6$
 \\ \hline
 $\tau_1$  & $(1, 0)$ & $(-1, -0.610)$ & $(-1, 0.610)$ & $(1, 0.242)$ & $(-1, 0.072)$ & $(2, -0.036)$
 \\
 $\tau_2$ & $(-1, 0)$ & $(0, 1.032)$ & $(0, 0.078)$ & $(0, 0.299)$ & $(0, 0.007)$ & $(-1, -0.421)$
 \\
 $\tau_3$ & $(-1, 0)$ & $(0, -0.078)$ & $(0, -1.032)$ & $(-2, 0.299)$ & $(2, 0.007)$ & $(-2, 0.414)$
 \\
 $\tau_4$ & $(1, 0)$ & $(0, -0.395)$ & $(0, 0.395)$ & $(1, -0.847)$ & $(-1, -1.271)$ & $(0, 0.635)$
 \\
 $\tau_5$ & $(1, 0)$ & $(1, -0.292)$ & $(1, 0.292)$ & $(1, -0.643)$ & $(-1, 0.259)$ & $(1, -0.130)$
 \\
 $\tau_6$ & $(-1, 0)$ & $(0, 0.344)$ & $(0, -0.344)$ & $(-1, 0.650)$ & $(1, 0.926)$ & $(0, -0.463)$
  \\ \hline

 & $\automchar_7$ & $\automchar_8$ &
    $\automchar_9$ & $\automchar_{10}$ & $\automchar_{11}$
 \\ \hline
    $\tau_1$ & $(-1, 0.157)$ & $(-1, 0.261)$ & $(1, -0.356)$ & $(-2, -0.428)$ & $(1, 0.261)$ &
 \\
 $\tau_2$ & $(0, -0.762)$ & $(1, 0.282)$ & $(-1, -0.154)$ & $(2, 0.374)$ & $(-1, 1.0138)$ &
 \\
 $\tau_3$ & $(0, 1.068)$ & $(1, 1.014)$ & $(0, 0.381)$ & $(1, -0.160)$ & $(-1, 0.282)$ &
 \\
 $\tau_4$ & $(1, -1.059)$ & $(-2, -0.304)$ & $(2, -0.722)$ & $(-3, 0.549)$ & $(2, -0.304)$ &
 \\
 $\tau_5$ & $(0, -0.192)$ & $(-1, -0.721)$ & $(2, 1.013)$ & $(-3, 0.754)$ & $(1, -0.721)$ &
 \\
 $\tau_6$ & $(0, 0.788)$ & $(0, -0.532)$ & $(0, -0.162)$ & $(1, -1.088)$ & $(0, -0.532)$ &
  \end{tabular}
  \end{center}

  We then compute the subgroup~$\heckegroup_{L,(1)}^-$ of unitary Hecke
  characters~$\automchar$ such that~$\automchar^\sigma \automchar$ has finite order.
  It admits a basis~$\automchar_1',\dots,\automchar_6'$, where

  \[
    \automchar_1' = \automchar_1,
    \automchar_2' = \automchar_2\automchar_3,
    \automchar_3' = \automchar_8^{-1}\automchar_{11},
    \automchar_4' = \automchar_5^{-1}\automchar_9\automchar_{10}^{-1},
    \automchar_5' = \automchar_5\automchar_6^2,
    \automchar_6' = \automchar_4\automchar_5\automchar_7^{-2}.
  \]

  We display the~$\automchar_j'$ as above.

  \begin{center}
  \begin{tabular}{c|cccccc}
 & $\automchar_1'$ & $\automchar_2'$ & $\automchar_3'$ & $\automchar_4'$ & $\automchar_5'$ & $\automchar_6'$
 \\ \hline
 $\tau_1$ & $(1,0)$ & $(-2,0)$ & $(2,0)$ & $(4,0)$ & $(3,0)$ & $(2,0)$
 \\
 $\tau_2$ & $(-1,0)$ & $(0,1.110)$ & $(-2,0.732)$ & $(-3,-0.534)$ & $(-2,-0.835)$ & $(0,1.830)$
 \\
 $\tau_3$ & $(-1,0)$ & $(0,-1.110)$ & $(-2,-0.732)$ & $(-3,0.534)$ &
    $(-2,0.835)$ & $(0,-1.830)$
 \\
 $\tau_4$ & $(1,0)$ & $(0,0)$ & $(4,0)$ & $(6,0)$ & $(-1,0)$ & $(-2,0)$
 \\
 $\tau_5$ & $(1,0)$ & $(2,0)$ & $(2,0)$ & $(6,0)$ & $(1,0)$ & $(0,0)$
 \\
 $\tau_6$ & $(-1,0)$ & $(0,0)$ & $(0,0)$ & $(-2,0)$ & $(1,0)$ & $(0,0)$
  \end{tabular}
  \end{center}

  Projecting on the values~$(k_{\tau_1},k_{\tau_2},k_{\tau_4},k_{\tau_5},k_{\tau_6})$, we obtain a
  lattice generated by the columns of a matrix, say $B$, and with
  LLL-reduced basis with matrix $B'$, where
  \[
   B= \begin{pmatrix}
       1 & -2 &  2 &  4 &  3 & 2\\
      -1 &  0 & -2 & -3 & -2 & 0\\
       1 &  0 &  4 &  6 & -1 & -2\\
       1 &  2 &  2 &  6 &  1 & 0\\
      -1 &  0 &  0 & -2 &  1 & 0
    \end{pmatrix},
    \quad\quad
    B'=
    \begin{pmatrix}
       1 & 0 & -2 &  0 & 0\\
      -1 & 1 & -1 &  1 & 1\\
       1 & 0 &  0 &  0 & 2\\
       1 & 2 &  0 &  0 & 0\\
      -1 & 0 &  0 & -2 & 0
    \end{pmatrix}.
  \]
  From this it is clear that there is no~$\Omega^0$-shady character of~$L$, as
  such a character would have to satisfy~$k_{\tau_2}=0$
  and~$k_{\tau_1},k_{\tau_4},k_{\tau_5},k_{\tau_6} \in \{\pm 1\}$.
  However, there exists an~$\cH^\bullet$-shady character,
  namely~$\automchar_\shady = \automchar_1$. Moreover, $\automchar_\shady^{\pm
  1}$ are the only~$\cH^\bullet$-shady characters of~$L$, so that the $1$st
  Betti numbers of~$Y_{1}$ and~$Y_{c}$ differ by~$1$.
  In fact, we can independently compute the homology and see that we have
  \[
    H_1(P\Gamma_{1},\Z) \cong (\Z/2\Z)^4\quad
    \text{ and }\quad
    H_1(P\Gamma_{c},\Z) \cong (\Z/2\Z)^3 \oplus \Z.
  \]
  The maximal cyclic subgroups of~$P\Gamma_{1}$ have order~$2$ or~$6$,
  and the maximal cyclic subgroups of~$P\Gamma_{c}$ have order~$2$ or~$3$.

  The eigenvalues of the Hecke operators acting on $H_1(\cY,\Q) = \Q$ are
  rational, so one expects them to correspond to an elliptic curve~$E$ over~$F$.
  In this case, since this cohomology class comes from a Hecke character, CM
  theory implies that this elliptic curve exists, and one can find it
  explicitly:
  \[
    E : y^2 + a_3 y
    = x^3  + a_2 x^2
    + a_4x
    + a_6,
  \]
  where
  \begin{eqnarray*}
    a_2 &=& \alpha^5 - \alpha^4 - 3\alpha^3 + \alpha^2 + 3\alpha + 2, \\
    a_3 &=& \alpha^5 - \alpha^4 - 3\alpha^3 + \alpha^2 + \alpha + 2, \\
    a_4 &=& \alpha^3 - 2\alpha, \\
    a_6 &=& 12159\alpha^5 - 15642\alpha^4 - 31998\alpha^3 + 9172\alpha^2 + 21688\alpha + 42423.
  \end{eqnarray*}
  In other words, the fact that the elliptic curve~$y^2+y=x^3$ with $j$-invariant~$0$, after
  base-change to~$F$ and a sextic twist, acquires everywhere good reduction, is
  an algebraic explanation for the fact that~$Y_1$ and~$Y_c$ are not
  $1$-isospectral!

  \begin{figure}
    \centering
    \includegraphics[scale=0.15]{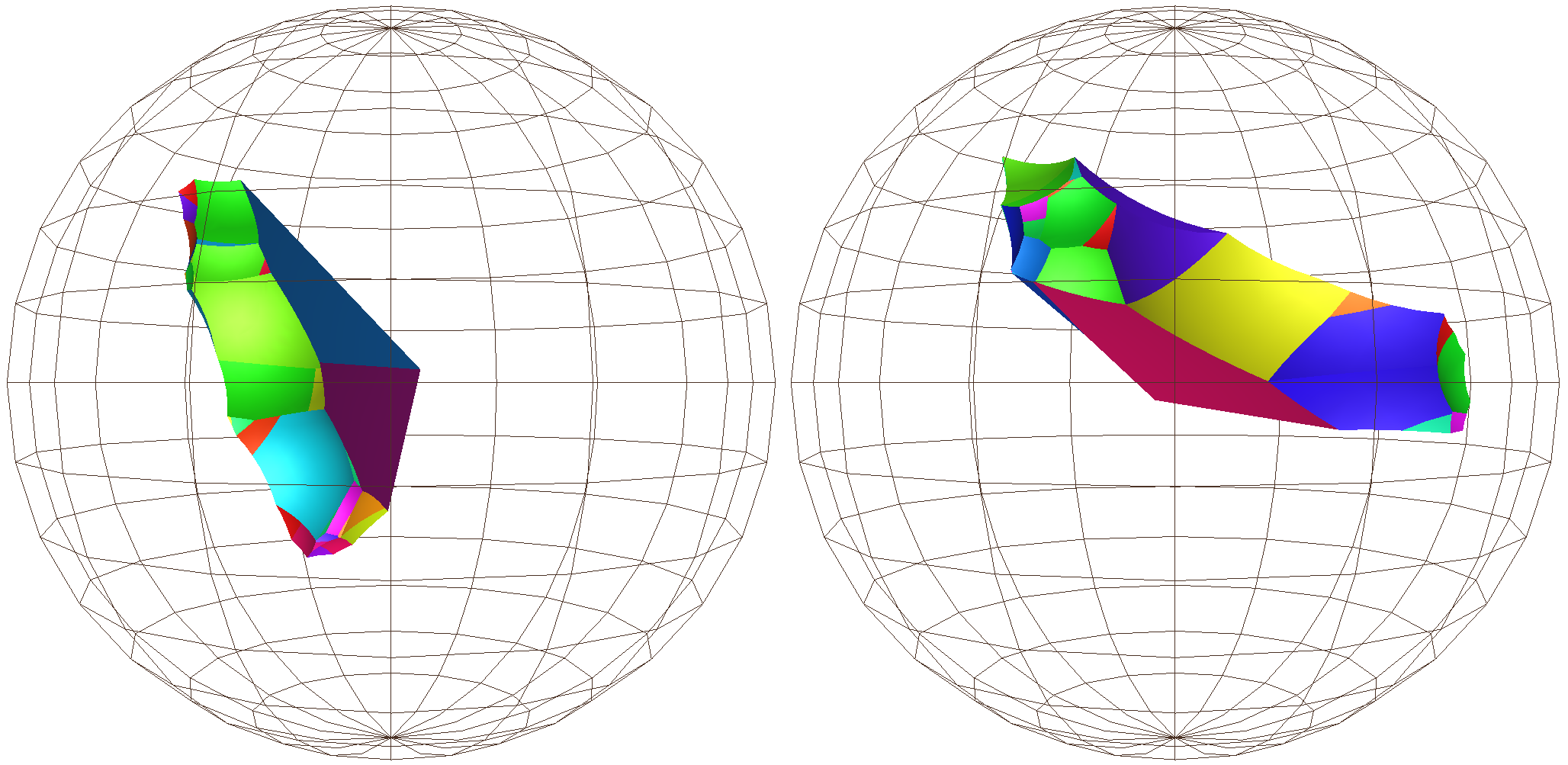}
    \caption{Isospectral $3$-orbifolds with distinct Betti numbers (volume
    $\approx 3.397$)}
  \end{figure}
\end{ex}

The following example proves Theorem~\ref{thm:IntroRatlRegQuo}.
It is an example of non-isospectral $3$-orbifolds forming a Vign\'eras
pair with no $\cH^*$-shady character and non-zero $1$st
Betti numbers; in particular their regulator ratio is rational for a
non-trivial reason.

\begin{ex}\label{ex:Hnot0betti}
  Let~$F = \Q(\alpha)$ where~$\alpha^4 - 2\alpha^3 + 7\alpha^2 - 6\alpha - 3 = 0$.
  The field~$F$ is a number field of discriminant~$-10224$ and
  signature~$(2,1)$ (\href{http://www.lmfdb.org/NumberField/4.2.10224.2}{LMFDB \texttt{4.2.10224.2}}).
  Let~$D$ be the unique quaternion division algebra ramified at every real place
  and no finite place of~$F$. Let~$\fN = (1)$.
  We have~$C \cong \Ciso \cong C_2$, which therefore has a single non-trivial
  character~$\chi$, corresponding to the quadratic extension~$L = F(\zeta_{12})$.
  Let~$\sigma$ be the non-trivial automorphism of~$L/F$.

  Let~$C = \{1,c\}$. We have
  \[
    \vol(Y_{1}) = \vol(Y_{c}) = \frac{10224^{3/2}\zeta_F(2)}{2^{8}\pi^{6}} = 5.902455\dots.
  \]

  Take representatives~$(\tau_k)_k$ of~$\Hom(L,\CC)$ modulo complex
  conjugation such that
  \begin{eqnarray*}
    (\tau_k(\zeta_{12}))_k &=& (e^{-5\frac{2\pi}{12}i},e^{-5\frac{2\pi}{12}i},e^{\frac{2\pi}{12}i},e^{-\frac{2\pi}{12}i}), \\
    (\tau_k(\alpha))_k &\approx& (1.345, -0.345, 0.500 - 2.493i, 0.500 - 2.493i).
  \end{eqnarray*}

  We compute the group~$\heckegroup_{L,(1)}$ of unitary Hecke characters. It is
  isomorphic to~$\Z^7 \times \R$. Let~$\automchar_1,\dots,\automchar_7$ denote the
  computed basis of the canonical complement of~$\|\cdot\|^{i\R}$. We display
  the $\automchar_j$ as above.

  \begin{center}
  \begin{tabular}{c|ccccccc}
 & $\automchar_1$ & $\automchar_2$ & $\automchar_3$ & $\automchar_4$ & $\automchar_5$ & $\automchar_6$ & $\automchar_7$
 \\ \hline
 $\tau_1$ & \small $(1, 0)$ & \small $(3, 0)$  & \small $(-2, 0)$ & \small $(-2, 0.714)$ & \small $(-3, -0.140)$ & \small $(0, -0.140)$ & \small $(0, 0.953)$
 \\
 $\tau_2$ & \small $(1, 0)$ & \small $(3, 0)$ & \small $(-2, 0)$ & \small $(-2, -0.714)$ & \small $(-4, 0.140)$ & \small $(1, 0.140)$ & \small $(0, 0.239)$
 \\
 $\tau_3$ & \small $(-19, 0)$ & \small $(-33, 0)$ & \small $(26, -0.480)$ & \small $(26, 0)$ & \small $(43, -0.215)$ & \small $(-9, 0.215)$ & \small $(-6, -0.356)$
 \\
 $\tau_4$ & \small $(19, 0)$ & \small $(33, 0)$ & \small $(-26, 0.480)$ & \small $(-26, 0)$ & \small $(-42, 0.215)$ & \small $(10, -0.215)$ & \small $(6, -0.836)$
  \end{tabular}
  \end{center}

  We then compute the subgroup~$\heckegroup_{L,(1)}^-$ of unitary Hecke
  characters~$\automchar$ such that~$\automchar^\sigma \automchar$ has finite order.
  It admits a basis~$\automchar_1',\dots,\automchar_4'$, where
  \[
    \automchar_1' = \automchar_1,\quad
    \automchar_2' = \automchar_2,\quad
    \automchar_3' = \automchar_3,\quad
    \automchar_4' = \automchar_5^{-1}\automchar_6.
  \]

  We display the~$\automchar_j'$ as above.

  \begin{center}
  \begin{tabular}{c|cccc}
 & $\automchar_1'$ & $\automchar_2'$ & $\automchar_3'$ & $\automchar_4'$
 \\ \hline
 $\tau_1$ & $(1, 0)$ & $(3, 0)$  & $(-2, 0)$ & $(3, 0)$
 \\
 $\tau_2$ & $(1, 0)$ & $(3, 0)$ & $(-2, 0)$ & $(5, 0)$
 \\
 $\tau_3$ & $(-19, 0)$ & $(-33, 0)$ & $(26, -0.480)$ & $(-52, 0.430)$
 \\
 $\tau_4$ & $(19, 0)$ & $(33, 0)$ & $(-26, 0.480)$ & $(52, -0.430)$
  \end{tabular}
  \end{center}

  Projecting on the values~$(k_{\tau_1},k_{\tau_2},k_{\tau_3})$, we obtain a
  lattice generated by the columns of a matrix, say $B$, and with
  LLL-reduced basis with matrix $B'$, where
  \[
   B= \begin{pmatrix}
      1 & 3 & -2 & 3\\
      1 & 3 & -2 & 5\\
      -19 & -33 & 26 & -52
    \end{pmatrix},
    \quad\quad
   B'= \begin{pmatrix}
      -1 & -2 & 1\\
      1 & -2 & 1\\
      0 & 2 & 5
    \end{pmatrix}.
  \]
  From this it is clear that there is no~$\cH^\bullet$-shady character of~$L$,
  as such a character would have to satisfy~$k_\tau\in\{\pm 1\}$ for
  all~$\tau\in\Hom(L,\CC)$.
  However, there exists an~$\Omega^\bullet$-shady character
  of~$L$, namely~$\automchar_\shady =
  (\automchar_1')^{-1}(\automchar_2')^{-1}\automchar_4'$, and the set of
  $\Omega^\bullet$-shady characters  is~$\automchar_\shady^{\pm 1}
  \automchar_0^\Z$, where~$\automchar_0 =
  (\automchar_1')^{-1}(\automchar_2')^{-1}(\automchar_3')^{-2}$:

  \begin{center}
  \begin{tabular}{c|cc}
 & $\automchar_\shady$ & $\automchar_0$
 \\ \hline
 $\tau_1$ & $(-1, 0)$ & $(0, 0)$
 \\
 $\tau_2$ & $(1, 0)$ & $(0, 0)$
  \\
 $\tau_3$ & $(0, 0.430)$ & $(0, 0.960)$
  \\
 $\tau_4$ & $(0, -0.430)$ & $(0, -0.960)$
  \end{tabular}
  \end{center}

  For all~$n\in \Z$, the
  character~$\automchar_\shady\automchar_0^n$ contributes an
  eigenvalue~$\lambda(n) \approx 1+4(0.430+0.960 n)^2$ to the Laplace spectrum
  of~$\Omega^1(\cY)$ by Proposition~\ref{prop:diffFormsPsi}. The first few
  corresponding eigenvalues are
  \[
    1.741\dots,\quad 2.123\dots,\quad 8.735\dots,\quad 9.883\dots,\quad 23.107\dots,\quad 25.020\dots.
  \]

  Only~$\automchar_\shady^{\pm 1}$ contribute to the eigenvalue~$\lambda(0)$, so
  by Lemma~\ref{lem:negative}, the orbifolds~$Y_{1}$ and~$Y_{c}$ are
  not~$1$-isospectral. Since these characters are also~$\Omega^0$-shady, by the same
  argument these orbifolds are also not $0$-isospectral.


  Since there is no $\cH^\bullet$-shady character of~$L$, by
  Theorem~\ref{thm:ratlRegPsi}, we have
  \[
    \frac{\Reg_1(Y_{1})^2}{\Reg_1(Y_{c})^2}
    \in \Q^\times,
  \]
  and the Betti numbers of~$Y_{1}$ and~$Y_{c}$ are
  equal. In fact, we have
  \[
    H_1(P\Gamma_{1},\Z) \cong (\Z/2\Z)^2\oplus\Z\quad
    \text{ and }\quad
    H_1(P\Gamma_{c},\Z) \cong (\Z/2\Z)^2\oplus\Z,
  \]
  and in particular both $1$st Betti numbers are~$1$, so the rationality of the
  ratio of regulators is a non-trivial statement.
  Moreover, since the orbifolds are not isospectral, the Cheeger--M\"uller theorem does
  not say anything about this rationality.

  The maximal cyclic subgroups of~$P\Gamma_{1}$ have order~$2$, $3$, $4$
  or~$12$, and
  the maximal cyclic subgroups of~$P\Gamma_{c}$ have order~$2$ or~$3$.

  \begin{figure}
    \centering
    \includegraphics[scale=0.15]{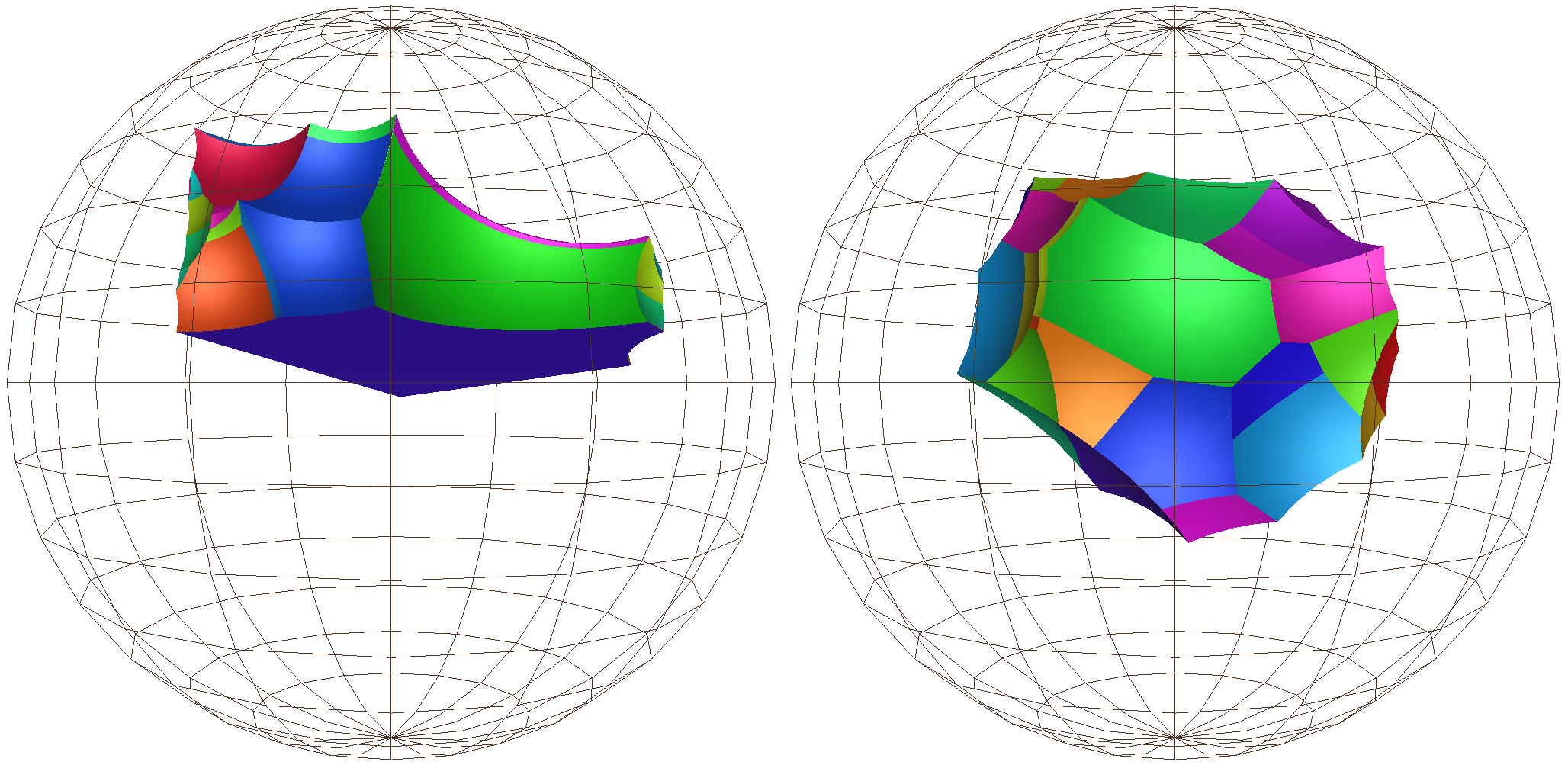}
    \caption{Non-isospectral $3$-orbifolds with Betti number~$1$ and rational
    regulator square quotient (volume $\approx 5.902$)}
  \end{figure}
\end{ex}



\end{document}